\documentclass{amsart}
\usepackage[mathscr]{eucal}% So-called Euler fonts, allows \mathscr
\usepackage{amssymb}% allows special arrows like >-> e.g.
\usepackage{pifont}
\usepackage{graphicx}
\usepackage{psfrag} % Necessary to put latex into .eps text fragments
\usepackage[usenames,dvipsnames]{color}
\usepackage[normalem]{ulem}% allows \sout to cross-out text.
\usepackage{amsthm}% allows theorem* , e.g.
\usepackage{bbold}% allows \unit \mathbb{1}
\usepackage{bbm}% same as above but slightly less nice; use if \mathbb should remain as `normal'
\usepackage{enumerate}% allows easy change of label
\usepackage{array}% allows array
\usepackage{amsmath}% allows pmatrix

\usepackage{hyperref}
\hypersetup{colorlinks=true,linkcolor={Brown},citecolor={Brown},urlcolor={Brown}}% replace BrickRed by Brown?

\usepackage{cleveref}% allows references via \cref and \Cref which automatically repeat the name (Thm, Prop, etc) of the ref.

\numberwithin{equation}{section}% makes equat numb contain the section % superseded by the above
\makeatletter
\@namedef{subjclassname@2020}{%
  \textup{2020} Mathematics Subject Classification}
\makeatother

\makeindex

\usepackage[all]{xy}
\CompileMatrices

\newdir{ >}{{}*!/-10pt/\dir{>}}

\swapnumbers %pour que le numéro apparaisse devant le théorème

\newtheorem{Thm}[equation]{Theorem}

\newtheorem{Prop}[equation]{Proposition}
\newtheorem{Lem}[equation]{Lemma}
\newtheorem{Cor}[equation]{Corollary}

\theoremstyle{remark}
\newtheorem{Rem}[equation]{Remark}
\newtheorem{Def}[equation]{Definition}
\newtheorem{Ter}[equation]{Terminology}
\newtheorem{Not}[equation]{Notation}
\newtheorem{Exa}[equation]{Example}
\newtheorem{Exas}[equation]{Examples}
\newtheorem{Cons}[equation]{Construction}

\newtheorem{Hyp}[equation]{Hypothesis}
\newtheorem{Rec}[equation]{Recollection}

\theoremstyle{definition}

\newtheorem*{Ack*}{Acknowledgements}
\newtheorem*{Org*}{Organization}

\newcommand{\nc}{\newcommand}
\nc{\dmo}{\DeclareMathOperator}

\dmo{\Ab}{Ab}
\dmo{\add}{add} % the additive hull (= closure under sums and direct summands)
\dmo{\Aut}{Aut}
\dmo{\bickMack}{\biMack^{\mathsf{ic}}_\kk} % the bicategory of i.c. Mackey 2-functors
\dmo{\biMack}{\mathsf{Mack}} % the bicategory of Mackey 2-functors
\dmo{\Ch}{Ch}% ground notation for chain complexes
\dmo{\CoInd}{CoInd}
\dmo{\Der}{D}% ground notation for derived categories
\dmo{\DER}{\mathsf{DER}}
\dmo{\Db}{D^b}% ground notation for bounded derived categories
\dmo{\End}{End}
\dmo{\Fun}{\mathrm{Fun}} % the 2-category of 2-functors from enriched category theory
\dmo{\Free}{\mathrm{free}}
\dmo{\Cofree}{\mathrm{cofree}}
\dmo{\forget}{\mathrm{forget}}
\dmo{\Hom}{Hom}
\dmo{\Ho}{Ho}
\dmo{\img}{im}
\dmo{\incl}{incl}
\dmo{\Ind}{Ind}
\dmo{\inj}{in} % notation for canonical injections
\dmo{\Inj}{Inj} % injective modules/objects
\dmo{\Ker}{Ker}
\dmo{\Kadd}{K_0^{add}}
\dmo{\Kex}{K_0^{exa}}
\dmo{\Kexhigher}{K^{exa}_*}
\dmo{\Ktr}{K_0^{tri}}
\dmo{\Mackey}{Mack} % the category of ordinary Mackey functors
\dmo{\CohMackey}{CohMack} % ordinary cohomological Mackey functors
\dmo{\Map}{Map}%
\dmo{\Mod}{Mod}% sheaves of modules
\dmo{\Comod}{Comod}
\dmo{\Nat}{Nat}
\dmo{\Qcoh}{Qcoh}% quasicoherent sheaves over a scheme
\dmo{\coh}{coh}% coherent sheaves over a scheme
\dmo{\fgmod}{mod}
\dmo{\fgfree}{free}
\dmo{\stmod}{stmod}
\dmo{\StMod}{StMod}
\dmo{\latt}{latt}
\dmo{\Mor}{Mor}%
\dmo{\Obj}{Obj}
\dmo{\Proj}{Proj} % projective modules/objects
\dmo{\fgproj}{proj} % fg projective modules/objects
\dmo{\pr}{pr}
\dmo{\2Fun}{\mathsf{2Fun}}
\dmo{\PsFunJJ}{\PsFun_{\JJ_!}^{\JJ^\prime\textrm{\!-}\mathsf{oplax}}}
\dmo{\PsFunJop}{\PsFun_{{{\JJ}_{{}_{*}}}}}
\dmo{\PsFunJ}{\PsFun_{\JJ_!}}
\dmo{\PsFunoplax}{\PsFun^{\mathsf{oplax}}}
\dmo{\PsFun}{\mathsf{PsFun}} % the bicategory of pseudofunctors
\dmo{\PsNat}{\mathsf{PsNat}}
\dmo{\PsMon}{\mathsf{PsMon}} % 2-cat of pseudomonoids
\dmo{\BrPsMon}{\mathsf{BrPsMon}}
\dmo{\SymPsMon}{\mathsf{SymPsMon}}
\dmo{\Rad}{Rad}
\dmo{\Res}{Res}
\dmo{\SH}{SH}% ground name for cat of spectra
\dmo{\Sh}{Sh}
\dmo{\Spanname}{{\sf Span}}
\dmo{\Spec}{Spec}
\dmo{\Stab}{Stab}% stable category of non-fin. gen. mod.
\dmo{\twoFun}{2\mathsf{Fun}}
\dmo{\tr}{tr}

\nc{\ababs}{{\sl ab absurdo}}
\nc{\Add}{\mathsf{Add}}
\nc{\ADD}{\mathsf{ADD}}
\nc{\ADDic}{\mathsf{ADD}^{\ic}}
\nc{\ADDer}{\mathsf{ADDer}}
\nc{\ADDick}{\mathsf{ADD}_\kk{}^{\!\!(\ic)}} % one or the other
\nc{\adhoc}{{\sl ad hoc}}
\nc{\adjto}{\rightleftarrows}
\nc{\adj}{\dashv\,}
\nc{\afortiori}{{\sl a fortiori}}
\nc{\aka}{{a.\,k.\,a.}\ }
\nc{\all}{\mathsf{all}}% all 1-cells.
\nc{\apriori}{{\sl a priori}}
\nc{\ass}{\mathrm{ass}} % associator
\nc{\bbA}{\mathbb{A}}
\nc{\bbB}{\mathbb{B}}
\nc{\bbC}{\mathbb{C}}
\nc{\bbD}{\mathbb{D}}
\nc{\bbF}{\mathbb{F}}
\nc{\bbI}{\mathbb{I}}
\nc{\bbM}{\mathbb{M}}
\nc{\bbN}{\mathbb{N}}
\nc{\bbP}{\mathbb{P}}
\nc{\bbQ}{\mathbb{Q}}
\nc{\bbR}{\mathbb{R}}
\nc{\bbZ}{\mathbb{Z}}
\nc{\bs}{\backslash}
\nc{\BurnG}{\cat{A}(G)}
\nc{\cat}[1]{\mathcal{#1}}
\nc{\Cat}{\mathsf{Cat}}
\nc{\CAT}{\mathsf{CAT}}
\nc{\cf}{{\sl cf.}\ }
\nc{\Cf}{{\sl Cf.}\ }
\nc{\colim}{\mathop{\mathrm{colim}}}
\nc{\costar}{**}% for the (-)^* embedding with \beta^{-1} on 2-cells
\nc{\co}{{\mathrm{co}}}
\nc{\DD}{\cat{D}}% a derivator
\nc{\Displ}{\displaystyle}
\nc{\diag}[1]{\overline{#1}} % (essential) diagonal-part operation
\nc{\offdiag}[1]{{#1}^\dagger} % off-diagonal-part operation
\nc{\doublequot}[3]{#1\backslash #2/#3}% double-cosets
\nc{\Ecell}{\rotatebox[origin=c]{90}{$\Downarrow$}} % treated as others too! But not to be used in-line!!! [I]
\nc{\eg}{{\sl e.g.}\ } % made similar to the previous two
\nc{\Eg}{{\sl E.g.}\ } % we needed a capital one!
\nc{\eps}{\varepsilon}
\nc{\equalby}[1]{\overset{\textrm{#1}}{=}}
\nc{\exact}{\mathsf{ex}}
\nc{\faithful}{\mathsf{faithful}}% faithful
\nc{\faith}{\mathsf{faithf}}
\nc{\final}{\textrm{\scriptsize{\ding{93}}}} % modified asterisk
\nc{\Funadd}{\Fun_{\amalg}}% additive functors, in the sense of Mack 1
\nc{\Funplus}{\Fun_{+}}% additive functors. Can be tweaked...
\nc{\fun}{\mathrm{fun}} % "funtorator" of pseudofunctors (or whatever it's called)
\nc{\GG}{\mathbb{G}}% the 2-category of finite groupoids `of interest'
\nc{\gpdG}{{\groupoidf_{\!\smallslash\!G}}} % the "correct" 2-category of groupoids for Mackey functors for G, i.e. the comma category of faithful functors to G
\nc{\gpdGzero}{{\groupoidf_{\!\smallslash\!G_0}}\!} % variation with G_0
\nc{\gpdfover}[1]{\groupoidf_{\!\smallslash\!#1}}
\nc{\gpd}{\groupoid}%
\nc{\gps}{\mathsf{groups}} % category of finite groups
\nc{\groconn}{\groupoid_{\mathsf{conn}}}% connected finite groupoids (auxiliary)
\nc{\groupoidf}{\groupoid{}^{\smallfaithful}}% finite groupoids with faithful morphisms
\nc{\gpdf}{\groupoidf} % short version
\nc{\groupoid}{\mathsf{gpd}}% finite groupoids
\nc{\group}{\mathsf{group}} % category of finite groups, variant...
\nc{\Gsets}{G\sset}
\nc{\HGfK}{\doublequot{H}{G}{f(K)}}%
\nc{\HGK}{\doublequot HGK}% most used
\nc{\Homcat}[1]{\Hom_{\cat #1}}
\nc{\hooklongleftarrow}{\longleftarrow\joinrel\rhook}
\nc{\hooklongrightarrow}{\lhook\joinrel\longrightarrow}
\nc{\hook}{\hookrightarrow}
\nc{\Hsets}{H\mathsf{-sets}}
\nc{\ic}{\mathsf{ic}}
\nc{\ICAdd}{\Add_{\ic}}%
\nc{\ICADD}{\ADD_{\ic}}%
\nc{\Idcat}[1]{\Id_{\cat{#1}}}
\nc{\id}{\mathrm{id}}
\nc{\Id}{\mathrm{Id}}
\nc{\ie}{{\sl i.e.}\ }
\nc{\into}{\mathop{\rightarrowtail}}
\nc{\inv}{^{-1}}
\nc{\Iout}[1]{\Ivo{\sout{#1}}}
\nc{\isocell}[1]{\undersett{ #1}{\overset{\sim}{\Ecell}}} % to be used ONLY for 2-cells in xypic diagrams [I]
\nc{\backisocell}[1]{\undersett{ #1}{\overset{\sim}{\Wcell}}} % to be used ONLY for 2-cells in xypic diagrams [I]
\nc{\Isocell}[1]{\undersett{ #1}{\overset{\sim}{\Longrightarrow}}}% ONLY inline with target and source [I]
\nc{\isoEcell}{\overset{\sim}{\Rightarrow}} % to be used ONLY in-line, with target and source [I]
\nc{\isotoo}{\stackrel{\sim}\longrightarrow}
\nc{\isoto}{\buildrel \sim\over\to}
\nc{\Ivo}[1]{{\color{OliveGreen}#1}}
\nc{\JJ}{\mathbb{J}}% old \II % the class of `faithful' guys, we can try other things.
\nc{\kk}{\Bbbk}
\nc{\KK}{\mathrm{KK}}
\nc{\leps}{{}^{\ell}\eps}
\nc{\leta}{{}^{\ell}\eta}
\nc{\loccit}{{\sl loc.\ cit.}}
\nc{\lotoo}[1]{\overset{#1}{\,\longleftarrow\,}}
\nc{\loto}[1]{\overset{#1}{\leftarrow}}
\nc{\lto}{\leftarrow}
\nc{\lun}{\mathrm{lun}} % left unitor
\nc{\Mackintro}[1]{(Mack\,\ref{Mack-#1-intro})}
\nc{\Mack}[1]{(Mack\,\ref{Mack-#1})}
\nc{\Mid}{\,\big|\,}
\nc{\MMod}{\,\text{-}\Mod}%
\nc{\PProj}{\,\text{-}\Proj}
\nc{\CComod}{\,\text{-}\Comod}
\dmo{\mods}{mod}%
\nc{\mmods}{\,\text{-}\mathrm{mod}}%
\nc{\MM}{\cat{M}}% a Mackey 2-functor
\nc{\Muniv}{\cat{M}_{\mathsf{univ}}}
\nc{\Ncell}{\rotatebox[origin=c]{0}{$\Uparrow$}} % treated as the others, for uniform size
\nc{\NEcell}{\rotatebox[origin=c]{135}{$\Downarrow$}} % North-East oriented 2-cell arrow
\nc{\NN}{\cat{N}}% another Mackey 2-functor
\nc{\noloc}{\nobreak\mspace{6mu plus 1mu}{:}\nonscript\mkern-\thinmuskip\mathpunct{}\mspace{2mu}}% mirror of \colon
\nc{\NWcell}{\rotatebox[origin=c]{-135}{$\Downarrow$}} % North-West oriented 2-cell arrow
\nc{\oEcell}[1]{\overset{\scriptstyle #1}{\Ecell}} % to be used ONLY for 2-cells in xypic diagrams [I]
\nc{\oWcell}[1]{\overset{\scriptstyle #1}{\Wcell}} % same, but for Wcell
\nc{\ointo}[1]{\overset{#1}{\rightarrowtail}}
\nc{\olto}[1]{\overset{#1}\lto}
\nc{\onto}{\mathop{\twoheadrightarrow}}
\nc{\op}{{\mathrm{op}}}
\nc{\xto}[1]{\xrightarrow{#1}}% much simpler!
\nc{\oto}[1]{\overset{#1}\to}
\nc{\Paul}[1]{{\color{Blue}#1}}
\nc{\pih}[1]{\tau_{1}#1}%
\nc{\Pout}[1]{\Paul{\sout{#1}}}
\nc{\qquadtext}[1]{\qquad\textrm{#1}\qquad}
\nc{\quadtext}[1]{\quad\textrm{#1}\quad}
\nc{\ra}{\rightarrow}
\nc{\reps}{{}^{r\!}\eps}
\nc{\restr}[1]{{|_{\scriptstyle #1}}}% to restrict a map
\nc{\reta}{{}^{r\!}\eta}
\nc{\run}{\mathrm{run}} % right unitor
\nc{\Sad}{\mathsf{Sad}}
\nc{\SAD}{\mathsf{SAD}}
\nc{\sbull}{{\scriptscriptstyle\bullet}}
\nc{\Scell}{\rotatebox[origin=c]{0}{$\Downarrow$}} % treated as the others, for uniform size
\nc{\SEcell}{\rotatebox[origin=c]{45}{$\Downarrow$}} % South-East oriented 2-cell arrow
\nc{\SET}[2]{\big\{\,#1\Mid#2\,\big\}}
\nc{\set}{\mathsf{set}} % category of finite sets
\nc{\Set}{\mathsf{Set}}% category of sets
\nc{\smallfaithful}{\mathsf{f}}% faithful
\nc{\smallslash}{{}^{\scriptscriptstyle/}}
\nc{\smat}[1]{\left(\begin{smallmatrix} #1 \end{smallmatrix}\right)}
\nc{\spanG}{{\widehat{\mathsf{gp}\,\,}\!\!\mathsf{d}}{}^\smallfaithful_{\!{}^{\scriptscriptstyle/}\!G}}% span of groupoids over G.
\nc{\Spanhat}{\textrm{\sf S}\widehat{\textrm{\sf pan}}} %
\nc{\Span}{\Spanname}% the span bicategory of a (2,1)-category
\nc{\sset}{\textrm{-}\set}
\nc{\str}{\mathsf{str}}
\nc{\SWcell}{\rotatebox[origin=c]{-45}{$\Downarrow$}} % South-West oriented 2-cell arrow
\nc{\too}{\mathop{\longrightarrow}\limits}
\nc{\tristars}{\begin{center} $ *\;*\;* $ \end{center}}
\nc{\tSpan}{\pih{\Spanname}}% 1-truncation of Span
\nc{\Unit}{\mathbb{1}}
\nc{\undersett}[1]{\underset{\scriptstyle #1}}
\nc{\un}{\mathrm{un}} % "unitor" of pseudofunctors
\nc{\vcorrect}[1]{{\vphantom{\vbox to #1em{}}}}
\nc{\Wcell}{\rotatebox[origin=c]{90}{$\Uparrow$}} % treated as the others, for uniform size
\nc{\what}[1]{\widehat{\cat{#1}}}% span category with name #1.
\nc{\xra}{\xrightarrow}

\nc{\xBur}{\mathrm{B^c}} % crossed Burnside ring
\nc{\xBurk}{ \mathrm{B}^{\mathrm{c}}_{\kk} } % crossed Burnside k-algebra
\nc{\Bur}{\mathrm{B}} % ordinary Burnside ring
\nc{\Burk}{\Bur_{\kk}} % ordinary Burnside k-algebra

\nc{\isoTo}{\overset{\sim}{\Rightarrow}}
\nc{\isoc}[3]{#1\,{\diamond}_{_{\!#3}}#2}
\nc{\Isoc}[3]{(\isoc{#1}{#2}{#3})}

\nc{\lproj}{\mathrm{Lp}} % left projection maps
\nc{\rproj}{\mathrm{Rp}} % right projection maps

\begin{document}

%------------------------------------------------------------------------------

\title{Green 2-functors}

\author{Ivo Dell'Ambrogio}
\date{\today}

\address{\ \medbreak
\noindent Univ.\ Lille, CNRS, UMR 8524 - Laboratoire Paul Painlev\'e, F-59000 Lille, France}
\email{ivo.dell-ambrogio@univ-lille.fr}
\urladdr{http://math.univ-lille1.fr/$\sim$dellambr}

\begin{abstract} \normalsize
We extend the theory of Mackey 2-functors \cite{BalmerDellAmbrogio20} by defining the appropriate notion of rings, namely \emph{Green 2-functors}. 
After providing the first results of our theory and abundant examples, we show how all classical Green functors familiar from representation theory and topology arise by decategorification, in various ways, of some Green 2-functor occurring in Nature.
\end{abstract}

\thanks{Author supported by Project ANR ChroK (ANR-16-CE40-0003) and Labex CEMPI (ANR-11-LABX-0007-01).}

 % 55P91

\maketitle

%------------------------------------------------------------------------------

\tableofcontents

%------------------------------------------------------------------------------
\vskip-\baselineskip\vskip-\baselineskip\vskip-\baselineskip

%------------------------------------------------------------------------------
\section{Introduction}
\label{sec:introduction} %
%\bigbreak
%------------------------------------------------------------------------------

Mackey 2-functors, introduced in \cite{BalmerDellAmbrogio20} \cite{BalmerDellAmbrogio21}, provide an axiomatic framework for the induction and restriction functors which are omnipresent in equivariant mathematics, at least in the case when the groups are \emph{finite} and the categories \emph{additive}. 
Thus Mackey 2-functors categorify the Mackey (1-) functors commonly used in representation theory and topology, whose values are abelian groups. 
A ring (\ie monoid) in the category of Mackey functors is called a \emph{Green functor}. 
Examples abound, and certainly provide the classical theory of Mackey functors with much of its power; see \cite{Webb00} \cite{Lewis80} \cite{Bouc97}.

We would similarly expect the notion of \emph{Green 2-functor}, \ie of a Mackey 2-functor equipped with a compatible multiplicative structure, to be quite useful.
The goal of this paper is to pin down a precise and workable definition of Green 2-functor, to develop its basic theory, and to provide many examples.

Recall that, formally, a Mackey 2-functor is a contravariant 2-functor 
\[\cat M\colon \gpd^\op \longrightarrow \ADD \] 
from the 2-category of finite groupoids to the 2-category of additive categories, satisfying a few reasonable axioms.
(Variations in the source and target of $\cat M$ are possible, \eg in order to consider Mackey 2-functors for a fixed group rather than global ones.)
In particular, we ask the ``restriction'' functor $i^*=\cat M(i)\colon \cat M(G)\to \cat M(H)$ along every \emph{faithful} morphism $i\colon H\to G$ of groupoids (\eg the inclusion of a subgroup $H\leq G$) to admit an ``induction'' functor $i_*\colon \cat M(H)\to \cat M(G)$ which is both left and right adjoint to~$i^*$. Both adjunctions are required to satisfy the base change formula, called \emph{Mackey formula} in this context. By the Rectification Theorem, one can always choose the adjunctions $i_*\dashv i^* \dashv i_*$ so that the two base-change maps are mutual inverses (\Cref{Rem:rectification}); this ``strict'' Mackey formula will be of crucial importance for the present paper.

We now define a \emph{Green 2-functor} to be a Mackey 2-functor $\cat M$ taking values in additive monoidal categories $\cat M(G)$ and strong monoidal additive functors $u^*\colon \cat M(G)\to \cat M(H)$, which we require to ``preserve inductions''. 
In order to make the latter condition precise, we first express the tensor product as an external pairing $\cat M(G_1)\otimes \cat M(G_2)\to \cat M(G_1 \times G_2)$; then we require the latter to satisfy base change for all faithful functors (as in the Mackey formula) with respect to \emph{both} variables $G_1,G_2$; see \Cref{Def:2Green} for details.
This is different from the analogous condition on the internal pairing $\cat M(G)\otimes \cat M(G)\to \cat M(G)$, and is similar to how \emph{monoidal derivators} are defined (see \cite{GPS14}). Compared to monoidal derivators, though,  we must deal with the extra difficulty of working with tensor products of additive categories rather than just Cartesian products. This complicates definitions and has the consequence that the (term-wise) tensor product of two Mackey 2-functors is \emph{not} a Mackey 2-functor (\Cref{Rem:2Mack-no-tensor}). The latter fact turns out to be rather inconsequential, and the most basic theory of Green 2-functors closely parallels that of monoidal additive derivators.
Similarly, we will also easily define braided and symmetric Green 2-functors and modules over them.

Examples of Green 2-functors in our sense are plentiful throughout equivariant mathematics, as we will try to impress upon the reader in \Cref{sec:exa}.
For instance: abelian and derived categories of group algebras in representation theory, equivariant stable homotopy categories in topology, categories of equivariant sheaves in geometry, and equivariant Kasparov categories in noncommutative geometry, all provide symmetric Green 2-functors.  
\begin{center}$*\;*\;*$\end{center}

Armed with the above definition we can now formulate our first results.

To begin with, Green 2-functors, unlike monoidal derivators (\cf \Cref{Rem:mon-der-proj}), always satisfy projection formulas:

\begin{Thm}[Projection formulas; \Cref{Thm:proj-formula}] 
\label{Thm:proj-formulas-Intro}
Let $\cat M$ be a Green 2-functor, and write $\otimes_G$ for the internal tensor product of the category~$\cat M(G)$.
For any faithful morphism $i\colon H\to G$, there are in $\cat M(G)$ canonical natural isomorphisms
\[
i_*( Y \otimes_H i^* X) \cong i_* (Y) \otimes_G X
\quad \textrm{ and } \quad
i_*(  i^* X \otimes_H Y ) \cong  X \otimes_G i_* (Y) 
\]
for all $X\in \cat M(G)$ and $Y\in \cat M(H)$.
\end{Thm}

More generally, this holds for any \emph{bimorphism} $(\cat M, \cat N) \to \cat L$ between three a~priori distinct Mackey functors, \ie for any external pairing which preserves inductions in both variables (\Cref{Def:bimorphism}). 
For example, it also holds for the action on left and right modules over a Green 2-functor. 
In fact we will show that a pairing between Mackey 2-functors is a bimorphism \emph{if and only if} it satisfies the projection formulas (see \Cref{Thm:proj-implies-bimorphism}). This provides an alternative definition of a Green 2-functor: It is a Mackey 2-functor whose underlying 2-functor $(-)^*$ takes values in monoidal additive categories and satisfies the projection formulas (\Cref{Def:quasi-Green}).

Let us stress that the projection formulas of \Cref{Thm:proj-formulas-Intro} are extremely nice, because the maps are given by mate transformations of the monoidal structure with respect to \emph{one} of the adjunctions, while their inverses are obtained by the corresponding mates for the \emph{other} adjunction. This is analogous to, and a consequence of, the above-mentioned strict Mackey formula for the underlying Mackey 2-functor~$\cat M$, and is an important ingredient for proving our next result.

For any Green 2-functor $\cat M$ and faithful $i\colon H\to G$, we may consider the object 
\[ A(i) := i_*(\Unit_H) \in \cat M(G)\] 
induced from the unit $\Unit_H$ of the tensor category~$\cat M(H)$. 
As restriction $i^*$ is strong monoidal, its left-and-right adjoint $i_*$ is both a lax and a colax monoidal functor, hence $A(i)$ becomes a monoid and a comonoid. Even better:

\begin{Thm}[Frobenius structure; Theorems~\ref{Thm:Frobenius} and \ref{Thm:half-braiding}]
\label{Thm:Frob-Intro}
The object $A(i)$ inherits a canonical structure of special Frobenius algebra in the tensor category~$\cat M(G)$, which lifts to a commutative Frobenius algebra in the monoidal center of $\cat M(G)$.
\end{Thm}

In fact, even forgetting the tensor structure on~$\cat M$, the ambijunction $i_*\dashv i^* \dashv i_*$ always induces a Frobenius monad $\mathbb M(i):=i_*i^*$ on~$\cat M(G)$. 
The point here is that the projection formula yields an isomorphism $\mathbb M(i) (X) \cong A(i) \otimes_G X$ of (co)monads so well-behaved that it allows us to transfer all this structure onto the object~$A(i)$; one can say the Frobenius monad $\mathbb M(i)$ is \emph{canonically represented} by the Frobenius algebra~$A(i)$.
The second part of the theorem, mostly relevant if $\cat M$ is not braided, says more precisely that the composite $A(i) \otimes_G X \cong i_*i^*X \cong X \otimes_G A(i)$ of the two projection isomorphisms provides a \emph{half-braiding} on $A(i)$ which is compatible with all the structure on $A(i)$ and also with any other half-braiding on any other object of~$\cat M(G)$.
In short, $A(i)$ is intrinsically commutative in the best possible way.

Morally, the latter results belong to the same school as the study of \emph{relative dualising objects} of~\cite{BalmerDellAmbrogioSanders16}: Both theories study a certain class of adjunctions through the properties of a canonically associated object. (The two are also roughly complementary because, in the shared examples, we get a non-trivial Frobenius object $A(i)$  precisely when the relative dualizing object $\omega_i$ is trivial. 
It is an interesting challenge to develop a framework encompassing both theories; \cf \Cref{Rem:sep-generality}.)
Still in the same vein, the next ``tensor monadicity'' is a consequence of \Cref{Thm:Frob-Intro} and the fact that special Frobenius algebras are in particular separable (co)algebras: 

\begin{Thm} [Tensor monadicity; \Cref{Thm:separable-monoidicity}]
\label{Thm:tensor-mon-Intro}
Suppose the Green 2-functor $\cat M$ takes values in additive monoidal categories which are idempotent complete.
Then for every faithful $i\colon H\to G$ there are canonical equivalences between $\cat M(H)$ and the Eilenberg--Moore categories of $A(i)$-modules and of $A(i)$-comodules in~$\cat M(G)$:
\[
A(i) \CComod_{\cat M(G)} 
\overset{\sim}{\longleftarrow}
\cat M(H)
\overset{\sim}{\longrightarrow}
A(i)\MMod_{\cat M(G)}
\]
If the Green 2-functor $\cat M$ is braided, both are equivalences of braided monoidal categories, where (co)modules are endowed with the usual tensor product over~$A(i)$.
\end{Thm}

This theorem refines the general separable monadicity of the underlying Mackey 2-functor (see  \cite[\S2.4]{BalmerDellAmbrogio20}).
It also recovers and improves the disparate results of \cite{BalmerDellAmbrogioSanders15} under a unified and more conceptual proof. 

Finally, let us explain how our Green 2-functors relate to ordinary Green 1-functors.  
(Here we understand \emph{ordinary} Mackey and Green functors in a very flexible way, covering all the usual variants in use for finite groups, including local and global ones, inflation functors, biset functors, etc.; see Definitions~\ref{Def:Mackey-fun} and~\ref{Def:Green-fun}.) 
There are at least two ways in which we can ``decategorify'' a Mackey 2-functor $\cat M$ in order to obtain an ordinary Mackey functor~$M$.
Firstly, and going back to \cite{Dress73}, we have the ordinary \emph{K-decategorification}, where we postcompose $\cat M$ with (some version of) the Grothendieck group~$\mathrm K_0$ in order to get $G\mapsto M(G)=\mathrm K_0 (\cat M(G))$; see \Cref{Cons:K-decat}. 
Secondly, we have the \emph{Hom-decategorification} introduced in \cite{BalmerDellAmbrogio21pp}, where we choose pairs of objects $X_G,Y_G\in \cat M(G)$ for all~$G$ in a coherent way in order to get $G\mapsto M(G)= \Hom_{\cat M(G)}(X_G,Y_G)$; see \Cref{Cons:Hom-decat}.

\begin{Thm} [Decategorification; Theorems~\ref{Thm:Green-K-decat}, \ref{Thm:Green-End-decat} and~\ref{Thm:Hom-decat-co-mon}]
\label{Thm:decats-Intro}
Let $\cat M$ be any Mackey 2-functor. Then:
\begin{enumerate}[\rm(a)]
\item Every Green 2-functor structure on $\cat M$ yields through K-decategorification an ordinary Green functor structure on $M\mapsto \mathrm K_0(\cat M(G))$. (For this, assume also that each value category $\cat M(G)$ is essentially small.)
\item Every coherent choice of objects with $X_G=Y_G$ yields through Hom-decat\-eg\-orif\-ication an ordinary Green functor $G\mapsto \End_{\cat M(G)}(X_G)$. In particular, if $\cat M$ is a Green 2-functor we may take each $X_G= Y_G$ to be the unit $\Unit_G\in\cat M(G)$.
\item Suppose $\cat M$ is a Green 2-functor. Every coherent choice of objects $X_G, Y_G$ in the tensor category~$\cat M(G)$ together with a comonoid structure on $X_G$ and a monoid structure on~$Y_G$ yields, through Hom-decat\-eg\-orif\-ication, a Green functor $M$ where $M(G)= \Hom(X_G,Y_G)$ has the induced convolution product.
\end{enumerate}
\end{Thm}

 Similar results hold for modules.
Other variants are possible, such as using higher algebraic K-theory if $\cat M$ takes values in exact categories, or using graded Hom-groups if $\cat M$ takes values in categories with a grading (\eg because they are triangulated). See \Cref{sec:origins} for details.

As listed at the end of the article, all classical examples of Green functors we can think of arise either as a K-decategorification or as a Hom-decategorification (most often at the tensor units) of some Green 2-functor occurring in Nature. Thus as it turns out, all these useful Green functors are mere shadows projected---at various angles---by richer, higher-dimensional structures.
Interestingly, some examples such as the Burnside ring and the complex representation ring are \emph{both} a K- and a Hom-decategorification of (different) Green 2-functors.

\begin{Org*}
Sections \ref{sec:Mackey}-\ref{sec:pairings} contain preliminaries on, respectively: Mackey 2-functors, tensor 2-categories of linear categories, and pairings of 2-functors.
We work throughout under more abstract hypotheses than in this Introduction, allowing Mackey 2-functors to have more general source 2-categories $\GG$ (instead of~$\gpd$), target 2-categories $\mathbb A$ (instead of $\ADD$), and classes of 1-morphisms $\JJ\subseteq \GG$ at which induction is defined. 
In particular, we can work linearly over any commutative coefficient ring~$\kk$, and at little cost our setup will allow applications beyond finite groups (\Cref{Rem:sep-generality}) or beyond groups altogether.
All hypotheses are detailed in~\Cref{sec:Mackey}.
\Cref{sec:proj} treats the projection formulas, which are then applied in
Sections~\ref{sec:Frob}-\ref{sec:monadicity}. 
Let us stress that these applications all hold as more abstract theorems on monoidal functors $i^*$ satisfying the hypotheses in~\Cref{Rem:extra-axiomatization}. 
\Cref{sec:exa}
offers many families of examples of Green 2-functors from throughout mathematics.
After some generalities on ordinary Mackey and Green functors in \Cref{sec:1-Green},
we conclude by explaining in \Cref{sec:origins} the various ways in which our Green 2-functors give rise to the classical Green functors.

\end{Org*}

\begin{Ack*}
I am grateful to K\"ur\c sat S\"ozer for references on monoidal bicategories and to Alexis Virelizier for informative discussions on Frobenius structures and monoidal centers which led to a simplification of \Cref{sec:half-braiding}. I owe many thanks to an anonymous referee for their careful reading and thoughtful comments.
\end{Ack*}

%------------------------------------------------------------------------------
\section{Mackey 2-functors}
\label{sec:Mackey}%
%\bigbreak
%------------------------------------------------------------------------------

We use the language of 2-categories in a standard way, as recalled for instance in \cite[App.\,A]{BalmerDellAmbrogio20}.
In particular, ``2-categories'' and ``2-functors'' will always be understood to be strict, in contrast to ``bicategories'' and ``pseudofunctors'' (\ie B\'enabou's homomorphisms).
Further recollections will be provided along the way.

\begin{Hyp} \label{Hyp:spannable-pair}
In this article, $(\GG;\JJ)$ will denote a \emph{spannable pair} in the sense of \cite[Def.\,2.31]{DellAmbrogio21ch}.
This means the following (see \emph{loc.\,cit.\ }for more details):
\begin{enumerate}[\rm(1)]
\item $\GG$ is an \emph{essentially small (2,1)-category}, \ie a 2-category whose Hom categories are small groupoids and which, up to equivalence, only has a small set of objects.
\item $\GG$ is an \emph{extensive} 2-category: it admits arbitrary finite coproducts, and for any pair of objects $G,H$ the induced pseudofunctor between comma 2-categories
\[
\GG/G \times \GG/H \longrightarrow \GG/(G\sqcup H), \quad  (X \overset{u}{\to} G, Y \overset{v}{\to} H) \mapsto (X \sqcup Y \overset{u\sqcup v}{\longrightarrow} G\sqcup H)
\]
is a biequivalence (\ie an equivalence of bicategories).
\item $\JJ$ is a 2-full 2-subcategory of $\GG$ containing all equivalences (in particular, all identity 1-cells) and closed under taking isomorphic 1-cells.
\item Pseudopullbacks of 1-cells of $\JJ$ exist and are again in~$\JJ$. That is, if $i\colon H\to G$ and $u\colon K\to G$ are 1-cells with $i\in\JJ$ then there exists a pseudopullback square
\begin{equation}
\label{eq:Mackey-square-in-Hyp}%
\vcenter{
\xymatrix@C=10pt@R=10pt{
& P \ar[ld]_-{p} \ar[rd]^{q} \ar@{}[dd]|{\isocell{\gamma}} & \\
H \ar[rd]_-{i} &  & K \ar[ld]^-{u} \\
& G &
}}
\end{equation}
 in~$\GG$ (which we usually call a \emph{Mackey square}; \cf \cite[\S2.1-2]{BalmerDellAmbrogio20}), and $q\in \JJ$.
\item
For every finite set $\{u_\ell\colon H_\ell\to G\}_\ell$ of 1-cells with common target, the induced 1-cell $(u_\ell)_\ell\colon \sqcup_\ell H_\ell\to G$ on the coproduct is in $\JJ$ iff every $u_\ell$ is. 
\end{enumerate}
These axioms guarantee that \Cref{Def:M2F} below makes sense, and also that we may use $(\GG;\JJ)$ in \Cref{sec:1-Green} to form a suitable (bi)category of spans and define an associated variant notion of (ordinary) Mackey and Green functors.
\end{Hyp}

\begin{Exas} \label{Exas:GGJJ}
The reader should keep the following examples in mind; in fact, in this article we will not look for applications beyond these. In all four cases, arbitrary Mackey squares exist and are provided by iso-comma squares of groupoids:
\begin{enumerate}[\rm(a)]
\item \label{it:exa:all}
$\GG= \JJ = \gpd$ is the 2-category of all finite groupoids, functors and natural transformations between them.
\item \label{it:exa:inflation}
$\GG= \gpd$ and $\JJ = \gpdf$ is the 2-full 2-subcategory of finite groupoids and faithful functors.
\item \label{it:exa:global-Mackey}
$\GG= \JJ = \gpdf$.
\item \label{it:exa:G-local}
$\GG = \JJ = \gpdG$ is the comma 2-category of finite groupoids and faithful functors over a fixed finite groupoid~$G$. Note that there is a biequivalence $\gpdG \simeq G\sset$ with the (1-)category of finite left $G$-sets, under which Mackey squares correspond to usual pullbacks of $G$-sets (see \cite[App.\,B]{BalmerDellAmbrogio20}).
\end{enumerate}
We will always identify a group $G$ with the one-object groupoid having $G$ as automorphism group, and a group morphism with the corresponding functor.
\end{Exas}

\begin{Def} [{Mackey 2-functor \cite{BalmerDellAmbrogio20}}]
\label{Def:M2F}
Suppose $\bbA$ is some \emph{additive 2-category} as in \cite[App.\,A.7]{BalmerDellAmbrogio20}; that is, $\bbA$ is a (possibly very large\footnote{What we mean by \emph{very large} is that the Hom categories of $\bbA$ may not be essentially small, \ie may have a proper class of objects, even when taken up to equivalence.}) 2-category admitting finite direct sums (\ie biproducts) of objects, whose Homs are additive categories and whose horizontal compositions are additive functors of both variables.  
Then an \emph{$\bbA$-valued Mackey 2-functor for $(\GG;\JJ)$}, or just \emph{Mackey 2-functor}, is a 2-functor 
\[ \MM\colon \GG^\op\longrightarrow \bbA \] 
satisfying the following axioms:
\begin{enumerate}[\rm(1)]
\item \emph{Additivity:} $\cat M$ sends finite coproducts to direct sums, $\cat M(\sqcup_\ell G_\ell) \overset{\sim}{\to} \oplus_\ell \cat M(G_\ell)$.
\item \emph{Induction:} if $(i\colon H\to G)\in \JJ$, the 1-cell $i^*:=\cat M(i)\colon \cat M(G)\to \cat M(H)$ admits in $\bbA$ a left adjoint $i_!$ and a right adjoint~$i_*$.  
\item \emph{Mackey formulas:} for every Mackey square \eqref{eq:Mackey-square-in-Hyp} with $i,q\in \JJ$, the left mate $\gamma_!\colon q_!p^*\Rightarrow u^*i_!$ and right mate $(\gamma^{-1})_*\colon u^*i_* \Rightarrow q_*p^* $ are invertible 2-cells of~$\bbA$.
\item \emph{Ambidexterity:} for every $i\in \JJ$, there is an isomorphism $i_!\simeq i_*$.
\end{enumerate}
Whenever convenient, we will suppose the choices of  $(\GG;\JJ)$ and $\bbA$ to be understood and will omit them from our terminology and notations.
\end{Def}

\begin{Exas} \label{Exas:bbA}
For the target 2-category~$\bbA$, the reader should keep in mind the following three examples, which will soon be review in~\Cref{sec:tensors}: 
the 2-category of all (possibly large) additive categories and additive functors; all $\kk$-linear additive categories and $\kk$-linear functors, over some commutative ring~$\kk$; or the 2-full 2-subcategory of idempotent complete categories.
\end{Exas}

\begin{Rem} \label{Rem:rectification}
By the Rectification Theorem \cite[Thm.\,1.2.1]{BalmerDellAmbrogio20}, every Mackey 2-functor for $(\GG;\JJ)$ as in \Cref{Exas:GGJJ} and with values in~$\bbA$ as in \Cref{Exas:bbA}  can be \emph{rectified} (note that this covers all concrete exemples mentioned in this article; see Sections \ref{sec:exa}-\ref{sec:origins}). 
What this means is that there exists a (unique) choice of specific adjunctions $i_! \dashv i^*$ and $i^*\dashv i_*$ for all~$i$, whose units and counits we denote by
\[
\leta\colon \Id \Rightarrow i^*i_! 
\qquad
\leps\colon i_!i^*\Rightarrow \Id
\qquad
\reta\colon \Id \Rightarrow i_*i^*
\qquad
\reps \colon i^*i_* \Rightarrow \Id ,
\]
which satisfy certain extra properties, including all of the following: For all $i\in \JJ$ the two adjoints $i_!=i_*$ are the same as 1-morphisms of~$\bbA$; 
the two induced pseudofunctors $\JJ^\mathrm{co}\to \bbA$ given by $i\mapsto i_!$ and $i\mapsto i_*$ agree; 
the \emph{strict Mackey formula} 
\begin{equation} \label{eq:strict-MF}
(\gamma_!)^{-1} = (\gamma^{-1})_*
\end{equation}
holds for every Mackey square as in~\eqref{eq:Mackey-square-in-Hyp}; and the \emph{special Frobenius relation}
\begin{equation} \label{eq:special-Frob}
\reps \circ \leta = \id_{\Id}
\end{equation}
holds for every $i\in \JJ$.
\end{Rem}

\begin{Hyp} \label{Hyp:rectification}
We will tacitly assume throughout the article that all Mackey 2-functors are rectified, \ie satisfy the extra conditions of \Cref{Rem:rectification}; this is because our proofs will make use of the relations~\eqref{eq:strict-MF} and~\eqref{eq:special-Frob}. 
As already mentioned, this hypothesis is actually redundant for all combinations of spannable pairs $(\GG;\JJ)$ and target 2-categories~$\bbA$ occurring in our concrete examples.
%Even for general spannable pairs $(\GG;\JJ)$, we assume throughout that our Mackey 2-functors are all rectified, \ie come with adjunctions as in \Cref{Rem:rectification}.
\end{Hyp}

\begin{Hyp} \label{Hyp:prods}
From \Cref{sec:pairings} onward, we will (mostly) assume that the 2-category $\GG$ admits arbitrary finite products and that products of 1-morphisms in~$\JJ$ are again in~$\JJ$. 
We call such spannable pairs \emph{Cartesian}. 
Of the pairs in \Cref{Exas:GGJJ}, $(\gpd;\gpd)$, $(\gpd;\gpdf)$ and $(\gpdG;\gpdG)$ are Cartesian, but $(\gpdf;\gpdf)$ is not because in this case~$\GG$ contains no nontrivial projection functors $G_1\times G_2\to G_i$.
The latter pair will be needed to handle a single but important example of Green 2-functor,  namely the stable module category (\Cref{Exa:quots}). 
%The notion of a \emph{quasi}-Green 2-functor will also be introduced especially for it (but see also \Cref{Rem:misalignment}).
\end{Hyp}

%------------------------------------------------------------------------------
\section{Linear categories}
\label{sec:tensors}%
%\bigbreak
%------------------------------------------------------------------------------

We briefly recall some generalities on linear categories.
Fix a commutative ring~$\kk$.

\begin{Cons}
\label{Rec:enriched-tensor}
%Let $\kk$ be a commutative ring.
If $\cat A, \cat B$ are two $\kk$-linear categories (\ie categories enriched in $\kk$-modules), their \emph{tensor product} $\cat A \otimes_\kk \cat B$ is the $\kk$-linear category with object-set $\Obj \cat A \times \Obj \cat B$, Hom spaces
$
(\cat A \otimes_\kk \cat B) ((A,B),(A',B')) :=  \cat A(A,A') \otimes_\kk  \cat B(B,B')
$,
and composition $(f \otimes g )(f'\otimes g')= (ff')\otimes (gg')$.
The (plain) functor $\cat A \times \cat B \to \cat A\otimes_\kk \cat B$ which is the identity on objects and $(f,g)\mapsto f\otimes g$ on maps is the \emph{universal $\kk$-bilinear} functor: It is $\kk$-linear in both variables separately, and for every $\kk$-linear category $\cat C$ it induces an equivalence (in fact an isomorphism)
\begin{equation} \label{eq:UP-tensor}
\Fun_\kk (\cat A \otimes_\kk \cat B, \cat C) \overset{\sim}{\longrightarrow}\Fun_{\kk\textrm{-bilin}} (\cat A \times \cat B , \cat C)
\end{equation}
between the category of $\kk$-linear functors on $\cat A\otimes_\kk \cat B$ and that of $\kk$-bilinear functors on $\cat A \times \cat B$ (with morphisms on both sides provided by natural transformations).
\end{Cons}

If the $\kk$-linear categories $\cat A,\cat B$ are additive (\ie admit all finite direct sums), there is no reason for their tensor product to be additive, and similarly with the property of being idempotent complete  (\cite[Def.\,A.6.9]{BalmerDellAmbrogio20}).
(Note in particular that the coordinate-wise constructions do not suffice: In $\cat A\otimes_\kk \cat B$, an object of the form $(\bigoplus_i A_i , \bigoplus_i B_i)$ is the direct sum $\bigoplus_{i,j} (A_i,B_j)$ rather than $\bigoplus_i(A_i,B_i)$; similarly, an idempotent endomorphism $e=e^2$ of $(A,B)$ is in general not simply of the form $p\otimes q$ for idempotents $p$ on~$A$ and $q$ on~$B$, hence it does not help to split idempotents in each coordinate.)
Luckily, simple universal constructions let us always adjoin finite direct sums and split subobjects.
Denote by $\CAT_\kk$ the (very large) 2-category of all (non-necessarily small) $\kk$-linear categories. 

\begin{Cons} [{Additive hull}]
\label{Cons:tensor-addhull}
The fully faithful inclusion $\ADD_\kk \hookrightarrow \CAT_\kk$ of the 2-subcategory of additive $\kk$-linear categories is reflexive: There is an \emph{additive hull} 2-functor $(-)^\oplus\colon \CAT_\kk \to \ADD_\kk$ as well as a family of $\kk$-linear functors $\cat C\to \cat C^\oplus$, pseudonatural in~$\cat C \in \CAT_\kk$, which induce equivalences
\[
\Fun_\kk (\cat C^\oplus , \cat D) \overset{\sim}{\longrightarrow} \Fun_\kk (\cat C, \cat D)
\]
for all additive $\kk$-linear categories~$\cat D$.
Concretely, the additive hull $\cat C^\oplus$ has finite lists of objects of~$\cat C$ for its objects and the evident matrix spaces for its Hom $\kk$-modules (see \eg \cite[\S2]{DellAmbrogioTabuada14} for details). 

% Ivo: extra details:
%as the category whose objects are the finite strings $C_1\ldots C_n$ in the alphabet $\Obj \cat C$ and whose Hom $\kk$-modules are the matrix spaces
%%
%\[
%\cat C^\oplus ( C_1\ldots C_m , D_1\ldots D_n) := \bigoplus_{i = 1}^n \bigoplus_{j= 1}^m  \cat C(C_j, D_i)
%\]
%%
%with composition combining the composition of~$\cat C$ with matrix multiplication in the evident way. The category $\cat C^\oplus$ is additive, with the string $C_1\ldots C_n$ providing a direct sum $C_1\oplus \ldots \oplus C_n$, the empty string in particular providing a zero object. The canonical (fully faithful) $\kk$-linear functor $\cat C\to \cat C^\oplus$ sends $f\in \cat C( C, D)$ to the $1\times 1$~matrix whose only entry is~$f$.
\end{Cons}

\begin{Cons} [{Idempotent completion}]
\label{Cons:tensor-ic} 
The inclusion $\ADDic_\kk\hookrightarrow \ADD_\kk$ of the 2-subcategory of idempotent complete categories is reflexive: 
There is an \emph{idempotent completion} 2-functor $(-)^\natural\colon \ADD_\kk \to \ADDic_\kk$, as well as $\kk$-linear functors $\cat D\to \cat D^\natural$, pseudonatural in~$\cat D\in \ADD_\kk$, which induce equivalences
\[
\Fun_\kk ( \cat D^\natural , \cat E) \overset{\sim}{\longrightarrow} \Fun_\kk (\cat D, \cat E)
\]
for all idempotent complete additive $\kk$-linear~$\cat E$.
Concretely, the objects of $\cat D^\natural$ are objects of $\cat D$ endowed with an idempotent $e= e^2$, and its morphisms are those of $\cat D$ which absorb the idempotents on source and target (see \eg \cite[A.6.10]{BalmerDellAmbrogio20}).
% Ivo: details:
%Concretely, the idempotent completion $\cat D^\natural$ can be defined as the category whose objects are pairs $(D,e)$ with $D$ an object of~$\cat D$ and $e=e^2\in \cat D(D,D)$ an idempotent morphism, and whose Hom $\kk$-modules are the submodules
%%
%\[
%\cat D^\natural ((D,e),(D',e')) := e' \circ \cat D(D,D') \circ e = \{ f\mid e'fe = f\} \;\; \subseteq \;\; \cat D(D,D')
%\]
%%
%with composition as in~$\cat D$. 
%The (fully faithful) canonical  functor $\cat D\to \cat D^\natural$ sends $D$ to $(D,\id_D)$ and a map $f\in \cat D(D,D')$ to `itself'.
\end{Cons}

\begin{Rem}
Constructions \ref{Rec:enriched-tensor}-\ref{Cons:tensor-ic} preserve the smallness of categories, hence restrict as operations on the respective full 2-subcategories $ \Cat_\kk \subset \CAT_\kk$,  $\Add_\kk\subset \ADD_\kk$ and $\Add^\ic_\kk \subset \ADD^\ic_\kk$ of those objects which happen to be small categories.
\end{Rem}

\begin{Rem} \label{Rem:ic-vs-addhull}
Note that the idempotent completion $(-)^\natural$ also makes sense for $\kk$-linear categories which are not additive, and that it preserves the property of being additive.
Thus $(\cat C^\oplus)^\natural$ is always additive and idempotent complete; but this is not true of $(\cat C^\natural )^\oplus$ (just consider the algebra $\cat C= \kk$). 
Usually, we only consider the property of being idempotent complete for categories which are already additive.
\end{Rem}

\begin{Def}
\label{Def:tensor-addic} 
We define the \emph{tensor product} of two additive $\kk$-linear categories $\cat A,\cat B$ to be the additive hull $(\cat A\otimes_\kk \cat B)^\oplus \in \ADD_\kk$, as in \Cref{Cons:tensor-addhull}, of the tensor product of \Cref{Rec:enriched-tensor}.
Similarly, we define the \emph{tensor product} of two idempotent complete (additive) $\kk$-linear categories $\cat A,\cat B$ to be the idempotent completion
 $((\cat A\otimes_\kk \cat B)^\oplus)^\natural$  as in \Cref{Cons:tensor-ic} of the additive hull; it belongs to $\ADDic_\kk$ by \Cref{Rem:ic-vs-addhull}.
In the following, we will simply write $\cat A\otimes \cat B$ for all three tensor products, counting on the context to clarify which one is meant.
\end{Def}

% Ivo: don't need, so should leave out really!
%\begin{Rem}
%\label{Rem:special-saturation} 
%It so happens that if $\cat A$ and $\cat B$ are additive, then $(\cat A \otimes_\kk \cat B)^\natural$ is automatically idempotent complete \emph{and} additive, although $\cat A \otimes_\kk \cat B$ may not be additive! 
%(The reason is that the direct sum of a finite family $\{(A_i,B_i)\}_i $ of objects in $ \cat A\otimes_\kk \cat B$  can be constructed as a retract of the direct sum $\bigoplus_{i,j} (A_i,A_j)$, which is represented by $(\bigoplus_i A_i, \bigoplus_j B_j)$ in $\cat A \otimes_\kk \cat B$.)
%Thus, in the definition of the tensor product of idempotent complete additive $\cat A,\cat B$, we do not need to take the additive hull as a middle step: 
%The canonical comparison $(\cat A\otimes_\kk \cat B)^\natural \overset{\sim}{\to} ((\cat A\otimes_\kk \cat B)^\oplus)^\natural$ is an equivalence.
%Sadly, we have no use here for this observation.
%\end{Rem}

\begin{Rem} \label{Rem:UP-tensors}
By combining the universal property \eqref{eq:UP-tensor} with that of the 2-reflections $(-)^\oplus$ and $(-)^\natural$, we obtain analogous universal properties for the tensor product $\cat A \otimes \cat B$ of additive, resp.\ idempotent complete additive, $\kk$-linear categories. Namely, the canonical (plain) composite functor
$ \cat A\times \cat B \to \cat A \otimes \cat B $
is pseudonatural in $\cat A,\cat B\in \CAT_\kk$ and induces an equivalence 
\begin{equation}  \label{eq:UP-others}
\Fun_\kk (\cat A \otimes \cat B, \cat C) \overset{\sim}{\longrightarrow}\Fun_{\kk\textrm{-bilin}} (\cat A \times \cat B , \cat C)
\end{equation}
for all $\cat C \in \ADD_\kk$, resp.\ all $\cat C \in \ADDic_\kk$.
\end{Rem}

\begin{Exa}
\label{Exa:additive-tensor}
By specializing \Cref{Rec:enriched-tensor} and \Cref{Def:tensor-addic} to the case of the integers $\kk = \mathbb Z$, we respectively obtain the tensor product of preadditive (\ie $\mathbb Z$-linear), additive, and idempotent complete additive categories.
\end{Exa}

The next proposition says that each of the above three tensor products defines a symmetric tensor structure on the corresponding 2-category. 
More precisely, recall that a \emph{symmetric monoidal bicategory} is by definition a one-object tricategory in the sense of \cite{GPS95} equipped with a symmetry (see also \cite{McCrudden00} \cite[\S12]{JohnsonYau21}). A \emph{symmetric monoidal 2-category} is one whose underlying bicategory is strict.

\begin{Prop} \label{Prop:tensor-ADD}
Each of the (very large) 2-categories $\CAT_\kk$, $\ADD_\kk$ and $\ADD_\kk^\ic$ becomes a symmetric monoidal 2-category, when equipped with the tensor product of \Cref{Rec:enriched-tensor} or \Cref{Def:tensor-addic}, respectively. These structures restrict to the essentially small 2-categories $\Cat_\kk$, $\Add_\kk$ and $\Add_\kk^\ic$. The unit objects are, respectively: the commutative ring $\kk$ seen as a one-object $\kk$-category, the category $\kk^\oplus \simeq \fgfree(\kk)$ of finitely generated free $\kk$-modules, and the category $(\kk^\oplus)^\natural \simeq \fgproj (\kk)$ of finitely generated projective modules.
\end{Prop}

\begin{proof}
This must be well-known to the experts, but let us sketch a proof as details are scarce in the literature.
Like any other bicategory with finite products, the (very large) 2-category $\CAT$ admits the Cartesian symmetric monoidal structure; see \cite[Thm.\,2.15]{CKWW07}. 
(In fact, $\CAT$ even has finite products in the strict sense, enjoying a universal property up to isomorphism of Hom-categories, so the Cartesian structure is easier to establish here than in general.)
This means that we have two pseudofunctors $\times$ and $1$ for the product and unit, as well as: an associator and left and right unitor adjoint equivalences; a pentagonator and 2-unitor invertible 2-cells; a symmetry adjoint equivalence $\sigma\colon \cat A\times \cat B \overset{\sim}{\to} \cat B \times \cat A$; left and right hexagonator invertible 2-cells; syllepsis invertible 2-cells $\sigma_{\cat B, \cat A}\circ \sigma_{\cat A, \cat B}\overset{\sim}{\Rightarrow} \Id_{\cat A\otimes \cat B}$.
Then for $\mathbb A\in \{ \CAT_\kk, \ADD_\kk,\ADD_\kk^\ic \}$, the universal property of the corresponding tensor product $\otimes$ ensures that we may extend from $\times$ to~$\otimes$ along \eqref{eq:UP-tensor} or \eqref{eq:UP-others}, respectively, all the aforementioned structure. 
It also follows from the  uniqueness of the extended 1- and 2-cells that they still satisfy the numerous necessary coherence relations. 
We leave the lengthy verifications to the reader. 
To do this, it is wise to first upgrade the equivalences \eqref{eq:UP-others} to adjoint equivalences and to exploit that \eqref{eq:UP-tensor} is an isomorphism. 
It also helps that the underlying bicategory $\mathbb A$ is strict. 
\end{proof}

\begin{Rem} \label{Rem:unit-objects}
In the tensor 2-category $\CAT_\kk$, a $\kk$-linear functor $F\colon \Unit=\kk \to \cat A$ on the tensor unit is uniquely determined by an object of~$\cat A$, namely the image by $F$ of the unique object of the $\kk$-category~$\kk$. Similarly in $\ADD_\kk$ and $\ADD_\kk^\ic$, where the tensor unit $\Unit$ is either $\fgfree(\kk)$ or $\fgproj(\kk)$, a $\kk$-linear (hence additive) functor $F\colon \Unit \to \cat A$ is uniquely determined up to equivalence by the object $F(\kk)\in \cat A$.
In the following, we will tacitly and notationally identify the 1-morphism $F\colon \Unit \to \cat A$ with the corresponding object of~$\cat A$.
\end{Rem}

\begin{Rem} \label{Rem:int-hom-distr}
The tensor 2-categories $\CAT_\kk$, $\ADD_\kk$ and $\ADD_\kk^\ic$ admit internal Homs. 
Indeed, the category $\Fun_\kk(\cat B,\cat C)$ of $\kk$-linear functors and natural transformations between any two $\kk$-linear categories is again $\kk$-linear, and is additive or idempotent complete whenever the target~$\cat C$ is. Besides, the natural equivalence
\[
\Fun_\kk(\cat A\otimes \cat B, \cat C) \simeq \Fun_\kk (\cat A , \Fun_\kk (\cat B, \cat C))
\]
is $\kk$-linear. In particular, tensoring with any object has a right 2-adjoint and therefore preserves direct sums, which are in particular products and coproducts. More precisely, 
 the projection functors $\cat B \oplus \cat C= \cat B\times \cat C\to \cat B$ and $\cat B \oplus \cat C = \cat B\times \cat C\to \cat C$
induce in all three cases an equivalence
$
\cat A \otimes (\cat B \oplus \cat C) \overset{\sim}{\to} (\cat A \otimes \cat B ) \oplus (\cat A \otimes \cat C)
$.
\end{Rem}

In a monoidal 2-category, it makes sense to speak of monoids and actions which are associative and unital only  up to coherent isomorphisms:

\begin{Def}[{\cite{DayStreet97} \cite{McCrudden00}}] \label{Def:pseudo-mon}
A \emph{pseudomonoid} $\cat M = (\cat M, m, u, \alpha, \lambda , \rho)$  in a monoidal 2-category~$\bbA$ is an object $\cat M$ of $\bbA$ together with a ``product'' 1-morphism $m\colon \cat M\otimes \cat M\to \cat M$, a ``unit'' 1-morphism $u\colon \Unit \to \cat M$, and invertible 2-morphisms
\[
\vcenter{
\xymatrix@!=14pt{
& & \cat M\otimes \cat M \ar@{}[dd]|{\Scell\; \alpha} \ar[drr]^{m } & & \\
(\cat M \otimes \cat M)\otimes \cat M
  \ar[urr]^{m \otimes \Id}
   \ar[dr]_\simeq & & & & \cat M \\
& \cat M \otimes (\cat M \otimes \cat M) \ar[rr]^-{\Id \otimes m} & & \cat M \otimes \cat M \ar[ur]_m & 
}}
\quad
\vcenter{
\xymatrix@C=14pt{
 &
\cat M \otimes \cat M 
  \ar@{<-}[dl]_-{u \otimes \Id}
  \ar@{<-}[dr]^-{\Id \otimes u}
  \ar[dd]_(.35)m
  \ar@{}[ddl]|{\;\;\;\;\lambda\,\SWcell}
  \ar@{}[ddr]|{\SEcell\,\rho\;\;\;\;} & \\
\Unit \otimes \cat M 
  \ar[dr]_\simeq &  &
\cat M \otimes \Unit
  \ar[dl]^\simeq \\
& \cat M & 
}}
\]
satisfying MacLane's coherence axioms for associators and left and right unitors. 
Suppose $\bbA$ is symmetric with symmetry~$\sigma$. A pseudomonoid is said to be \emph{braided} if it is equipped with a \emph{braiding}, \ie with an invertible 2-cell $\beta \colon m \Rightarrow m \circ \sigma$ satisfying the usual braid axioms.
The pseudomonoid is \emph{symmetric} if the braiding is a symmetry, \ie it satisfies $\beta^2 = \id_m$.
There are 2-categories $\PsMon(\bbA)$, $\BrPsMon(\bbA)$ and $\SymPsMon(\bbA)$ whose objects are the pseudomonoids, braided pseudomonoids and symmetric pseudomonoids in~$\bbA$. 
The (``strong'') 1-morphisms $f= (f,f^2,f^0)\colon (\cat M, m, \alpha, \lambda, \rho)\to (\cat M', m', \alpha', \lambda', \rho')$ in $\PsMon(\bbA)$ are 1-morphisms $f\colon \cat M\to \cat N$ of the underlying objects of~$\bbA$ equipped with 2-isomorphisms $f^2\colon m' \circ (f \otimes f)\Rightarrow f\circ m$ and $f^0\colon u' \Rightarrow f \circ u$ satisfying suitable associativity and unitality axioms. A 2-morphism of such is a 2-morphism of the underlying 1-morphisms $f$ which is compatible with the extra structure $f^2,f^0$. In the braided and symmetrical cases, 1- and 2-morphisms are required to be compatible with~$\beta$. 

Analogously, one can define a \emph{left module} over a pseudomonoid $\cat M$ in~$\bbA$ to be an object $\cat N$ equipped with a \emph{left pseudoaction} of~$\cat M$, \ie with a 1-morphism $\ell\colon \cat M\otimes \cat N\to \cat N$ together with coherent associativity and left-unit 2-isomorphisms $\alpha$ and~$\lambda$ (similarly as above). Then 1-morphisms and 2-morphisms of left modules are defined similarly to those for pseudomonoids, and of course one can also define \emph{right} modules and their 1- and 2-morphisms.

\end{Def}

\begin{Exa} \label{Exa:monADD}
A (braided, symmetric) pseudomonoid in $\bbA = \ADD_\kk$ or $\ADDic_\kk$ amounts, via \Cref{Rem:UP-tensors}, to an (idempotent complete) additive $\kk$-linear category $\cat M$ equipped with a (braided, symmetric) monoidal structure such that its tensor product $\otimes\colon \cat M\times \cat M\to \cat M$ is $\kk$-linear in both variables. A 1-morphism between such is a $\kk$-linear (braided) strong monoidal functor, and a 2-morphism a (braided) monoidal natural transformation. 
In this context, modules are often called \emph{(monoidal) category modules} or \emph{module categories}.
\end{Exa}

%------------------------------------------------------------------------------
\section{Internal vs external pairings}
\label{sec:pairings}%
%\bigbreak
%------------------------------------------------------------------------------

Our definition of Green 2-functors will require us to understand how internal and external pairings of 2-functors are related. This is a very general yoga.

\begin{Hyp}
In this section we suppose $\mathbb A$ to be a (possibly very large) symmetric monoidal 2-category. 
Its tensor product and unit will  be denoted by $\otimes\colon \mathbb A\times \mathbb A\to \mathbb A$ and~$\Unit\colon 1\to \bbA$.
The examples to keep in mind are $\ADD_\kk$ and $\ADD_\kk^\ic$ as in \Cref{Prop:tensor-ADD}.
We will write $\ADDick$ to indicate either choice of these two. 
\end{Hyp}

\begin{Ter}
Given three 2-functors $\cat M, \cat N, \cat L\colon \GG^\op\to \mathbb A$, we can consider the following two composite 2-functors of two variables:
\[
\xymatrix{
\GG^\op\times \GG^\op \ar[rr]^-{\cat M\times \cat N} \ar@/_2ex/[drr]_(.4){\cat M \,\boxtimes\, \cat N\,:=} &&
 \mathbb A \times \mathbb A \ar[d]^{\otimes} \\
&& \mathbb A
}
\quad\quad\quad
\xymatrix{
\GG^\op\times \GG^\op \ar[rr]^-{\times} \ar@/_2ex/[drr]_(.4){\cat L(- \times -)\,:=} &&
 \GG^\op \ar[d]^{\cat L} \\
&& \mathbb A
}
\]
We will occasionally call $\cat M\boxtimes \cat N$ the \emph{external} tensor product of $\cat M$ and~$\cat N$. 
The \emph{internal} (or \emph{diagonal}) tensor product of $\cat M$ and~$\cat N$ is the composite $\cat M\otimes \cat N := (\cat M \boxtimes \cat N)\circ \Delta $,
where $\Delta \colon \GG^\op \to \GG^\op \times \GG^\op$ is the diagonal 2-functor $G\mapsto (G,G)$.
\end{Ter}
 
\begin{Prop} \label{Prop:tensor-pre-Green}
The internal tensor product $(\cat M, \cat N)\mapsto \cat M \otimes \cat N$ turns the 2-category $\2Fun(\GG^\op,\bbA)$ of 2-functors, pseudonatural transformations and modifications into a symmetric monoidal 2-category, where all structural equivalences are defined diagonally.
\end{Prop}

\begin{proof}
 More precisely, the composite pseudofunctors
 \[
 \xymatrix@C=17pt{
 \2Fun(\GG^\op, \bbA) \times  \2Fun(\GG^\op, \bbA) 
  \ar[r]^-{\times} & 
\2Fun(\GG^\op \times \GG^\op, \bbA \times \bbA) 
  \ar[rr]^-{\2Fun(\Delta, \otimes)} &&  
\2Fun(\GG^\op, \bbA)
 }
 \]
and $\GG^\op \to 1 \overset{\Unit}{\to} \bbA$ provide the tensor product and the unit, respectively. 
The coherent structural adjoint equivalences and 2-isomorphisms are all defined objectwise in~$\mathbb A$, in the only way that makes sense. 
We leave details to the reader.
\end{proof}

Let us record the following generality on pseudomonoids (a similar result, whose formulation we leave to the reader, holds for left and right modules):

\begin{Prop} \label{Prop:Green-lifting}
Let $\cat M\in \2Fun(\GG^\op,\bbA)$. 
To give a pseudomonoid structure on $\cat M$ (\Cref{Def:pseudo-mon}) is equivalent to choosing a lifting of the 2-functor $\cat M\colon \GG^\op\to \bbA$  to the 2-category $\PsMon(\bbA)$ of pseudomonoids in~$\bbA$ along the forgetful 2-functor $\PsMon(\bbA)\to \bbA$. 
If $\cat M$ is braided (resp.\ symmetric), the lifting is to braided (symmetric) pseudomonoids and their 1- and 2-morphisms.
\end{Prop}

\begin{proof}
Direct from the definitions. 
Concretely, and for future reference, if we are  given a pseudomonoid $\cat M$ in $\2Fun(\GG^\op,\bbA)$, its multiplication and unit are pseudonatural transformations $\odot^\cat M\colon \cat M\otimes \cat M \to \cat M$ and $\Unit^\cat M\colon \Unit \to \cat M$ whose components 
$\odot^\cat M_G\colon \cat M(G)\otimes \cat M(G)\to \cat M(G)$
and
$\Unit^\cat M_G\colon \Unit \to \cat M(G)$ 
at $G\in \Obj \GG$ are 1-cells of $\mathbb A$, 
and whose components
\begin{equation} \label{eq:2-cell-comp}
\vcenter{
\xymatrix{
\cat M(G) \otimes \cat M(G)
  \ar[r]^-{\odot^\cat M_G}
  \ar[d]_{u^*\otimes u^*}  &
\cat M(G)
  \ar[d]^{u^*}
  \ar@{}[dl]|{{}^\simeq\;\SWcell\;\odot^\cat M_u} \\
\cat M(H) \otimes \cat M(H)
  \ar[r]_-{\odot^\cat M_H}  &
\cat M(H)
}}
\quad\quad \textrm{and} \quad\quad
\vcenter{
\xymatrix{
\Unit_\mathbb A
  \ar[r]^-{\Unit^\cat M_G}
  \ar@{=}[d] & 
\cat M(G)
 \ar[d]^{u^*}
 \ar@{}[dl]|{{}^\simeq \; \SWcell \; \Unit^\cat M_u}  \\
\Unit_\mathbb A
 \ar[r]_-{\Unit^\cat M_H} &
\cat M(H) 
}}
\end{equation}
at each $u\colon H\to G$ are invertible 2-cells of~$\mathbb A$.
The 1-cell components provide the multiplication and unit of pseudomonoids $(\cat M(G), \odot^\cat M_G, \Unit^\cat M_G)$, whose coherent associativity and unitality isomorphisms $\alpha_G,\lambda_G,\rho_G$ are the $G$-components of those of the given monoid $\cat M$ (which are modifications).
The 2-cell components~\eqref{eq:2-cell-comp}, on the other hand, equip each 1-cell $u^*= \cat M(u)$ with the structure
\[ %\begin{equation} 
(\odot^\cat M_u)^{-1} \colon \odot^\cat M_H \circ (u^* \otimes u^*)  \overset{\sim}{\Longrightarrow} u^* \circ \odot^\cat M_G
\quad\quad\quad
(\Unit^\cat M_u)^{-1} \colon \Unit^\cat M_H  \overset{\sim}{\Longrightarrow} u^* \circ \Unit^\cat M_G
\] %\end{equation}
of a strong 1-morphism of pseudomonoids.
(The inversion is because, to agree with \cite{BalmerDellAmbrogio20}, we orient the 2-cell components of our pseudonatural transformations in the \emph{colax} direction, but we prefer to orient strong morphisms in the \emph{lax} direction, as is commonly done.)
The remaining verifications are straightforward.
\end{proof}

\begin{Rem} \label{Rem:ADDick-case}
If $\bbA= \ADDick$, the lifting lands in the 2-category $\PsMon(\bbA)$  of (idempotent complete) additive $\kk$-linear monoidal categories, $\kk$-linear strong tensor functors, and monoidal natural transformations; and similarly with braided or symmetric monoidal categories, functors and transformations;
\cf \Cref{Exa:monADD}.
\end{Rem}

\begin{Rem} \label{Rem:2Mack-no-tensor}
Suppose $\bbA=\ADDick$. 
Then the internal tensor product of two additive 2-functors $\cat M,\cat N\colon \GG^\op\to \bbA$ is almost never additive. 
 In particular the 2-subcategory $\biMack(\GG;\JJ) \subset \2Fun(\GG^\op,\bbA)$ of Mackey 2-functors (\cite[\S6.3]{BalmerDellAmbrogio20}) is not a tensor 2-subcategory of that of \Cref{Prop:tensor-pre-Green}. 
The reason is that, in these examples, the tensor product distributes over direct sums (\Cref{Rem:int-hom-distr}), so that
\begin{align*}
(\cat M\otimes \cat N)(G_1\sqcup G_2) 
&= \cat M(G_1 \sqcup G_2) \otimes \cat N(G_1 \sqcup G_2) \\
& \simeq \big( \cat M(G_1) \oplus \cat M(G_2) \big) \otimes \big( \cat N(G_1) \oplus \cat N(G_2) \big)
\end{align*}
is typically not equivalent to 
\begin{align*}
(\cat M\otimes \cat N)(G_1) \oplus (\cat M \otimes \cat N)(G_2) 
&= \big( \cat M(G_1) \otimes \cat N(G_1) \big) \oplus \big( \cat M(G_2) \otimes \cat N(G_2) \big)
\end{align*}
for additive $\cat M,\cat N$ and general $G_1,G_2\in \GG$, except in trivial cases.
Similarly, the tensor unit of $\2Fun(\GG^\op,\ADDick)$, \ie the constant 2-functor $G\mapsto \cat M(G)= \Unit$, is not additive because $\Unit \oplus \Unit \not\simeq \Unit$. 
The underlying problem is that the tensor structure on $\bbA$ we are using here is not the Cartesian one.
\end{Rem} 

\begin{Def} \label{Def:pairing}
Suppose now that the spannable pair $(\GG;\JJ)$ is Cartesian (Hypotheses \ref{Hyp:spannable-pair} and~\ref{Hyp:prods}).
An \emph{external pairing} from $(\cat M,\cat N)$ to~$\cat L$ is a pseudonatural transformation $\boxdot\colon \cat M\boxtimes \cat N\to \cat L(- \times -)$ of 2-functors $\GG^\op \times \GG^\op \to \mathbb A$. Together with modifications as morphisms, they form a category $\PsNat(\cat M \boxtimes \cat N , \cat L(-\times -))$. 
\end{Def}

\begin{Rem}
Concretely, an external pairing $\boxdot\colon (\cat M,\cat N)\to \cat L$ consists of a family 
\[
\boxdot = \boxdot_{G_1,G_2}\colon \cat M(G_1) \otimes \cat N(G_2) \longrightarrow \cat L(G_1\times G_2)
\]
of 1-morphisms of $\mathbb A$ for all pairs of objects $G_1,G_2\in \GG$, together with a family $\boxdot= \boxdot_{u_1,u_2}$ of invertible 2-morphisms of~$\mathbb A$ 
\[
\xymatrix{
\cat M(G_1) \otimes \cat N(G_2) \ar[d]_{u_1^* \otimes u_2^*} \ar[rr]^-{\boxdot_{G_1,G_2}} &&
 \cat L(G_1 \times G_2) \ar[d]^-{(u_1 \times u_2)^*}  \\
\cat M(H_1) \otimes \cat N(H_1) \ar[rr]_-{\boxdot_{H_1,H_2}} \ar@{}[urr]|{{}^\simeq \SWcell \; \boxdot_{u_1,u_2}} &&
 \cat L(H_1 \times H_2)
}
\]
for all pairs of 1-morphisms $u_1\colon H_1\to G_1$ and $u_2\colon H_2\to G_2$. 
This data is subject to the usual functoriality and naturality axioms (see \eg \cite[App.\,A]{BalmerDellAmbrogio20}). 
\end{Rem}

\begin{Prop} \label{Prop:pairings}
Let $(\GG;\JJ)$ be a Cartesian pair.
For any triple of 2-functors $\cat M, \cat N, \cat L\colon \GG^\op\to \mathbb A$, there is a canonical adjoint equivalence of categories
\[
\PsNat ( \cat M\otimes \cat N, \cat L) \simeq \PsNat ( \cat M \boxtimes \cat N, \cat L(- \times -)) \,.
\]
\end{Prop}

\begin{Ter} \label{Ter:pairings}
The pseudonatural transformation $\cat M\otimes \cat N\to \cat L$ corresponding to an external pairing $\boxdot\colon (\cat M,\cat N)\to \cat L$ will be called \emph{internal pairing} and will be typically denoted by~$\odot$; and vice-versa. We will simply say \emph{pairing} for either $\odot$ or the corresponding~$\boxdot$, as context will clarify.
\end{Ter}

\begin{proof}
The proof is virtually identical to that of \cite[Thm.\,7.7]{GPS14}. 
As in \emph{loc.\,cit.}, the point is that the product 2-functor $-\times- \colon \GG^\op \times \GG^\op \to \GG^\op$ is left 2-adjoint to the diagonal 2-functor $\Delta\colon \GG^\op \to \GG^\op \times \GG^\op $. 
The unit and counit are given (in $\GG$ rather than $\GG^\op$ for clarity) by the diagonals $\delta_G \colon G\to G\times G$ and the projections $(\pi_1,\pi_2)\colon (G_1\times G_2, G_1\times G_2)\to (G_1, G_2)$, which exist because $-\times -$ is the categorical product in~$\GG$.
The claimed equivalence is between transformations of the form
\[
\vcenter{
\xymatrix{
\GG^\op\times \GG^\op \ar[rr]^-{\cat M \,\boxtimes\, \cat N} \ar[d]_\times && \mathbb A \ar@{=}[d] \ar@{}[dll]|{\SWcell} \\
\GG^\op \ar[rr]_-{\cat L} && \mathbb A
}}
\quad \quad
\textrm{and}
\quad \quad
\vcenter{
\xymatrix{
\GG^\op\times \GG^\op \ar[rr]^-{\cat M\,\boxtimes\, \cat N}  && \mathbb A \ar@{=}[d] \ar@{}[dll]|{\SEcell} \\
\GG^\op \ar[rr]_-{\cat L} \ar[u]^-{\Delta} && \mathbb A
}}
\]
and therefore can be obtained by the 2-categorical mate correspondence resulting from the above-mentioned 2-adjunction.
(See \cite[Prop.\,3.5]{Lauda06} if necessary.) 
\end{proof}

\begin{Rem} \label{Rem:mate-corr-pairing}
For later reference, let us make   more explicit the mate correspondence of \Cref{Prop:pairings}. Given a pairing $\boxdot\colon (\cat M, \cat N)\to \cat L$, the composites
\[
\xymatrix{
\odot_G \colon\quad \cat M(G)\otimes \cat N(G) \ar[r]^-{\boxdot_{G,G}} & \cat L(G\times G) \ar[r]^-{\delta_G^*} & \cat L(G) 
}
\]
(where as above $\delta_G\colon G\to G\times G$ is the diagonal 1-morphism of~$G\in \Obj \GG$) and 
\[
\odot_u \colon \;
\vcenter{
\xymatrix{
\cat M(G) \otimes \cat N(G) \ar[d]_-{u^* \otimes u^*} \ar[r]^-{\boxdot_{G,G}} &
 \cat L(G\times G) \ar[r]^-{\delta_G^*} \ar[d]^{(u\times u)^*} \ar@{}[ld]|{\SWcell\;\boxdot_{u,u}} &
  \cat L(G) \ar[d]^{u^*} \\
\cat M(H) \otimes \cat N(H) \ar[r]_-{\boxdot_{H,H}} & \cat L(H\times H) \ar[r]_-{\delta_H^*} & \cat L(H) 
}
}
\]
define the components of the transformation $\odot \colon \cat M \otimes \cat N\to \cat L$.
Conversely, any pseudonatural transformation $\odot\colon \cat M \otimes \cat N\to \cat L$ can be extended to a pseudonatural $\boxdot \colon \cat M \boxtimes \cat N\to \cat L(-\times -)$, \ie a pairing,
with components given by 
\[
\boxdot_{G_1,G_2} \colon \;
\xymatrix{
\cat M(G_1) \otimes \cat N(G_2) \ar[r]^-{\pi_1^* \otimes \pi_2^*} &
 \cat M(G_1\times G_2) \otimes \cat N(G_1 \times G_2) \ar[r]^-{\odot_{G_1\times G_2}} &
  \cat L(G_1 \times G_2)
}
\]
(where as before $\pi_i \colon G_1 \times G_2\to G_i$ denotes the projection 1-morphism) and
\[
\boxdot_{u_1,u_2} \colon \!\!
\vcenter{
\xymatrix{
\cat M(G_1) \otimes \cat N(G_2) \ar[r]^-{\pi_1^* \otimes \pi_2^*} \ar[d]_{u_1^* \otimes u_2^*} &
 \cat M(G_1\times G_2) \otimes \cat N(G_1 \times G_2) \ar[r]^-{\odot_{G_1\times G_2}} \ar[d]_{(u_1\times u_2)^* \otimes (u_1\times u_2)^*} &
  \cat L(G_1 \times G_2) \ar[d]^{(u_1\times u_2)^*} \ar@{}[dl]|{\SWcell\; \odot_{u_1\times u_2}} \\
\cat M(H_1) \otimes \cat N(H_2) \ar[r]_-{\pi_1^* \otimes \pi_2^*} & 
\cat M(H_1\times H_2) \otimes \cat N(H_1 \times H_2) \ar[r]_-{\odot_{H_1\times H_2}} &
  \cat L(H_1 \times H_2)\,.
}}
\]
The assignments $\boxdot\mapsto \odot$ and $\odot \mapsto \boxdot$ can then be extended to modifications as well in the evident way, in order to define pseudofunctors.
\end{Rem}

%------------------------------------------------------------------------------
\section{The definition of Green 2-functors}
\label{sec:Green}%
%\bigbreak
%------------------------------------------------------------------------------

We now consider pairings of Mackey 2-functors. This leads us to our definition of  (plain, braided or symmetric) Green 2-functors as well as of left and right modules over them.

Let $(\GG;\JJ)$ be a Cartesian spannable pair as in Hypotheses~\ref{Hyp:spannable-pair}, \ref{Hyp:rectification} and~\ref{Hyp:prods}.
For simplicity we now specialize to $\bbA = \ADDick$, which as before denotes either of the symmetric monoidal 2-categories $\ADD_\kk$ and $\ADDic_\kk$ of \Cref{Prop:tensor-ADD}. 
Equip $\2Fun(\GG^\op,\ADDick)$ with the symmetric monoidal structure of \Cref{Prop:tensor-pre-Green}. We still write $\otimes$ for both the tensor product of $\ADDick$ and of $\2Fun(\GG^\op,\ADDick)$.

\begin{Rem} \label{Rem:adj-tensor-ADD}
For any Mackey 2-functor $\cat M$ and any two morphisms $i_1\colon H_1\to G_1$ and $i_2\colon H_2\to G_2$ in~$\JJ$, we can obtain two adjunctions 
\[
\xymatrix{
\cat M(G_1) \otimes \cat M(G_2) 
 \ar[dd]|{i_1^*\, \otimes \, i_2^*} \\ \\
\cat M(H_1) \otimes \cat M(H_2)
 \ar@/^7ex/[uu]_\dashv^{(i_1)_! \,\otimes\, (i_2)_!}
  \ar@/_7ex/[uu]^\dashv_{(i_1)_* \,\otimes\, (i_2)_*}
}
\]
from the four adjunctions $(i_k)_!\dashv i_k^* \dashv (i_k)_*$ $(k=1,2)$, by decomposing $i_1^* \otimes i_2^* = (i_1^* \otimes \Id)\circ (\Id \otimes i_2^*)$ and noticing that the operation of tensoring in $\bbA$ with a fixed object is 2-functorial and therefore preserves adjunctions. 
\end{Rem}

\begin{Prop} \label{Prop:bimorphism}
Consider a pairing $\boxdot\colon (\cat M,\cat N)\to \cat L$ (\Cref{Def:pairing}) where the three 2-functors $\cat M,\cat N,\cat L\colon \GG^\op\to \ADDick$ are all Mackey 2-functors for the Cartesian pair $(\GG;\JJ)$.
Then the following are equivalent:
\begin{enumerate}[\rm(1)]
\item \label{it:mor}
For all $(i_1\colon H_1\to G_2)$ and $(i_2\in H_2\to G_2)$ in~$\JJ$, the left mate
\[
\vcenter{
\xymatrix{
\cat M(G_1) \otimes \cat N(G_2)
 \ar@{<-}[d]_{(i_1)_! \otimes (i_2)_!}
  \ar[r]^-{\boxdot} & 
 \cat L(G_1 \times G_2)
  \ar@{<-}[d]^{(i_1\times i_2)_!}
   \ar@{}[dl]|{\NWcell \;(\boxdot^{-1})_!} \\
\cat M(H_1) \otimes \cat M(H_2)
 \ar[r]_-{\boxdot}  &
  \cat M(H_1 \times H_2)
}}
\]
and the right mate
\[
\vcenter{
\xymatrix{
\cat M(G_1) \otimes \cat N(G_2)
 \ar@{<-}[d]_{(i_1)_* \otimes (i_2)_*}
  \ar[r]^-{\boxdot} & 
 \cat L(G_1 \times G_2)
  \ar@{<-}[d]^{(i_1\times i_2)_*}
   \ar@{}[dl]|{\SEcell \;\boxdot_*} \\
\cat M(H_1) \otimes \cat M(H_2)
 \ar[r]_-{\boxdot}  &
  \cat M(H_1 \times H_2)
}}
\]
obtained from the 2-cell component 
$\boxdot=\boxdot_{i_1,i_2}\colon (i_1\times i_2)^*\circ \boxdot \overset{\sim}{\Rightarrow} \boxdot \circ (i_1^* \otimes i_2^*)$
 are mutually inverse (\cf \Cref{Rem:adj-tensor-ADD} for the adjoints on the left-hand side; also, those on the right make sense because $\JJ$ is closed under Cartesian products).
\item \label{it:left-mor}
For all $i_1,i_2\in \JJ$, the left mate $(\boxdot_{i_1,i_2}^{-1})_!$ of part~\eqref{it:mor} is invertible.
\item \label{it:right-mor} 
For all $i_1,i_2\in \JJ$, the right mate $(\boxdot_{i_1,i_2})_*$ of part~\eqref{it:mor} is invertible.
\item \label{it:left-special}
As in part \eqref{it:left-mor}, but only for pairs $(i_1,i_2)$ where $i_1$ or $i_2$ is an identity 1-cell.
\item \label{it:right-special}
As in part \eqref{it:right-mor}, but only for pairs $(i_1,i_2)$ where $i_1$ or $i_2$ is an identity 1-cell.
\end{enumerate}
\end{Prop}

\begin{proof}
For a fixed $G\in \GG$, the 2-functors $\cat M(-) \otimes \cat N(G)$ and $\cat M(G)\otimes \cat N(-)$ again satisfy the axioms of (rectified) Mackey 2-functors for~$(\GG;\JJ)$, because tensoring in $\ADDick$ with a fixed object preserves direct sums (\Cref{Rem:int-hom-distr}), adjunctions, and isomorphisms.
It is also easy to see that $\cat L(- \times G)$ and $\cat L(G \times -)$ are again Mackey 2-functors (for the Mackey formulas, use that $-\times G$ and $G\times -$ preserve Mackey squares and also, by hypothesis, 1-cells of $\JJ$).
By writing $i_1 \times i_2$ as a composite $(i_1 \times \Id) \circ (\Id \times i_2)$ and by the compatibility of mates with pasting of squares, it follows easily from the above observations that the equivalences between the five properties reduce to the analog of \eqref{it:mor}$\Leftrightarrow$\eqref{it:left-mor}$\Leftrightarrow$\eqref{it:right-mor} for a pseudonatural transformation between two Mackey 2-functors. The latter is proved in \cite[Prop.\,6.3.1]{BalmerDellAmbrogio20}.
\end{proof}

\begin{Def}[Bimorphism] \label{Def:bimorphism}
We will say that a pairing $\boxdot\colon (\cat M,\cat N)\to \cat L$ of Mackey 2-functors is a $\JJ$-\emph{bimorphism}, or just \emph{bimorphism}, or a \emph{morphism of two variables}, if it satisfies the equivalent conditions of \Cref{Prop:bimorphism}.
\end{Def}

\begin{Def}[Green 2-functor]
 \label{Def:2Green}
We define a \emph{Green 2-functor} for $(\GG;\JJ)$ to be a pseudomonoid object $\cat M= (\cat M, \odot^\cat M, \Unit^\cat M)$  (\Cref{Def:pseudo-mon}~-- here we mostly suppress the associators and unitors from notations) in the symmetric monoidal 2-category $\2Fun(\GG^\op,\ADDick)$ (\Cref{Prop:tensor-pre-Green}), satisfying two properties:
\begin{enumerate}[\rm(1)]
\item The underlying 2-functor $\cat M\colon \GG^\op\to \ADDick$ is a Mackey 2-functor for $(\GG;\JJ)$.
\item The external pairing $\boxdot^\cat M \colon \cat M \boxtimes \cat M\to \cat M(- \times -)$ associated with the (internal) multiplication $\odot^\cat M \colon \cat M\otimes \cat M\to \cat M$, as in \Cref{Prop:pairings}, is a $\JJ$-bimorphism in the sense of \Cref{Def:bimorphism}.
\end{enumerate}
Thus in plain English, a Green 2-functor is a Mackey 2-functor which is also a pseudomonoid in 2-functors such that its multiplication preserves inductions in both variables.
We will speak of a \emph{braided} or  \emph{symmetric} Green 2-functor if it is braided or symmetric as a pseudomonoid.
\end{Def}

\begin{Rem} \label{Rem:pointwise-structure}
By \Cref{Prop:Green-lifting} and \Cref{Rem:ADDick-case}, the data of a Green 2-functor amounts to that of a 2-functor $\cat M\colon \GG^\op\to \PsMon(\ADDick)$ with values in monoidal (idempotent complete) additive $\kk$-linear categories, strong monoidal $\kk$-linear functors and monoidal transformations; after forgetting to $\ADDick$, each $i^*=\cat M(i)$ ($i\in \JJ$) admits a left-and-right adjoint $i_!=i_*$ satisfying various nice properties.
Similarly, a braided (or symmetric) Green 2-functor corresponds to a 2-functor with values in braided (symmetric) $\kk$-linear monoidal categories, braided strong monoidal $\kk$-linear functors and braided monoidal transformations.
\end{Rem}

\begin{Rem} \label{Rem:natural-formulation}
In the definition of a Green 2-functor~$\cat M$, the use of the external pairing $\boxdot^\cat M$ arises quite naturally because it is needed for the notion of bimorphism of Mackey 2-functors, which is a clear analog of the (one~variable) morphisms used in \cite{BalmerDellAmbrogio20}. 
It is \emph{a~priori} not obvious how to translate the bimorphism condition in terms of the \emph{internal} pairing, \ie of the lifting to monoidal categories of \Cref{Prop:Green-lifting}. 
In the next section we will see that this translation is precisely (a specific form of) the left and right projection formulas. The resulting reformulation of Green 2-functors (\Cref{Def:quasi-Green}) also make sense for non-Cartesian spannable pairs. 
\end{Rem}

\begin{Exa}  \label{Exa:illustration-tensor-preserves-ind}
Let us illustrate \Cref{Def:2Green} and the bimorphism condition concretely using the ur-example of a Green 2-functor~$\cat M$, namely that of linear representations over a field~$\kk$ (see \Cref{Exa:rep-theory}). 
Take $(\GG;\JJ)=(\gpd; \gpdf)$ and $\bbA= \ADD^{\ic}_\kk$. Then $\cat M(G)$ is the functor category $\kk G\MMod=(\kk \MMod)^G$ and the 2-functoriality $G\mapsto \cat M(G)$ (``restriction'') is precomposition with functors and natural transformations in~$\gpd$. Focussing attention to groups~$G$, for familiarity, this is restriction along homomorphisms and conjugations thereof.  Induction along a subgroup inclusion $i\colon H\hookrightarrow G$ is computed by the familiar Kan extension formulas: $i_!=\kk G\otimes_{\kk H}(-) \cong \Hom_{\kk H}(\kk G, -) = i_*$. 
The internal (symmetric) tensor product $\odot$ on $\cat M(G)$ is the tensor product as $\kk$-vector spaces equipped with diagonal $G$-action; the corresponding external product $M \boxdot M'$ of $M\in \cat M(G)$ and $M'\in \cat M(G')$ is again $M\otimes_\kk M'$, this time equipped with the induced $G\times G'$-action.
As in \Cref{Rem:pointwise-structure}, to get a Green 2-functor we must still check that the external tensor product ``preserves induction in both variables''. 
For the first variable (say), by computing the (say, left) mate as in \Cref{Prop:bimorphism}\,\eqref{it:mor} for $i_1$ the inclusion $H\hookrightarrow G$ of a subgroup and $i_2=\Id_{G'}$ the identity of another finite group~$G'$, this translates as the requirement that the canonical morphism
\[
\kk(G \times G') \otimes_{\kk (H \times G')} (N \boxdot M)
\overset{\sim}{\longrightarrow} 
\big ( \kk G \otimes_{\kk H} N \big ) \boxdot M 
\]
of $\kk(G\times G')$-modules is invertible for all $N\in \cat M(H)$ and $M \in \cat M(G)$ (which it is!). 
\end{Exa}

\begin{Rem} \label{Rem:unit-cond}
\Cref{Def:2Green} imposes no extra condition on the unit $\Unit^\cat M\colon \Unit \to \cat M$ of a Green 2-functor. In particular, for $(i\colon H\to G)\in \JJ$, we emphatically do \emph{not} ask that the right mate $\Unit_G^\cat G \Rightarrow i_*\circ \Unit_H^\cat M$ and left mate  $i_!\circ \Unit_H^\cat M\Rightarrow \Unit_G^\cat M$ of its component $\Unit^\cat M_i\colon i^*\circ \Unit^\cat M_G\overset{\sim}{\Rightarrow} \Unit^\cat M_H$ and its inverse~$(\Unit^\cat M_i)^{-1}$, respectively, be invertible. Identifying the functor $\Unit_G=\Unit_G^\cat M\colon \Unit\to \cat M(G)$ with the object of $\cat M(G)$ that classifies it, as in \Cref{Rem:unit-objects}, this means that we do not ask that the two maps of~$\cat M(G)$
\[ 
\Unit_G \to i_*(\Unit_H)
\quad \textrm{ and } \quad 
i_!(\Unit_H)\to \Unit_G
\]
 be invertible.
Indeed, they are virtually never so in the exemples; \cf \Cref{sec:Frob}.
\end{Rem}

We can similarly define left and right actions of a Green 2-functor:

\begin{Def} \label{Def:actions}
Let $\cat M=(\cat M, \odot^\cat M, \Unit^\cat M)$ be a Green 2-functor. 
A \emph{left module} over $\cat M$ is a Mackey 2-functor $\cat N$ equipped with a left pseudoaction $\odot^\cat N\colon \cat M\otimes \cat N\to \cat N$  in $\2Fun(\GG^\op, \bbA)$ (\Cref{Def:pseudo-mon}) 
which commutes with induction in both variables, \ie such that the associated pairing $\boxdot^\cat N\colon (\cat M,\cat N)\to \cat N$ is a $\JJ$-bimorphism. Right modules are defined similarly. 
\end{Def}

\begin{Rem} \label{Rem:Green-2fun-internal}
In the classical theory of Mackey functors, one can define Green functors as monoids in the tensor category of Mackey functors (see \Cref{Prop:1-Green-equivs}), and similarly for actions. 
The analogous fully internal definition of Green 2-functor should also work, at least when $(\GG;\JJ)=(\gpd;\gpdf)$ or $(\gpdG;\gpdG)$, in which case we know how to construct a manageable bicategory $\Spanhat$ of ``Mackey 2-motives'' (see \cite[Ch.\,6-7]{BalmerDellAmbrogio20}). See \Cref{Rem:misalignment} for more on this.
\end{Rem}

%------------------------------------------------------------------------------
\section{The projection formulas}
\label{sec:proj}%
%\bigbreak
%------------------------------------------------------------------------------

Let $\cat M, \cat N, \cat L$ be three Mackey 2-functors for $(\GG;\JJ)$ with values in $\bbA = \ADDick$ and suppose we are given a bimorphism $\odot\colon  (\cat M, \cat N)\to \cat L$, as in \Cref{sec:Green}. 
Then in particular, for every $(i\colon H\to G)\in \JJ$ the restriction functor $i^*\colon \cat L(G)\to \cat L(H)$ comes equipped with natural isomorphisms (for $X_1 \in \cat M(G),X_2\in \cat N(G)$)
\[ \odot_i \colon i^*(X_1\odot_G X_2) \overset{\sim}{\longrightarrow}  i^*(X_1)\odot_H i^* (X_2) \] 
as well as two adjunctions $(i_!\dashv i^*, \leta, \leps)$ and $(i^*\dashv i_*, \reta, \reps)$.

\begin{Def}[Projection maps] 
\label{Def:Frobs}
Given this data, the \emph{right projection maps} are the following two composite natural transformations:
\[
\xymatrix{
X \odot_G i_*Y
  \ar[d]_{\reta} \ar@{..>}[r]^-{\rproj^1_{X,Y}} & 
i_*(i^*X \odot_H Y ) \\
i_*i^* (X \odot_G i_*Y )
  \ar[r]^-\sim_-{i_*(\odot_i)} & 
i_*(i^* X \odot_H i^*i_* Y )
  \ar[u]^{i_*(i^*X \odot_H \reps )}
}
\;\;
\xymatrix{
i_*Y \odot_G X
   \ar[d]_{\reta} \ar@{..>}[r]^-{\rproj^2_{Y,X}}  &
i_*(Y \odot_H  i^*X  )  \\
i_*i^* (i_*Y \odot_G  X  ) 
   \ar[r]^-\sim_-{i_*(\odot_i)}  &
i_*(i^*i_* Y \odot_H  i^* X  )
  \ar[u]^{i_*( \reps\, \odot_H i^*X )}
}
\]
The \emph{left projection maps} are defined dually, as follows:
\[
\xymatrix{
X \odot_G i_!Y
  \ar@{<-}[d]_{\leps} \ar@{<..}[r]^-{\lproj^1_{X,Y}} & 
i_!(i^*X \odot_H Y ) \\
i_!i^* (X \odot_G i_!Y )
  \ar@{<-}[r]^-\sim_-{i_!(\odot_i^{-1})} & 
i_!(i^* X \odot_H i^*i_! Y )
  \ar@{<-}[u]^{i_!(i^*X \odot_H \leta )}
}
\;\;
\xymatrix{
i_!Y \odot_G X
   \ar@{<-}[d]_{\leps}
    \ar@{<..}[r]^-{\lproj^2_{Y,X}}  &
i_!(Y \odot_H  i^*X  )  \\
i_!i^* (i_!Y \odot_G  X  ) 
   \ar@{<-}[r]^-\sim_-{i_!(\odot_i^{-1})}  &
i_!(i^*i_! Y \odot_H  i^* X  )
  \ar@{<-}[u]^{i_!( \leta\, \odot_H i^*X )}
}
\]
Note that the variables are $X\in \cat M(G)$ and $Y\in \cat N(H)$ for the two left-hand side diagrams, but $X\in \cat N(G)$ and $ Y\in \cat M(H)$ for the two right-hand side ones. 
\end{Def}

\begin{Rem} 
In our chosen notations $\lproj^1, \lproj^2, \rproj^1, \rproj^2$ for \Cref{Def:Frobs}, ``left'' and ``right'' refer to whether we make use of the left or right adjunction, while ``1'' and ``2'' refer to whether it is the first or second variable which moves past induction $i_*=i_!$. Other conventions are in use in similar situations; \cf \Cref{sec:half-braiding}. 
\end{Rem}

\begin{Rem} 
The definition of left or right projection maps is exceedingly general: it suffices  to have, in a monoidal 2-category, a (strong) morphism of pseudomonoids (\Cref{Def:pseudo-mon}) admitting a left or right adjoint in the underlying plain 2-category.  
\end{Rem}

\begin{Thm}[The projection formulas] 
\label{Thm:proj-formula}
For any bimorphism  $\odot \colon (\cat M , \cat N)\to \cat L$ of Mackey 2-functors $\cat M,\cat N, \cat L$ (\Cref{Def:bimorphism}) and any $(i\colon H\to G)\in \JJ$, the left and right projection maps (\Cref{Def:Frobs}) are mutually inverse
\[
\xymatrix{
 X \odot_G i_*(Y)
 \ar@<.7ex>[r]^-{\rproj^1}
 \ar@{}[r]|-\sim &
i_*(i^*X \odot_H  Y)
  \ar@<.7ex>[l]^-{\lproj^1}
}
\quad\quad
\xymatrix{
i_*(Y) \odot_G X
 \ar@<.7ex>[r]^-{\rproj^2}
 \ar@{}[r]|-\sim &
i_*(Y \odot_H i^* X)
  \ar@<.7ex>[l]^-{\lproj^2}
}
\]
for all $X\in \MM(G)$ and all $Y \in \cat N(H)$, resp.\ all $X\in \cat N(G)$ and all $Y \in \MM(H)$.
\end{Thm}

\begin{Exa}
Most notably, this holds for the multiplication $(\cat M, \cat M)\to \cat M$ of a Green 2-functor~$\cat M$, or more generally for the action $(\cat M, \cat N)\to \cat N$ or $(\cat N, \cat M)\to \cat N$ of a Green 2-functor~$\cat M$ on a left or right $\cat M$-module~$\cat N$ (Definitions \ref{Def:2Green} and~\ref{Def:actions}).
\end{Exa}

\begin{Rem} \label{Rem:proj-braid}
If $\cat M$ is a braided Green 2-functor, $\cat M(G)$ and $\cat M(H)$ are braided tensor categories with braidings $\beta^G$ and~$\beta^H$, say, and the monoidal structure $\odot^{-1}_i$ of $i^*$ is compatible with them. 
It follows easily that the first and second projection maps are braid-conjugate:
$\lproj^2  \circ i_!(\beta^H) = \beta^G \circ \lproj^1$ and $\rproj^2 \circ \beta^G = i_*(\beta^H) \circ \rproj^1$.
\end{Rem}

In order to prove the theorem, we first need an elementary observation:

\begin{Lem} 
\label{Lem:Mackey-for-proj}
For any 1-cell $u\colon H\to G$ in~$\GG$, the following commutative square
\begin{equation} \label{eq:Mackey-for-proj}
\vcenter{
\xymatrix@!C@R=12pt@C=10pt{
& H \ar[dl]_-{(\Id, u)} \ar[dr]^-u & \\
H \times G \ar[dr]_-{u \times \Id} & & G \ar[dl]^-{\delta_G} \\
& G\times G&
}}
\end{equation}
is a Mackey square of~$\GG$.
\end{Lem}

\begin{proof}
For any $T\in \Obj \GG$, we may apply the product-preserving 2-functor $\GG(T,-)$ to the square~\eqref{eq:Mackey-for-proj} in order to obtain a square of the same form---with the functor $\GG(T,u)\colon \GG(T,H)\to\GG(T,G)$ instead of $u\colon H\to G$ etc.---in the 2-category $\textsf{Gpd}$ of all (non-necessarily finite) groupoids. But  in $\mathsf{Gpd}$, any square of the form \eqref{eq:Mackey-for-proj} for any functor $u\colon H\to G$ is a Mackey square; indeed, it is an easy exercise to verify that the comparison functor $\langle (\Id,u), u, \id \rangle$ to the iso-comma square (in  the notations of \cite[\S2.1]{BalmerDellAmbrogio20}) is an equivalence $H \overset{\sim}{\to} ((u \times\Id)/\delta_G)$.
As $T\in \Obj \GG$ was arbitrary, we can now conclude as in \cite[Prop.\,2.1.11, Rmk.\,2.1.6]{BalmerDellAmbrogio20} that the original square is a Mackey square of~$\GG$.
\end{proof}

\begin{proof}[Proof of \Cref{Thm:proj-formula}]
We only prove the second formula, $\lproj^2=(\rproj^2)^{-1}$, leaving the symmetrical proof of the first one to the reader. 
Note that the two sides of the formula car be rewritten as
%\[ 
$i_*( Y \odot_H i^* X) = i_* \circ \odot_H \circ ( \Id_{\cat M(H)} \otimes i^*) (Y,X) 
$ %\]
and
%\[ 
$i_* (Y) \odot_G X = \odot_G \circ (i_* \otimes \Id_{\cat N(G)} ) (Y,X) 
$ %\]
respectively. This lets us abstract away $X$ and $Y$ from all notations, which we will do from now on. 

By \Cref{Def:Frobs} and \Cref{Rem:mate-corr-pairing}, the natural transformation $\rproj^2$ is equal to the following pasting in~$\bbA$ involving the external product~$\boxdot$ associated with~$\odot$:
\[
\xymatrix{
\cat M(H) \otimes \cat N(G)
 \ar@{}[rd]|{\varepsilon \otimes \id \;\SWcell}
 \ar[r]^-{i_*\otimes \Id}
  \ar@/_4ex/[dr]_{\Id \otimes i^*} &
\cat M(G) \otimes \cat N(G) 
\ar@/^6ex/[rr]^-{\odot_G}
 \ar[r]^-{\boxdot}
  \ar[d]|{i^*\otimes i^*}
   \ar@{}[dr]|{\SWcell \, \boxdot_{i,i}} & 
 \cat L(G\times G) 
  \ar[d]|{\;\;(i\times i)^*} 
  \ar[r]^(.5){\delta_G^*} &
   \cat L(G)
       \ar@{}[dr]|{\SWcell  \eta}
         \ar[d]_{i^*}
          \ar@{=}@/^4ex/[dr] & \\
         &
\cat M(H) \otimes \cat N(H)
\ar@/_6ex/[rr]_-{\odot_H}
 \ar[r]_-{\boxdot}
& 
     \cat L(H\times H)
       \ar[r]_(.5){\delta_H^*} &
         \cat L(H)
          \ar[r]_-{i_*}
      &
            \cat L(G)
}
\]
This can be further decomposed as follows in order to separate the two variables, by the functoriality property of the pseudonatural transformation~$\boxdot$:
\[
\xymatrix{
\cat M(H) \otimes \cat N(G)
 \ar@{}[rd]|{\varepsilon \otimes \id \;\SWcell}
 \ar[r]^-{i_*\otimes \Id}
  \ar@{=}@/_4ex/[dr] &
\cat M(G) \otimes \cat N(G) 
 \ar[r]^-{\boxdot}
  \ar[d]|{i^*\otimes \Id}
   \ar@{}[dr]|{\SWcell \, \boxdot_{i,\Id}} & 
 \cat L(G\times G) 
  \ar[d]|{\;\;(i\times \Id)^*} 
  \ar[r]^-{\delta_G^*} &
   \cat L(G)
       \ar@{}[dr]|{\SWcell  \eta}
         \ar[d]_{i^*}
          \ar@{=}@/^4ex/[dr]
           \ar@{}[dl]|{\SWcell \id} & \\
         &
\cat M(H) \otimes \cat N(G)
 \ar[r]
   \ar[d]_{\Id \otimes i^*}
    \ar@{}[dr]|{ \SWcell \, \boxdot_{\Id, i}} & 
     \cat L(H\times G)
      \ar[d]|{(\Id \times i)^*}
       \ar[r]_-{(\Id,i)^*} &
         \cat L(H)
          \ar[r]_-{i_*}
           \ar@{=}[d] &
            \cat L(G)
         \\
&
\cat M(H) \otimes \cat N(H) 
 \ar[r]_-{\boxdot} & 
  \cat L(H \times H)
   \ar[r]_-{\delta_H^*} &
     \cat L(H) & 
}
\]
By the compatibility of mates with horizontal composition of squares (\ie insert a unit-counit relation for $(i\times \Id)^*\dashv (i\times \Id)_*$), this is equal to the following pasting, where the two squares at the top are the right mates of $\boxdot_{i,\Id}$ and $\id$, respectively:
\[
\xymatrix{
\cat M(G) \otimes \cat N(G)  
 \ar[rr]^-{\boxdot_{G,G}}
   \ar@{}[drr]|{{}^\simeq \SEcell} & & 
 \cat L(G\times G) 
  \ar[r]^-{\delta_G^*} 
     \ar@{}[dr]|{{}^\simeq \SEcell}
  & \cat L(G)
         & \\
\cat M(H) \otimes \cat N(G)
 \ar[rr]^-{\boxdot_{H,G}}
  \ar[u]^{i_*\otimes \Id}
   \ar[d]_{\Id \otimes i^*}
    \ar@{}[drr]|{{}^\simeq \SWcell} & &
     \cat L(H\times G)
      \ar[d]_{(\Id \times i)^*}
       \ar[r]_-{(\Id,i)^*}
        \ar[u]^{(i \times \Id)_*} &
         \cat L(H)
          \ar[u]_{i_*}
           \ar@{=}[d] &
         \\
\cat M(H) \otimes \cat N(H) 
 \ar[rr]_-{\boxdot_{H,H}} & &
  \cat L(H \times H)
   \ar[r]_-{\delta_H^*} &
     \cat L(H) &
}
\]
As indicated, both mates at the top are invertible, the top-left one by the compatibility of $\boxdot$ with inductions in the first variable, the top-right one by the Mackey formula for the Mackey 2-functor~$\cat L$ at the Mackey square of \Cref{Lem:Mackey-for-proj} (for $u=i$).
This shows that the second right projection map $\rproj^2$ is invertible, because it decomposes as a pasting of natural isomorphisms.

It remains to verify that the dually defined $\lproj^2$ provides the inverse isomorphism. 
Since $(-)_! = (-)_*$ as pseudofunctors  on~$\JJ$ (\Cref{Rem:rectification}), this is now immediate from \Cref{Prop:bimorphism}\,\eqref{it:mor} applied to the bimorphism~$\boxdot$ together with the strict Mackey formula~\eqref{eq:strict-MF} applied to the square~\eqref{eq:Mackey-for-proj}.
\end{proof}

\begin{Rem} 
It appears that all the various projection formulas (a.k.a.\ push-pull formulas) in geometry, topology and representation theory should arise from a similar argument applied to a suitable 2-functor equipped with a compatible pairing. 
\end{Rem}

\begin{Rem} \label{Rem:mon-der-proj}
Note however that the above proof does \emph{not} yield a general projection formula in the context of monoidal derivators.
The argument there fails because the square \eqref{eq:Mackey-for-proj} is not equivalent to a comma square in $\Cat$ if $H$ and $G$ are not groupoids, even for $u\colon H\to G$ faithful. Hence the base-change formula of derivators need not hold for it.
See \cite[Warning 7.21]{GPS14} for a simple counterexample. 
 \end{Rem}

As it turns out, in the context of Mackey 2-functors the converse of the projection theorem is also true:

\begin{Thm} \label{Thm:proj-implies-bimorphism}
Suppose that $\odot \colon (\cat M , \cat N)\to \cat L$ is any pairing of Mackey 2-functors $\cat M,\cat N, \cat L$ for $(\GG;\JJ)$. If $\odot$ satisfies the projection formulas for every $(i\colon H\to G)\in \JJ$, as in the conclusion of \Cref{Thm:proj-formula}, then it is a $\JJ$-bimorphism. 
Together with \Cref{Thm:proj-formula}, this means the projection formulas are precisely the translation in terms of the internal~$\odot$ of the bimorphism condition for the corresponding external~$\boxdot$.
\end{Thm}

\begin{proof}
To prove that the associated external pairing~$\boxdot$ is a bimorphism, we check it satisfies condition \eqref{it:right-mor} of \Cref{Prop:bimorphism}. 
Thus for $i\colon G'\to G$ and $j\colon H'\to H$ in~$\JJ$, we must show the invertibility of the following right mate:
\[
\xymatrix{
\cat M(G') \otimes \cat N(H')
  \ar[r]^-{\boxdot_{G',H'}}
  \ar[d]_{i_* \otimes j_*}
  \ar@{}[dr]|{(\boxdot_{i,j})_*\;\NEcell \;\;} & 
\cat L(G'\times H') 
  \ar[d]^{(i\times j)_*} \\
\cat M(G) \otimes \cat N(H)
 \ar[r]_-{\boxdot_{G,H}} &
\cat L(G\times H)
}
\]
Similarly to the proof of \Cref{Thm:proj-formula}, it is easy (!) to see that this 2-cell decomposes as in the following pasting diagram (where we omit the symbols $\otimes$ and $\times$ between objects in order to save space):
\[
\xymatrix@!C=16pt@R=34pt{
{ \cat M(G') \cat N( H') }
  \ar[d]_{\Id \,\otimes \, \Id}
  \ar[rrrrr]^-{\pr^*_1 \,\otimes\, \pr_2^*} &&&&&
\cat  M(G' H')  \cat N(G' H')
  \ar[drr]^-{\odot_{G' \times H'}} && \\
\cat M(G')  \cat N(H')
  \ar@{}[drrr]|{\id\otimes \id_*\;\NEcell \;\;}
  \ar[d]_{\Id \, \otimes \, j_*}
  \ar[rrr]^-{\pr^*_1 \,\otimes\, \pr_2^*} &&& 
\cat  M(G' H) \cat  N( G'  H')
  \ar@{}[rrrr]|{\rproj^1\NEcell \;\;}
  \ar[urr]^-{(\Id \times j)^* \,\otimes\, \Id \;\;}
  \ar[dr]^{\Id \,\otimes \, (\Id \times j)_*} &&&& 
\cat  L(G' H')
  \ar[dl]_-{(\Id \times j)_*}
  \ar@/^4ex/[dd]^-{(i \times j)_*} \\
\cat  M(G')  \cat N(H)
   \ar[d]_-{\Id \,\otimes \, \Id}
   \ar[rrr]^-{\pr^*_1 \,\otimes\, \pr_2^*} &&&& 
\cat  M(G' H)  \cat N(G' H)
   \ar[rr]^-{\odot_{G'\times H}}  && 
\cat  L(G' H)
  \ar@{}[r]|(1){\id_*\,\Ecell \;\;}
   \ar[dr]_-{(i \times \Id)_*} & \\
\cat  M(G') \cat N(H)
  \ar@{}[drrr]|{\id_*\otimes\id\;\NEcell \;\;}
   \ar[d]_{i_* \,\otimes \, \Id}
   \ar[rrr]^-{\pr^*_1 \,\otimes\, \pr_2^*}  &&& 
\cat  M(G'H) \cat N(GH)
  \ar@{}[rrrr]^{\rproj^2 \, \Ncell \;\;}
   \ar[ur]^{\Id \, \otimes \, (i \times \Id)^*}
   \ar[drr]^{\;\; \; (i \times \Id )_* \, \otimes \, \Id}
&&&& 
\cat L(GH) \\
{\cat M(G) \cat N(H)}
   \ar[rrrrr]^-{\pr^*_1 \,\otimes\, \pr_2^*}   &&&&& 
\cat M(GH) \cat N(GH)
  \ar[urr]_-{\odot_{G\times H}} &&
}
\]
We claim each of the 2-cells appearing above is invertible. 
Indeed, the two pentagons are the right projection maps $\rproj^1$ (for~$\Id\times j\in \JJ$) and $\rproj^2$ (for~$i\times \Id\in \JJ$), which are invertible by hypothesis.
The two unmarked trapezes on the left are commutative by the strict 2-functoriality of~$(-)^*$. 
The right mate~$\id_*$ in the right triangle is the pseudo\-functorial identification for~$(-)_*$, which is invertible (\cite[A.2.10]{BalmerDellAmbrogio21}).
Finally, the two trapezes marked $\id \otimes \id_*$ and $\id_*\otimes \id$ are invertible by the (strict) Mackey formula for the commutative squares
\begin{equation} \label{eq:squares-quasi-Green}
\vcenter{
\xymatrix@!C=10pt@R=10pt{
& G'\times H' \ar[ld]_-{\pr_2} \ar[rd]^{\Id \times j}  & \\
H' \ar[rd]_-{j} &  & G'\times H \ar[ld]^-{\pr_2} \\
& H &
}}
\quad\quad \textrm{ and } \quad\quad
\vcenter{
\xymatrix@!C=10pt@R=10pt{
& G'\times H \ar[ld]_-{\pr_1} \ar[rd]^{i \times \Id}  & \\
G' \ar[rd]_-{i} &  & G\times H \ar[ld]^-{\pr_1} \\
& G &
}}
\end{equation}
respectively, which can be checked to be Mackey squares in any Cartesian (2,1)-category $\GG$ by reasoning precisely as for \Cref{Lem:Mackey-for-proj}.
This concludes the proof.
\end{proof}

Thanks to \Cref{Thm:proj-implies-bimorphism} we now have an equivalent definition of Green 2-functor which makes sense beyond Cartesian pairs:

\begin{Def} \label{Def:quasi-Green}
Let $(\GG;\JJ)$ be any spannable pair, non necessarily Cartesian~(!).
A \emph{(braided, symmetric) Green 2-functor} for  $(\GG;\JJ)$ is a (braided, symmetric) pseudo\-monoid object $\cat M= (\cat M, \odot^\cat M, \Unit^\cat M)$ in  $\2Fun(\GG^\op,\ADDick)$, whose underlying 2-functor $\cat M\colon \GG^\op\to \ADDick$ is a Mackey 2-functor for $(\GG;\JJ)$ and whose internal tensor products satisfy the projection formulas  for all~$i\in \JJ$.  
That is, more precisely, for all $i\in \JJ$ the four projection maps of \Cref{Def:Frobs} satisfy $\lproj^1 = (\rproj^1)^{-1}$ and $\lproj^2 = (\rproj^2)^{-1}$; or equivalently, $\lproj^1$ and $\lproj^2$ are invertible; or $\rproj^1$ and $\rproj^2$ are.
\end{Def}

\begin{Rem} \label{Rem:extra-axiomatization}
Next, in Sections \ref{sec:Frob}-\ref{sec:monadicity} we will detail some consequences of the projection formula. 
These results concern a Green 2-functor applied to a \emph{fixed} 1-morphism $(i\colon H\to G)\in \JJ$, and the astute reader will notice that everything remains true if we only assume given: a strong monoidal functor $i^*\colon \cat M(G)\to \cat M(H)$ between monoidal categories, together with a functor $i_!=i_*\colon \cat M(H)\to \cat M(G)$ and two adjunctions $(i_!,i^*,\leta,\leps)$ and $(i^*,i_*,\reta, \reps)$, such that $\lproj^1=(\rproj^1)^{-1}$ and $\lproj^2 =(\rproj^2)^{-1}$ (with the four projection maps defined as in \Cref{Def:Frobs}) and such that the special Frobenius relation $\reps \circ \leta = \id$ holds as in~\eqref{eq:special-Frob}.
\end{Rem}

\begin{Rem} \label{Rem:Beren}
Working unbeknownst to us and under very different hypotheses (even as formulated in \Cref{Rem:extra-axiomatization}), Sanders \cite{Sanders21pp} has obtained similar results; in particular, as in our \Cref{Thm:Frobenius} he produced a special Frobenius structure on $i_*(\Unit)$  from any suitable monoidal adjunction~$i^*\dashv i_*$.
Amusingly, although almost orthogonal, our two theorems seems to cover many of the same examples.
\end{Rem}

%------------------------------------------------------------------------------
\section{Induced Frobenius algebras}
\label{sec:Frob}%
%\bigbreak
%------------------------------------------------------------------------------

The goal of this section is to show that for any Green 2-functor $\cat M$ and every $i\colon H\to G$ in~$\JJ$, the induced object $A(i):= i_!(\Unit_H)=i_*(\Unit_H)$ of the monoidal category $\cat M(G)$ inherits a special Frobenius structure.
As we will see in \Cref{sec:monadicity}, this improves and conceptually clarifies the results of \cite{BalmerDellAmbrogioSanders15}.

\begin{Rec}[Frobenius structures] \label{Rec:Frob}
In any monoidal category, a \emph{Frobenius structure}, also called \emph{Frobenius monoid} or \emph{Frobenius algebra}\footnote{The traditional name notwithstanding, the notion of a \emph{Frobenius algebra/monoid} is self-dual. To underline this, we will occasionally prefer the term \emph{Frobenius structure} or \emph{Frobenius object}.},
is an object carrying both a monoid structure and a comonoid one, in such a way that the comultiplication and counit are morphisms of left and right modules~-- or equivalently, such that multiplication and unit are morphisms of left and right comodules. The condition on the unit/counit is actually redundant; indeed, Frobenius structures admit numerous equivalent axiomatizations. 
In a braided monoidal category, a Frobenius structure is \emph{commutative} if it is commutative as a monoid (or equivalently, if it is cocommutative as a comonoid), and is \emph{special} if the comultiplication is a section for the multiplication. See \eg \cite{Street04} \cite{Kock04} \cite[\S5]{HeunenVicary19}.
\end{Rec}

\begin{Rec}[Frobenius monads] \label{Exa:Frob-monad}
Recall that a \emph{monad} on a category $\cat C$ is just a monoid in the endofunctor category $\End(\cat C)$, the monoidal structure of the latter being composition of functors. 
A \emph{comonad} is a monad in $\End(\cat C)^\op$, \ie an endofunctor of $\cat C$ equipped with a coassociative and counital comultiplication. 
A \emph{Frobenius monad} 
is a Frobenius structure in $\End(\cat C)$. Unwinding definitions, it consists of a monad $(\bbM, m, u)$ together with a comonad $(\bbM,d,e)$ on the same functor~$\bbM\colon \cat C\to \cat C$, with the property that the following (self-dual) diagram commutes:
\begin{equation} \label{eq:Frob-monad}
\vcenter{
\xymatrix{
& \bbM^2 \ar[ld]_{d\, \bbM} \ar[d]_{m} \ar[dr]^{\bbM\, d} & \\
\bbM^3 \ar[dr]_{\bbM\, m} & \bbM \ar[d]_d & \bbM^3 \ar[dl]^{m \, \bbM} \\
& \bbM^2 & 
}}
\end{equation}
Next we recall some well-known general facts on (co)monads and (co)monoids.
\end{Rec}

\begin{Prop} \label{Prop:Frob-co-monad}
Consider two consecutive adjunctions:
\[
\xymatrix{
\cat C
 \ar[d]|{i^*} \\
\cat D
 \ar@/^3ex/[u]^{i_!}_{\dashv}
 \ar@/_3ex/[u]_{i_*}^{\dashv}
}
\]
As before, we write $\reta$ and~$\reps$ for the unit and counit of the adjunction $i^*\dashv i_*$, and $\leta$ and~$\leps$ for those of $i_!\dashv i^*$.
Then:
\begin{enumerate}[\rm(1)]
\item \label{it:adj-monad} 
The adjunction $i^*\dashv i_*$ induces on~$\cat C$ a monad $\mathbb M= (i_*i^*, i_* \reps i^*, \reta)$.
\item \label{it:adj-comonad}
The adjunction $i_!\dashv i^*$ induces on~$\cat C$ a comonad $\mathbb C= (i_! i^*, i_! \leta  i^*, \leps)$.
\item \label{it:adj-Frob}
If $i_*$ and $i_!$ are the same functor, \eqref{it:adj-monad} and \eqref{it:adj-comonad} combine to define a Frobenius monad on $\cat C$ with underlying functor $i_*i^* = i_!i^*$.
\end{enumerate}

\end{Prop}

\begin{proof}
Parts \eqref{it:adj-monad} and \eqref{it:adj-comonad} are standard and easily verified, and so is \eqref{it:adj-Frob}, although the latter is sometimes obscured by the choice of axiomatization. For completeness, let us note that the following two half-squares
\[
\xymatrix{
& i^*i_* \ar[dl]_{\leta\,i^*i_*} \ar[d]_(.45){\reps} \ar[dr]^{i^*i_* \,\leta} & \\
i^*i_*i^*i_* \ar[dr]_{i^*i_* \,\reps} & { \Id} \ar[d]_(.45){\leta}   & i^*i_*i^*i_* \ar[dl]^{\reps\,i^*i_*} \\
& i^*i_* &
}
\]
commute by the naturality of~$\leta$ and $\reps$, respectively, and that the Frobenius property \eqref{eq:Frob-monad} follows by applying $i_*\circ - $ and~$- \circ i^*$ to this diagram.
\end{proof}

\begin{Prop} \label{Prop:induced-co-monoid}
Let $i^*\colon \cat C\to \cat D$ be any strong (braided) monoidal functor between any two (braided) monoidal categories. Then:
\begin{enumerate}[\rm(1)]
\item \label{it:lax}
A right adjoint $i_*\colon \cat D\to \cat C$ of $i^*$ inherits the structure of a lax monoidal functor from the strong monoidal structure of the latter.
If $i^*$ is braided then so is~$i_*$.

\item \label{it:colax}
Dually, a left adjoint $i_!\colon \cat D\to \cat C$ of $i^*$ inherits the structure of a colax monoidal functor from the strong monoidal structure of the latter,
braided if $i^*$ is.

\item \label{it:co-monoid}
It follows that a right adjoint $i_*$ sends (commutative) monoids in $\cat D$ to (commutative) monoids in~$\cat C$, and a left adjoint $i_!$ preserves (cocommutative) comonoids.

\item \label{it:co-monoid-unit}
In particular, since the unit $\Unit_\cat D$ is always both a (commutative) monoid and a (cocommutative) comonoid in a unique way, $i_*(\Unit_\cat D)$ is canonically a (commutative) monoid and $i_!(\Unit_\cat D)$ a (cocommutative) comonoid.
Explicitly, the multiplication~$\mu$ and comultiplication~$\delta$ are respectively given by
\[
\vcenter{
\xymatrix@C=14pt@R=14pt{
& i_*(\Unit) \otimes i_*(\Unit) 
  \ar `l[dl] `[dddd]_\mu^{:=}   [dddd]
   \ar `r[r] `[ddd]^(.3){\mathrm{lax}} [ddd]
   \ar[d]^{\eta} &
   \\
& i_*i^*(i_*\Unit \otimes i_*\Unit) \ar[d]_{\simeq}^{i_*(\mathrm{strong}^{-1})} \\
& i_*(i^*i_*\Unit \otimes i^*i_*\Unit) \ar[d]^{i_*(\varepsilon \,\otimes\, \varepsilon)} \\
& i_*(\Unit \otimes \Unit) \ar[d]^{\simeq} \\
& i_*(\Unit)
}}
\quad \textrm{ and } \;
\vcenter{
\xymatrix@C=14pt@R=14pt{
&
i_! (\Unit) \ar[d]_{\simeq}
 \ar `r[r] `[dddd]^{\delta}_{=:} [dddd]
 & & \\
 &
i_! (\Unit \otimes \Unit)
 \ar[d]_{i_!(\eta \,\otimes\, \eta)}
    \ar `l[dl] `[ddd]_(.7){\mathrm{colax}\,}   [ddd]
 & \\
&
i_! (i^*i_! \Unit \otimes i^*i_! \Unit) \ar[d]^{\simeq}_{i_!(\mathrm{strong})} & \\
&
i_! i^*(i_! \Unit \otimes i_! \Unit) \ar[d]_{\varepsilon} & \\
&
i_! (\Unit) \otimes i_! (\Unit) &
}}
\]
%%%
and the unit and counit by 
\[
\iota\colon
\xymatrix{
\Unit_\cat C \ar[r]^-\eta & i_*i^*(\Unit_\cat C) \cong i_*(\Unit_\cat D)
}
\quad \textrm{ and }\quad
\epsilon\colon
\xymatrix{
i_!(\Unit_\cat D) \cong i_!i^*(\Unit_\cat C) \ar[r]^-\varepsilon & \Unit_\cat C
}
\]
(\cf \Cref{Rem:unit-cond}).
\end{enumerate}
\end{Prop}

\begin{proof}
As this is all standard tensorial lore, we just recall the constructions for the reader's convenience and omit the straightforward verifications.

For~\eqref{it:lax}, the lax structure on $i_*$ consists of the multiplication (for $X,Y\in \cat D $)
\[
\xymatrix@C=20pt{
i_*(X) \otimes_\cat C i_*(Y) \ar[r]^-{\eta} &
 i_*i^*(i_*(X) \otimes_\cat C i_*(Y)) \cong i_*(i^*i_*X \otimes_\cat C i^*i_*Y) \ar[r]^-{i_*(\varepsilon \,\otimes\, \varepsilon)} &
   i_*(X \otimes_\cat D Y)
}
\]
with unit map
$
\xymatrix@1{
\Unit_\cat C \ar[r]^-{\eta} &
 i_*i^* (\Unit_\cat C) \cong i_* (\Unit_\cat D),}
 $ 
both obtained from the strong monoidal structure of~$i^*$ and the adjunction $i^*\dashv i_*$.
The colax structure \eqref{it:colax} on $i_!$ is obtained dually from the adjunction $i_!\dashv i^*$.

For \eqref{it:co-monoid}, just combine the lax structure of~\eqref{it:lax} with any given multiplication $\mu$ and unit $\iota$ on an object $X=Y\in \cat D$, as follows:
\[
\xymatrix@C=16pt{
i_*(X) \otimes_\cat C i_*(X) \ar[r]^-{\textrm{lax}} &
 i_*(X \otimes_\cat D X) \ar[r]^-{i_*(\mu)} &
  i_*(X)
}
,\quad
\xymatrix@C16pt{
\Unit_\cat C \ar[r]^-{\textrm{lax}} &
 i_*(\Unit_\cat D) \ar[r]^-{i_* (\iota)} &
  i_*(X)
}.
\]
These are an associative multiplication on~$i_*(X)$, commutative if $\mu $ is commutative, and its unit map. 
Dually with $i_!$ and comonoids.

For \eqref{it:co-monoid-unit}, the coherent isomorphism $\Unit_\cat D\otimes_\cat D\Unit_\cat D\cong \Unit_\cat D$ and its inverse provide the unique (co)\-multi\-pli\-ca\-tion on~$\Unit_\cat D$ with (co)unit map $\id_{\Unit_\cat D}$. This is always (co)\-commut\-ative as soon as $\cat D$ is braided.
\end{proof}
\begin{center}$*\;*\;*$\end{center}

From now on, let $\cat M$ be a Green 2-functor and denote by~$\otimes=\otimes^\cat M$ its tensor structure. Let $i\colon H\to G$ be in~$\JJ$.
As $i^*$ is a strong monoidal functor (\Cref{Rem:pointwise-structure}), we can apply \Cref{Prop:induced-co-monoid}  to it. 
The two next lemmas show the relationship between the resulting structures and the projection isomorphisms of \Cref{Thm:proj-formula}.

\begin{Lem} \label{Lem:proj-lax-colax}
The lax and colax structures on $i_*=i_!$ are related to the right and left projection isomorphisms (\Cref{Def:Frobs}) by the following commutative triangles:
\begin{align*}
\vcenter{
\xymatrix@C=-14pt{
& i_!(i^*X \otimes Y) \ar[ld]_{\mathrm{colax}\;\;} \ar[dr]^-{\lproj^1}_\simeq & \\
i_!i^* X \otimes i_!Y \ar[rr]_-{\leps \,\otimes \, \id} & & X \otimes i_!Y
}}
\quad\quad
\vcenter{
\xymatrix@C=-14pt{
X \otimes i_*Y \ar[rr]^-{\reta \,\otimes\, \id} \ar[dr]_-{\rproj^1}^\simeq & & i_*i^* X \otimes i_*Y \ar[dl]^{\mathrm{lax}} \\
& i_*(i^* X \otimes Y)&
}}
\\
\vcenter{
\xymatrix@C=-14pt{
& i_!(Y \otimes i^* X) \ar[dl]_-{\lproj^2}^\simeq \ar[dr]^{\mathrm{colax}} & \\
i_!Y \otimes X & & i_! Y \otimes i_!i^* X \ar[ll]^-{\id \,\otimes\, \leps}
}}
\quad\quad
\vcenter{
\xymatrix@C=-14pt{
{i_*Y} \otimes i_*i^* X \ar[dr]_{\mathrm{lax}}& & i_*Y \otimes X \ar[dl]_{\simeq}^-{\rproj^2} \ar[ll]_-{\id\,\otimes \,\reta} \\
& i_*(Y \otimes i^*X) & 
}}
\end{align*}
\end{Lem}

\begin{proof}
Consider the following diagram:
\[
\xymatrix{
& i_* Y \otimes X
 \ar `l[dl] `[ddd]_{\rproj^2}   [ddd]
 \ar[rr]^-{\id \, \otimes \, \reta}
  \ar[d]^{\reta}
   \ar@{}[dddl]_{\textrm{def.}}
   && 
i_* Y \otimes i_*i^* X
 \ar[d]_{\reta}
  \ar `r[r] `[ddd]^{\mathrm{lax}} [ddd]
   \ar@{}[dddr]^{\textrm{def.}}
  & \\
& i_*i^* (i_* Y \otimes X)
 \ar[rr]^-{i_*i^* (\id \,\otimes\, \reta)}
  \ar[d]_{\simeq}^{i_*(\otimes_i)} &&
i_*i^* ( i_* Y \otimes i_*i^* X )
  \ar[d]^\simeq_{i_*(\otimes_i)} & \\
& i_*( i^*i_* Y \otimes i^*X )
 \ar[rr]^-{i_*(\id\,\otimes\,i^*\reta)}
  \ar[d]^{i_*(\reps\,\otimes\,\id)} &&
i_*(i^*i_* Y \otimes i^*i_*i^* X)
 \ar[d]_{i_*(\reps\,\otimes\,\reps)}  & \\
& i_*(Y \otimes i^*X) \ar@{=}[rr] && i_*(Y \otimes i^*X) &
}
\]
The top square commutes by the naturality of~$\reta$, the middle one by the naturality of~$\otimes_i$, and the bottom one by the triangular identity $(\reps \,i^*)(i^* \,\reta)=\id_{i^*}$ for the adjunction~$i^*\dashv i_*$. 
This proves the fourth claimed identity. 
The proofs for the other three are symmetrical and are left to the reader.
\end{proof}

\begin{Lem} \label{Lem:decomps-Frob}
The multiplication $\mu$ and comultiplication $\delta$ on $A:=i_!(\Unit_H)=i_*(\Unit_H)$ are related to the projection maps  by the following commutative triangles:
\[
\vcenter{
\xymatrix@C=2pt{
& i_*\Unit \otimes i_*\Unit \ar[dl]_{\rproj^1}^\simeq \ar[dr]^{\rproj^2}_\simeq \ar[dd]_{\mu} & \\
i_*(i^*i_* \Unit \otimes \Unit) \ar[dr]_{i_*(\reps \,\otimes\, \Unit)} & & i_*(\Unit \,\otimes\, i^*i_* \Unit) \ar[dl]^{i_*(\Unit \,\otimes\, \reps)} \\
& i_*\Unit &
}}
\quad\quad
\vcenter{
\xymatrix@C=2pt{
& i_!\Unit \ar[dl]_{i_!(\leta \,\otimes\, \Unit)} \ar[dr]^{i_!(\Unit \,\otimes\, \leta)} \ar[dd]_\delta & \\
i_!(i^*i_!\Unit \otimes \Unit) \ar[dr]_{\lproj^1}^\simeq & & i_!(\Unit \otimes i^*i_! \Unit) \ar[dl]_\simeq^{\lproj^2} \\
& i_!\Unit \otimes i_!\Unit &
}}
\]
\end{Lem}

\begin{proof}
The following commutative diagram displays one of the claimed relations:
\[
\xymatrix@R=16pt{
& A\otimes A
 \ar@{}[dddr]|{\textrm{def.}}
 \ar[d]_{\reta}
  \ar@/^6ex/[dddr]^{\rproj^2}_\simeq
  \ar `l[dl] `[ddddd]_\mu^{\;\;\;\textrm{def.}}   [ddddd]
    & \\
& i_*i^*( i_*\Unit \otimes i_* \Unit) \ar[d]^\simeq_{i_*(\otimes_i)} & \\
& i_*(i^*i_*\Unit \otimes i^*i_*\Unit) \ar[dd]_{i_*(\reps\,\otimes\, \reps)} \ar[dr]^{\;\; i_*(\reps \,\otimes\, \Unit)} & \\
& & i_*(\Unit \otimes i^*i_*\Unit) \ar[dl]^{\;\; i_*(\Unit \,\otimes\, \reps)} \\
& i_*(\Unit \otimes \Unit) \ar[d]^\simeq & \\
& A&
}
\]
The other three decompositions are symmetrical and are left to the reader.
\end{proof}

\begin{Cor} \label{Cor:specialness}
The special relation holds: $\mu \circ \delta = \id_A$.
\end{Cor}

\begin{proof}
The diagram
\[
\xymatrix@R=16pt{
& A
  \ar[dr]^{\;\;\;i_!(\Unit \,\otimes\, \leta)}
  \ar[dl]_{i_!(\leta \,\otimes\, \Unit)\;\;} 
  \ar[dd]_\delta & \\
 i_!(i^*i_! \Unit \otimes \Unit)
   \ar@{=}[dd] \ar[dr]|{\lproj^1} && 
i_!( \Unit \otimes i^*i_! \Unit)
  \ar[dl]|{\;\;\lproj^2\;} \ar@{=}[dd] \\
&A \otimes A
  \ar[dd]_\mu 
  \ar[dr]|{\rproj^2}
  \ar[dl]|{\rproj^1} & \\
i_*(i^*i_! \Unit \otimes \Unit)
  \ar[dr]_{i_*(\reps\,\otimes \,\Unit)\;\;} & & 
i_*(\Unit \otimes i^*i_*\Unit)
  \ar[dl]^{\;\;i_*(\Unit \,\otimes\, \reps)} \\
 & A &
}
\]
is commutative by \Cref{Thm:proj-formula} and \Cref{Lem:decomps-Frob}. 
We conclude with the special Frobenius property~\eqref{eq:special-Frob} (either the left or right half of the diagram suffices).
%$\reps\circ \leta = \id$ of the ambijunction $i_!\dashv i^* \dashv i_*$ for the rectified Mackey 2-functor~$\cat M$; see 
\end{proof}

\begin{Thm} \label{Thm:Frobenius}
Let $\cat M$ be any Green 2-functor. 
Let $i\colon H\to G$ be in~$\JJ$ and
write $A:=i_!(\Unit_H)=i_*(\Unit_H)$. 
Then the canonical monoid and comonoid structures on $A$ obtained from the adjunctions $i_!\dashv i^*\dashv i_*$ (\Cref{Prop:induced-co-monoid}) turn $A$ into a special Frobenius object in the monoidal category~$\cat M(G)$, which is commutative as soon as $\cat M$ is braided. In particular, $A$ is a separable algebra.
\end{Thm}

\begin{proof}
The monoid $(A,\mu,\iota)$ is separable because by \Cref{Cor:specialness} its multiplication admits a bilinear section, namely the comultiplication~$\delta$.

To prove that the monoid and comonoid of \Cref{Prop:induced-co-monoid} assemble into a Frobenius structure on~$A$, we must still verify that $\delta$ is indeed bilinear with respect to~$\mu$.
By a well-known reduction (see \eg \cite[Lemma~5.4]{HeunenVicary19}), it suffices to have the following \emph{Frobenius relation}:
\begin{equation}  \label{eq:Frob}
(A \otimes \mu)(\delta \otimes A) = (\mu \otimes A )( A \otimes \delta) .
\end{equation}
A direct verification of this identity from the definitions appears to be very complicated. 
We will instead translate this problem in terms of monads, after which it will become  trivial.

To this end, note that we have two monads on the category~$\cat M(G)$. 
The first monad $\bbM$ is induced by the adjunction $i^*\dashv i_*$, whose underlying functor (also denoted~$\bbM$) is $i_*i^*\colon \cat M(G)\to \cat M(G)$, and whose multiplication and unit are
\[
m:= i_*\,\reps \,i^* \colon \bbM^2 = i_*i^*i_*i^* \longrightarrow i_*i^* = \bbM
\quad \textrm{ and } \quad
u:=\reta\colon \Id \longrightarrow i_*i^* = \bbM
\]
respectively.
The second monad is $A\otimes(-)\colon \cat M(G)\to \cat M(G)$, with multiplication and unit obtained by tensoring with the multiplication $\mu$ and unit~$\iota$ of the monoid~$A$.
For $Y=\Unit$, the second right projection map specializes to a natural isomorphism
\begin{equation} \label{eq:proj-monads}
\pi_X := \rproj^2_{\Unit,X} \colon A\otimes X = i_*(\Unit) \otimes X \overset{\sim}{\longrightarrow} i_*(\Unit \otimes i^*X ) \cong \bbM X
\end{equation}
which happens to always be an isomorphism between the two monads, \ie  $\pi_X$ identifies the multiplications and units.
This is proved in \cite[Lemma~2.8]{BalmerDellAmbrogioSanders15} (see also \cite[Lemma 6.5]{BLV11});
note that the projection map defined in \emph{loc.\,cit.\ }is precisely our~\eqref{eq:proj-monads} thanks to \Cref{Lem:proj-lax-colax}.

We also have a dual result, as follows. The adjunction $i_!\dashv i^*$ can be viewed as an adjunction $(i^*)^\op \dashv (i_!)^\op$ between $\cat M(G)^\op$ and~$\cat M(H)^\op$, where the left (!) adjoint $(i^*)^\op$ is a strong monoidal functor for the canonical monoidal structures on the opposite categories.  
Therefore we may apply \cite[Lemma~2.8]{BalmerDellAmbrogioSanders15} to this adjunction as well. 
Translating back in terms of the original adjunction $i_!\dashv i^*$ on the original categories $\cat M(G)$ and~$\cat M(H)$, and using another identity of \Cref{Lem:proj-lax-colax}, this says that $(\rproj^2_{Y,X})^{-1}=\lproj^2_{Y,X}$ specializes for $Y=\Unit$ to a morphisms of \emph{comonads}
\[
\pi_X^{-1} \colon \bbM  X= i_!i^* X \cong i_!(\Unit \otimes i^* X) \overset{\sim}{\longrightarrow} A \otimes X .
\]
which, by \Cref{Thm:proj-formula}, is the inverse map of~\eqref{eq:proj-monads}. Now $\bbM$ is the comonad on $\cat M(G)$ induced by the adjunction $i_!\dashv i^*$, whose comultiplication and counit are
\[
d:= i_! \,\leta\, i^* \colon \bbM \Longrightarrow \bbM^2
\quad \textrm{ and } \quad
e:=\leps \colon \bbM \Longrightarrow \Id
,
\]
and $A\otimes(-)$ is equipped with the comultiplication $\delta\otimes (-)$ and counit $\epsilon \otimes (-)$ induced by tensoring with the comonoid structure of~$A$.

Altogether, $\pi_X$ is both an isomorphism of monads and comonads. 

Consider now the following diagram of natural transformations between endofunctors of~$\cat M(G)$, where $\pi^{(2)}$ and $\pi^{(3)}$ denote the twofold and threefold application of $\pi$ (\ie the horizontal composites of $\pi$ with itself in the 2-category of categories):
\[
\xymatrix@!C@C=10pt{
& A^{\otimes 2} \otimes (-)
 \ar[dl]_{\delta \otimes A \otimes-}
  \ar[ddr]^(.3){A \otimes \delta \otimes -}|{\phantom{M}}
   \ar[rrr]_-\sim^-{\pi^{(2)}} & & &
 \bbM^2
   \ar[dl]_{d \, \bbM}
    \ar[ddr]^{\bbM \, d} & \\
A^{\otimes 3} \otimes (- )
 \ar[ddr]_{A \otimes \mu \otimes -}
  \ar[rrr]^(.3){\pi^{(3)}}_(.3){\sim} & & &
 \bbM^3
  \ar[ddr]^(.3){\bbM \, m} & & \\
& & A^{\otimes 3}\otimes (-)
 \ar[dl]^{\mu \otimes A \otimes -}
  \ar[rrr]^(.35){\pi^{(3)}}_(.35){\sim}|(.50){\phantom{M}} & & & 
\bbM^3
  \ar[dl]^{m \, \bbM} \\
& A^{\otimes 2} \otimes (-)
 \ar[rrr]^-{\pi^{(2)}}_-\sim & & & 
\bbM^2 &
}
\]
By the previous discussion, the four squares involving powers of $\pi$ all commute. 
The square on the right also commutes because $\bbM$ is a Frobenius monad; see \Cref{Prop:Frob-co-monad}\,\eqref{it:adj-Frob}.
We conclude that the left square also commutes. By applying this to the unit object~$\Unit \in \cat M(G)$ we obtain the claimed Frobenius relation~\eqref{eq:Frob}. 
\end{proof}

\begin{Rem} \label{Rem:self-duality}
Like all Frobenius objects, $A$ is its own left and right tensor-dual. The unit and counit 
 $\alpha \colon \Unit\to A\otimes A$
  and 
 $ \beta \colon A\otimes A\to \Unit$
of this duality are given by
\[ 
\alpha= \delta \circ \iota
\quad \textrm{ and } \quad
\beta=\epsilon\circ \mu \,,
\] 
where $(A,\mu,\iota)$ and $(A,\delta,\epsilon)$ are the monoid and comonoid structures of \Cref{Prop:induced-co-monoid}\,\eqref{it:co-monoid-unit}.
Note that the two projection formulas and the two adjunctions easily imply that $A$ is self-dual; but the fact that one can specifically take the above two maps to establish a duality utilises the Frobenius structure on~$A$, \ie ultimately the fact that the left and right projection maps are mutual inverses.
\end{Rem}

%------------------------------------------------------------------------------
\section{The canonical half-braiding}
\label{sec:half-braiding}%
%\bigbreak
%------------------------------------------------------------------------------

Let $\cat M$ be a Green 2-functor and let $i\colon H\to G$ be in~$\JJ$. 
In this section we upgrade the previous results by showing that the induced Frobenius object $A(i):= i_!(\Unit) =i_* (\Unit)$ of \Cref{sec:Frob} belongs to the monoidal center of~$\MM(G)$ in the sense of \cite{JoyalStreet91b}. This means that it comes equipped with a half-braiding, \ie a prescribed way for $A(i)$ to naturally commute with all objects of~$\MM(G)$. 
Moreover, this half-braiding is fully compatible with the Frobenius structure on~$A(i)$.
Recall:

\begin{Def} 
\label{Rec:half-braiding}
If $A$ is an object in a monoidal category~$\cat C$, a \emph{half-braiding on~$A$} is a family $\sigma = \{\sigma_X\colon A \otimes X\overset{\sim}{\to} X\otimes A\}_{X\in \cat C}$ of isomorphisms of $\cat C$ natural in~$X$ and satisfying the braid relation
\begin{equation} \label{eq:half-braiding}
\sigma_{X\otimes Y} = (X \otimes \sigma_Y)(\sigma_X \otimes Y) \quad \textrm{ for all } X,Y\in \cat C\,.
\end{equation}
(This also implies $\sigma_{\Unit}=\id_A$.)
Pairs $(A,\sigma)$ consisting of an object $A \in \cat C$ and a half-braiding $\sigma$ on~$A$, together with morphisms of $\cat C$ which commute with the half-braidings, form a category  $\cat Z(\cat C)$ called the \emph{monoidal center} or \emph{Drinfeld center} of~$\cat C$. It is a braided monoidal category for the evident tensor product and braiding induced, respectively, by the tensor of $\cat C$ and by the specified half-braidings on its objects; this makes the forgetful functor $\cat Z(\cat C)\to \cat C$ strong monoidal. 
We refer to \cite[Ch.\,VIII]{Kassel95} for details, context and examples.
Conceptually, consider the delooping $\mathrm B\cat C$, \ie the bicategory with a single object whose monoidal endo-category is~$\cat C$. 
Then $\cat Z(\cat C)$ is precisely just $\End(\Id_{\mathrm B\cat C})$, the (naturally braided monoidal)  category of endo-pseudonatural transformations of the identity pseudofunctor of~$\mathrm B\cat C$.  
\end{Def}

\begin{Thm} \label{Thm:half-braiding}
Let  $\MM$ be a Green 2-functor, let $i\colon H\to G$ be in~$\JJ$, and recall the left and right projection isomorphisms $\lproj^1$ and $\rproj^2$ of \Cref{Thm:proj-formula}. Then:
\begin{enumerate}[\rm(1)]
\item
\label{it:hb1}
Write $A:=A(i)$ for short. The family of composite isomorphisms
\[
\sigma_X \colon 
\xymatrix@C=17pt{
A\otimes X \ar[rr]^-{\rproj^2_{\Unit, X}}_-\sim && i_*(\Unit \otimes i^*X) \cong i_*(i^*X \otimes \Unit) \ar[rr]^-{\lproj^1_{X,\Unit}}_-\sim && X\otimes A
}
\]
defines a half-braiding $\sigma$ on~$A$ in the monoidal category~$\MM(G)$ (\Cref{Rec:half-braiding}).
In case $\cat M$ is braided, $\sigma$ coincides with the given braiding on~$\cat M(G)$.
\item 
\label{it:hb2}
The canonical Frobenius structure on~$A$ of \Cref{Thm:Frobenius} turns $(A,\sigma)$ into a special commutative Frobenius object in the monoidal center~$\cat Z(\cat M(G))$.
\item 
\label{it:hb3}
The right projection maps $\rproj^2\colon A\otimes (-)  \Rightarrow i_*i^*$ and $\rproj^1 \colon (-)\otimes A \Rightarrow i_*i^*$
are isomorphisms of colax monoidal and lax monoidal endofunctors of~$\cat M(G)$.
\end{enumerate}
\end{Thm}

We will obtain most of the theorem's claims by specializing results of~\cite{BLV11}. 
This conveniently saves us from a large number of verifications, but also requires some care in translating between terminologies. 
Here is our dictionary:

\begin{Prop} \label{Prop:Hopf-monad}
Let $\cat M$ be a Green 2-functor and let $i\colon H\to G$ be in~$ \JJ$. 
Then, in the terminology of~\cite{BLV11} (see the proof below for explanations):
\begin{enumerate}[\rm(a)]
\item
\label{it:Hopf}
The adjunction $i_!\dashv i^*$ is a \emph{Hopf (comonoidal) adjunction}.
The endofunctor $E= i_!i^*$ of $\cat M(G)$ inherits the structure of a \emph{comonoidal comonad}, and the endofunctor $F= i^*i_!$ of $\cat M(H)$ the structure of a \emph{Hopf (comonoidal) monad}.

\item 
\label{it:co-Hopf}
Dually, $i^* \dashv i_*$ is a \emph{Hopf monoidal adjunction} turning $E= i_*i^*$ into a \emph{monoidal monad} on $\cat M(G)$ and $F= i^*i_*$ into a \emph{Hopf monoidal comonad} on~$\cat M(H)$. 

\item 
\label{it:Hopf-operators}
The \emph{Hopf operators} $\mathbb H^\ell, \mathbb H^r$ of the comonoidal adjunction $i_!\dashv i^*$ are inverse to those of the monoidal adjunction $i^*\dashv i_*$, which moreover identify with our projection isomorphisms (\Cref{Def:Frobs}) as follows: 
\begin{align*}
\mathbb H^r_{X,Y} &= \lproj^1_{X,Y} \colon i_!(i^*X \otimes Y) \to X \otimes i_!(Y) \\
\mathbb H^\ell_{Y,X} &= \lproj^2_{Y,X}\colon i_!(Y \otimes i^*X) \to i_!(Y) \otimes X \,.
\end{align*}
Similarly, the two \emph{fusion operators} $H^\ell, H^r$ of the Hopf monad $F=i^*i_!$ are inverses to those of the Hopf comonad~$F= i^*i_*$.
\end{enumerate}
\end{Prop}

\begin{Rem}
Even more could be said: 
$i_!\dashv i^* \dashv i_*$ is an example of an \emph{ambidextrous biHopf adjunction}, from which it follows \eg that induction $i_!=i_*$ is a \emph{Frobenius monoidal functor}. 
See \cite{Balan17} for explanations on this, and more generally on how (co)Hopf adjunctions are related to various Frobenius-type properties.
\end{Rem}

\begin{proof}[Proof of \Cref{Prop:Hopf-monad}]
The adjunction $i_!\dashv i^*$ is \emph{comonoidal}, meaning that $i_!,i^*$ are comonoidal (\ie colax monoidal) functors and $\leta,\leps$ are comonoidal transformations. Indeed, $i^*$ comes with a strong (hence both lax and colax) monoidal structure, and the colax structure induced on~$i_!$ (as in \Cref{Prop:induced-co-monoid}\,\eqref{it:colax}) turns the unit and counit of adjunction into comonoidal transformations.
It follows that the induced monad $i^*i_!$ on $\cat M(H)$ and, dually, the induced comonad $i_!i^*$ on $\cat M(G)$  (as in \Cref{Prop:Frob-co-monad}) are also comonoidal, \ie their underlying functors and all structural natural transformations are comonoidal.

By definition (\cite[\S2.8]{BLV11}), a comonoidal adjunction is a \emph{Hopf adjunction} if its left and right Hopf operators $\mathbb H^\ell, \mathbb H^r$ are invertible. 
A direct inspection of the definitions together with \Cref{Lem:proj-lax-colax} reveals that, as claimed in part~\eqref{it:Hopf-operators}, $\mathbb H^\ell =  \lproj^2$ and $\mathbb H^r = \lproj^1$ are precisely our two left projection maps, which we know to be invertible thanks to \Cref{Thm:proj-formula}. 
By \cite[Thm.\,2.15]{BLV11}, the comonoidal monad induced by a Hopf adjunction is automatically a Hopf monad in the sense of \cite[\S2.7]{BLV11}, \ie its two fusion operators $H^\ell,H^r$ are invertible.  

All claims in part~\eqref{it:Hopf} are now proved.
Part \eqref{it:co-Hopf} is formally dual, starting with the \emph{(lax) monoidal} adjunction~$i^*\dashv i_*$.
To conclude the proof of part~\eqref{it:Hopf-operators}, it suffices now to recall the simple relation between the fusion operators $\mathbb H^\ell, \mathbb H^r$ and the Hopf operators $H^\ell, H^r$, which we leave to the reader (see \cite[\S2.9]{BLV11}). 
\end{proof}

\begin{proof}[Proof of \Cref{Thm:half-braiding}]
As in \Cref{Prop:Hopf-monad}, we have the Hopf adjunction $i_!\dashv i^*$ with Hopf operators $\mathbb H^r_{X,Y}= \lproj^1_{X,Y}$ and $\mathbb H^\ell_{Y,X}=\lproj^2_{Y,X}$. 
We can therefore apply \cite[Thm.\,6.6]{BLV11} to immediately obtain the following three facts, corresponding to (portions of) the three claims of our theorem:
\begin{enumerate}[\rm(1)]
\item 
The composite natural isomorphism $\sigma := \lproj^1_{X,\Unit} \rproj^2_{\Unit,X} = \mathbb H^r_{X,\Unit} (\mathbb H^\ell_{\Unit, X})^{-1}$ defines a half-braiding on $A= i_!(\Unit)$.
\item 
This braiding turns the comonoid~$A$, defined precisely as in \Cref{Prop:induced-co-monoid}\,\eqref{it:co-monoid-unit}, into a cocommutative comonoid of the center $\cat Z(\cat M(G))$.

\item The natural map $(\mathbb H^\ell_{\Unit,X})^{-1} = \rproj^2_{\Unit, X} \colon A \otimes X \overset{\sim}{\to} i_!i^* (X)$ is an isomorphism of comonoidal monads (in fact, we already know from the proof of \Cref{Thm:Frobenius} that it is an isomorphism of the underlying plain monads).
\end{enumerate}
Taken together, these say that the comonoidal monad $i_!i^*$ is canonically represented by the induced central comonoid~$A$.
There is a formally dual result, stating that the monoid $A=i_*(\Unit)$ is central and represents the monoidal comonad~$i_*i^*$.

This covers parts \eqref{it:hb2} and~\eqref{it:hb3} of the theorem.
To complete the proof of part~\eqref{it:hb1}, suppose now that the Green 2-functor~$\cat M$ is braided. We must check that $\sigma$ agrees with the given braiding on the object~$A$. 
By hypothesis, the restriction functor $i^*$ is braided monoidal; in particular the central square in the following diagram is commutative, where the two braidings are denoted by $\beta_{X,Y}\colon X\otimes Y \overset{\sim}{\to} Y \otimes X$:
\[
\xymatrix{
& A \otimes X \ar[rr]^-{\beta_{A,X}} \ar[d]^{\reta} 
    \ar `l[dl] `[ddd]_{\rproj^2_{\Unit, X}}   [ddd]
     \ar@{}[dddl]_{\textrm{def.}} &&
 X\otimes A  \ar[d]_{\reta}
   \ar `r[dr] `[ddd]^{\rproj^1_{X, \Unit}}   [ddd]
    \ar@{}[dddr]^{\textrm{def.}}  & \\
& i_*i^*(A \otimes X) \ar[rr]^-{i_*i^*(\beta_{A,X})} \ar[d]^{i_*(\otimes_i)}_\simeq &&
 i_*i^*(X \otimes X) \ar[d]_{i_*(\otimes_i)}^\simeq & \\
& i_*(i^*A \otimes i^*X) \ar[d]^{i_*(\reps \,\otimes\, \id)} \ar[rr]^-{i_*(\beta_{i^*A,i^*X})} &&
 i_*(i^*X \otimes i^*A) \ar[d]_{i_*(\id \,\otimes\, \reps)} & \\
& i_*(\Unit \otimes i^*X) \ar[rr]^-\simeq &&
 i_*(i^*X \otimes \Unit) &
}
\]
The top square commutes by the naturality of $\reta$ and the bottom square by the naturality of~$\beta$ combined with the fact that $\beta_{\Unit, i^*X}$ agrees (by a braiding axiom) with the composite $\Unit \otimes i^*X \simeq i^*X \simeq i^*X \otimes \Unit$ of the left and right unitors. It follows that $\beta_{A,X} =  (\rproj^1_{X,\Unit})^{-1} \circ \rproj^2_{\Unit, X}= \lproj^1_{X,\Unit} \circ \rproj^2_{\Unit,X} = \sigma_X$, as claimed.
\end{proof}
\begin{center}$*\;\;*\;\;*$\end{center}

As our Mackey 2-functors are rectified, the unit-counit composite $\reps \circ \leta$ of any $(i\colon H\to G)\in \JJ$ is just the identity of~$\Id_{\cat M(H)}$. As we have seen, under the projection formula this gives rise to the speciality of the Frobenius object~$A(i)$. 
Pleasantly, the \emph{other} unit-counit composite also has a meaning for~$A(i)$:

\begin{Prop} \label{Prop:dim}
The component of the composite $\leps \circ \reta\colon \Id_{\cat M(G)} \Rightarrow \Id_{\cat M(G)}$   at the tensor unit $\Unit$ is the monoidal dimension $\dim A(i) = \mathrm{tr}(\id_{A(i)}) \in \End_{\cat M(G)}(\Unit)$.
\end{Prop}

\begin{proof}
The Frobenius object $A= A(i)$ admits the self-duality $(\alpha,\beta)$ of \Cref{Rem:self-duality},
hence its dimension (a.k.a.\ Euler characteristic) may be computed as the composite $\beta \circ \mathrm{braid} \circ \alpha \in \End(\Unit)$. 
Here any braiding ``$\mathrm{braid}$'' would do, as it would coincide with the canonical half-braiding by \Cref{Thm:half-braiding}\,\eqref{it:hb1}. 
In fact we can omit the braiding altogether, since $\beta = \epsilon \circ \mu$ and $\mu$ is always commutative by part~\eqref{it:hb2} (or dually, since $\alpha = \delta \circ \iota $ and $\delta$ is cocommutative).
The following commutative diagram
\[
\xymatrix@R=12pt{
& & & & \\
\Unit
 \ar@/_2ex/[ddr]_{\reta}
 \ar[dr]^\iota
 \ar[rr]^-\alpha
 \ar@/^6ex/[rrrr]^-{\dim A} &&  
 A \otimes A
  \ar[dr]_\mu
  \ar[rr]^-{\beta}  &&  \Unit \\
& 
A
 \ar[ur]_\delta
 \ar@/_2ex/@{=}[rr]  && 
A
 \ar[ur]^\epsilon & \\
& 
i_*i^*(\Unit )
 \ar@/_2ex/@{=}[rr]
 \ar[u]_\simeq && 
i_!i^*(\Unit)
 \ar[u]^\simeq
 \ar@/_2ex/[uur]_{\leps} &
}
\vspace{.2cm}
\]
which also makes use of the speciality of $A$ and the definition of its unit $\iota$ and counit~$\epsilon$, proves the claimed identity.
\end{proof}

\begin{Rem} \label{Rem:generalized-index}
(Cf.\ \cite[Rem.\,5.9]{BalmerDellAmbrogio21pp}.)
The object~$A(i)$ can be viewed as both a categorification and a generalization of the index of $H$ in~$G$, for \emph{two} different reasons. 
To understand this, let $\cat M$ be a Green 2-functor for $(\gpd;\gpdf)$, say.
As we shall see (\Cref{Thm:Green-K-decat}), by applying the $\mathrm K_0$-functor to $\cat M$ we get an ordinary Green functor~$M$ (global and with inflation maps; see \Cref{sec:1-Green}).
For any subgroup inclusion $i\colon H \hookrightarrow G$, the projection formula $i_*i^* \cong A(i) \otimes - $ immediately yields the identity $i_\sbull i^\sbull = [A(i)] \cdot \id_{M(G)}$, where $i_\sbull$ and $i^\sbull$ are the induction and restriction maps $M(H)\leftrightarrows M(G)$ and where $[A(i)]$ is the class of $A(i)$ in $M(G)=\mathrm K_0(\cat M(G))$.
Recall that classically a Mackey functor is said to be \emph{cohomological} iff $i_\sbull i^\sbull$ is multiplication by the index~$[G:H]$ for all such~$i$; hence $[A(i)]$ plays the role of the index, even when $M$ is not cohomological.
But there is a second reason: $\cat M$ itself, as a Mackey 2-functor, is \emph{cohomological} in the sense of \cite{BalmerDellAmbrogio21pp} iff the composite 
$\leps \circ \reta\colon \colon \Id_{\cat M(G)}\Rightarrow \Id_{\cat M(G)}$ is multiplication by $[G:H]$ for all such~$i$. (Cohomological Mackey 2-functors arise in linear representation theory and geometry and include the Green 2-functors of Examples~\ref{Exa:rep-theory}, \ref{Exa:quots} and~\ref{Exa:sheaves}.)
By \Cref{Prop:dim}, the isomorphism $i_*i^* \cong A(i) \otimes - $ identifies $\leps \circ \reta$ with multiplication by $\dim A(i)$; hence $\dim A(i)$ plays here the role of the index, even when $\cat M$ is not cohomological.
\end{Rem}

%------------------------------------------------------------------------------
\section{Tensor monadicity}
\label{sec:monadicity}%
%\bigbreak
%------------------------------------------------------------------------------

We now revisit the results of \cite{BalmerDellAmbrogioSanders15} in the context of Green 2-functors. 
In a nutshell, the projection formulas let us reformulate the (co)monadicity of the underlying Mackey 2-functors in terms of the tensor structure. 
\Cref{Thm:separable-monoidicity} provides a conceptual explanation for, and an improvement of, the various results in \cite{BalmerDellAmbrogioSanders15}, whose proofs at the time were rather \emph{ad~hoc} and dependent on each example.

\begin{Not} \label{Not:EM}
If $A=(A,\mu,\iota)$ is a monoid in a monoidal category $\cat C$, we write $A\MMod_\cat C$ for the category of (associative unital left) modules $(X, \rho\colon A\otimes X\to X)$ in~$\cat C$. The evident forgetful functor $A\MMod_\cat C\to \cat C$ admits a left adjoint functor, sending each $X\in \cat C$ to the \emph{free $A$-module} $(A\otimes X, \mu\otimes \id_X\colon A\otimes A \otimes X\to A\otimes X)$ over~$X$. Dually, for every comonoid $(B,\delta,\epsilon)$ in~$\cat C$ we have the category $B\CComod_\cat C$ of (left) comodules in~$\cat C$, and the forgetful functor $B\CComod_\cat C\to \cat C$ has a right adjoint sending $X\in \cat C$ to the \emph{cofree comodule} $(B\otimes X, \delta \otimes \id\colon B\otimes X \to B\otimes B\otimes X)$. 
\end{Not}

\begin{Thm} \label{Thm:separable-monoidicity}
Let $\cat M$ be a Green 2-functor taking values in idempotent complete additive categories. 
Let $i\colon H\to G$ be in~$\JJ$, and let $A(i) \in \cat M(G)$ be the induced Frobenius object (\Cref{Thm:Frobenius}).
Then $\cat M(H)$ is canonically equivalent to both the category of $A(i)$-comodules and that of $A(i)$-modules in~$\cat M(G)$ (\Cref{Not:EM}):
\[
\xymatrix@R=14pt{
& 
\cat M(G) 
 \ar[dd]|{\, i^*}
 \ar[ddl]|{\Cofree}
 \ar[ddr]|{\Free} & \\
 && \\
A(i) \CComod_{\cat M(G)}
 \ar@/^2ex/[uur]^-{\forget}
 \ar@/_2ex/@{..>}[uur] &
\cat M(H)
 \ar@/^2ex/[uu]^(.4){i_!}
 \ar@/_2ex/[uu]_(.4){i_*}
 \ar[l]^-{\sim}
 \ar[r]_-{\sim} & 
A(i)\MMod_{\cat M(G)}
 \ar@/_2ex/[uul]_{\forget}
 \ar@{..>}@/^2ex/[uul]
}
\]
These equivalences identify, respectively, the adjunction $i_!\dashv i^*$ with the forgetful-cofree adjunction for comodules, and the adjunction $i^*\dashv i_*$ with the free-forgetful adjunction for modules.
If the Green 2-functor $\cat M$ is braided, both are equivalences of braided monoidal categories, where modules and comodules are endowed with the usual tensor product of $\cat M(G)$ amalgamated over~$A(i)$.
\end{Thm}

A similar result holds for right modules and right comodules.

\begin{proof}
Leaving aside the tensor structures, we already get canonical equivalences
\begin{equation} \label{eq:co-monadic-version}
\xymatrix@R=2pt{
\bbM \CComod_{\cat M(G)} & \cat M(H) \ar[l]_-{\sim} \ar[r]^-{\sim} & \bbM\MMod_{\cat M(G)} \\
(i_!X, i_! \leta) & X \ar@{|->}[r] \ar@{|->}[l] & (i_*X, i_* \reps_X)
}
\end{equation}
with the Eilenberg--Moore categories of comodules over the comonad $\bbM = i_!i^*$ and of modules over the monad $\bbM= i_*i^*$, simply because $\cat M$ is a Mackey 2-functor. 
Indeed, the right-hand side equivalence is precisely \cite[Theorem~2.4.1]{BalmerDellAmbrogio20} (there only formulated for $\GG = \gpd$, but the proof works verbatim in general). 
The left-hand side one is obtained by applying the same result to the dual Mackey 2-functor $\cat M^\op$ of \cite[Remark~1.1.8]{BalmerDellAmbrogio20}, whose values are the opposite categories $\cat M(G)^\op$ and where the left and right adjunctions of $\cat M$ exchange their roles.
These equivalences identify the (co)free-forgetful adjunctions with $i_!\dashv i^*$ and $i^*\dashv i_*$, respectively.

Let us now reintroduce the tensors. 
Recall the isomorphism
$
\pi \colon A(i)\otimes (-) \overset{\sim}{\to} \bbM
$
of monads obtained in~\eqref{eq:proj-monads} from the projection formula (see also \Cref{Thm:half-braiding}\,\eqref{it:hb3}). 
It induces by precomposition an isomorphism of the respective module categories
 \[
\bbM\MMod_{\cat M(G)} \overset{\sim}{\longrightarrow} A(i)\MMod_{\cat M(G)}
 \]
sending $(X,\rho\colon \bbM X\to X)$ to $(X, \rho\pi_X\colon A(i)\otimes X\to X)$.
Dually, composition with $\pi$ induces an isomorphism
 \[
A(i)\CComod_{\cat M(G)} \overset{\sim}{\longrightarrow} \bbM\CComod_{\cat M(G)}  \,.
 \]
Combined with \eqref{eq:co-monadic-version}, this provides the two claimed equivalences of categories.

Suppose now that $\cat M$ is braided, so that the tensor categories $\cat M(G)$ and $\cat M(H)$ are braided and the strong tensor functor $i^*$ preserves the braiding. Moreover, the monoid $A(i)$ is automatically commutative. 
As $\cat M(G)$ is idempotent complete, the module category $A(i)\MMod_{\cat M(G)}$ is also idempotent complete and therefore can be equipped with the tensor product of modules over~$A(i)$, since the latter is defined by a split coequalizer (see \eg \cite[\S3]{DellAmbrogioHuglo21}). Dually, the tensor product of $A(i)$-comodules is defined as a certain split equalizer. 
We must check that the two equivalences are braided monoidal. 
Briefly: For the module case, recall (\Cref{Prop:induced-co-monoid}\,\eqref{it:lax}) that the right adjoint $i_*$ inherits a lax monoidal structure $\lambda_{X,Y}\colon i_*X \otimes i_*Y \to i_*(X \otimes Y)$ and $\iota= \eta_\Unit \colon \Unit \to i_*\Unit = A(i)$.
One can check that this structure transfers further onto the comparison functor~$E\colon \cat M(H)\overset{\sim}{\to}A(i)\MMod$ by taking the unique factorizations of $\lambda_{X,Y}$ and $\iota$ as $A(i)$-linear morphisms:
\[
\xymatrix{
EX \otimes EY \ar@{=}[r] \ar@{->>}[d] & i_* X \otimes i_* Y \ar[r]^-{\lambda_{X,Y}} & i_*(X \otimes Y) \\
EX \otimes_{A(i)} EY \ar@{-->}[rr]_{\overline{\lambda}_{X,Y}}& & E(X \otimes Y) \ar@{=}[u]
}
\quad\quad 
\xymatrix{
\Unit \ar[d]_{\eta_\Unit} \ar[r]^-{\iota} & A(i) \\
A(i) \ar@{-->}[ur]_{\overline{\iota}\,=\,\id} &
}
\]
As the induced maps $\overline{\lambda}_{X,Y}$ and $\overline{\iota}$ are isomorphisms, this is actually a strong monoidal structure on~$E$. It is clearly compatible with the braidings.

The case of comodules is dual and left to the reader.
\end{proof}

\begin{Rem} \label{Rem:sep-generality}
In \cite{BalmerDellAmbrogioSanders15} we considered families of categories $\cat M(G)$ for quite general topological groups~$G$, with precise hypotheses depending on the various examples; the tensor monadicity result, however, was (and can be) proved only for subgroups $H\leq G$ with finite discrete quotient~$G/H$. 
Our formalism of Mackey and Green 2-functors is only meant to treat finite groups, but can easily handle 2-categories of topological groupoids in order to capture the full generality of these examples, as long as one finds suitable spannable pairs~$(\GG;\JJ)$; we leave the latter task to interested readers.
However, we feel that  a better generalization to topological groupoids shouldn't be such an ``obvious'' one but rather one which also usefully describes (by new axioms) the observed behavior of subgroups with \emph{non-finite} quotient (\eg phenomena such as the Wirthm\"uller isomorphisms of \cite{BalmerDellAmbrogioSanders16} for the cases when $\omega_f\not\simeq \Unit$). This is an interesting problem that remains to be solved.
\end{Rem}

%------------------------------------------------------------------------------
\section{Examples of Green 2-functors}
\label{sec:exa}%
%\bigbreak
%------------------------------------------------------------------------------

Many of the examples of Mackey 2-functors treated in \cite[Ch.\,4]{BalmerDellAmbrogio20} are in fact symmetric Green 2-functors. Let us stress that, in all cases, the multiplicative structure is well-known, as indeed was the case for the underlying structure of Mackey 2-functor: 
Our present contribution consists in providing easily verifiable yet useful axioms for it. We now turn in this section to the ``easily verifiable'' part. 
For simplicitly, we take $\kk =\bbZ$ and leave to the reader the straightforward adjustments needed to work over other rings~$\kk$, when applicable.

A first vast family of examples comes from derivators.
By an \emph{additive monoidal derivator} we mean a monoidal derivator (\cite[\S7]{GPS14}) whose underlying derivator is additive (\ie it lands in $\ADD$ as a 2-functor on~$\Cat^\op$) and such that each  (internal or external) component of the tensor product is an additive functor of both variables.

\begin{Thm} [Additive monoidal derivators] \label{Thm:mon-der}
If $\cat D\colon \Cat^\op\to \ADD$ is an additive monoidal derivator, its restriction $\cat M:=\cat D|_{\gpd}$ is a Green 2-functor defined on $(\GG;\JJ)=(\gpd; \groupoidf)$.
If $\cat D$ is braided or symmetric then so is the Green 2-functor~$\cat M$.
Similarly, if $\cat E$ is an additive  left or right $\cat D$-module, its restriction $\cat N:=\cat E|_{\gpd}$ is a left or right $\cat M$-module.
\end{Thm}

\begin{proof}
Since $\cat D$ is in particular an additive derivator,  by \cite[Theorem~4.1.1]{BalmerDellAmbrogio20} its restriction $\cat M$ to finite groupoids is a Mackey 2-functor for $(\gpd; \groupoidf)$.
By definition, the multiplicative structure on the derivator $\cat D$ consists of a pseudomonoid structure on $\cat D$ seen as an object of the Cartesian 2-category $\2Fun(\Cat^\op, \ADD)$ of additive prederivators, such that its multiplication morphism is cocontinuous in both variables (\cite[Def.\,7.15]{GPS14}). 
Restricting along the inclusion $\gpd\hookrightarrow \Cat$, $\cat M $ is a pseudomonoid in the 2-category $\2Fun(\gpd^\op,\ADD)$ equipped with the \emph{Cartesian} tensor structure. 
However, being biadditive, the multiplication $\cat M\times \cat M\to \cat M$ extends along the (termwise) universal biadditive $\cat M \times \cat M \to \cat M \otimes \cat M$, where $\otimes$ is the diagonal tensor product of additive categories (\Cref{Exa:additive-tensor} and \Cref{Prop:tensor-pre-Green}).
The cocontinuity condition on~$\cat D$ specializes to the condition in \Cref{Prop:bimorphism}\,\eqref{it:left-mor}, so that we indeed obtain a Green 2-functor~$\cat M$. 
The discussion of braidings, symmetries and modules is similarly straightforward and left to the reader.
\end{proof}

\begin{Exa}[Model categories]
\label{Exa:models}
Let $\cat C$ be a monoidal model category in the sense of~\cite{Hovey99}.
Suppose the homotopy category $\Ho(\cat C)$ is additive, for instance because the model structure is stable. 
Then there is an additive monoidal derivator $\cat D$ whose value $\cat D(J)$ is the homotopy category of $J$-shaped diagrams for the levelwise weak equivalences (\cite[Ex.\,7.19]{GPS14}). 
Hence \Cref{Thm:mon-der} yields a Green 2-functor $\cat M = \cat D|_{\gpd}$ for $(\gpd; \groupoidf)$.
If the monoidal model category $\cat C$ is braided or symmetric then so is the Green 2-functor~$\cat M$.
\end{Exa}

Among the many concrete examples of Green 2-functors that can be derived this way, we should at least mention the following ones, for which there are well-known symmetric monoidal stable Quillen model structures getting the job done.

\begin{Exa} [Linear representations] 
\label{Exa:rep-theory}
The usual Mackey 2-functors of linear representations over a field (or commutative ring)~$\kk$ are symmetric Green 2-functors thanks to the tensor product over $\kk$ equipped with diagonal $G$-action. 
Here $\cat M(G)$ could be \eg the abelian category $\kk G\MMod$ of all representations (\cite[Ex.\,4.1.4]{BalmerDellAmbrogio20}), or its derived category $\mathrm D(\kk G)$ (\cite[Ex.\,4.1.5]{BalmerDellAmbrogio20}).
We can obtain more examples of Green 2-functors by restricting attention to full subcategories of the previous ones, as long as they are stable under restrictions, inductions and tensor products (which may depend on~$\kk$). For exemple: the category $\fgmod(\kk G)$ of finitely generated representations, or $\mathrm{perm}_\kk(G)$ of finitely generated permutation representations, or $\mathrm{proj}(\kk G)$ of finite projective modules, or the idempotent completion $\mathrm{perm}_\kk(G)^\natural$ (that is the category of trivial source modules), or categories of bounded complexes, etc. 
If $\kk$ is not a field, we may want to consider the subcategory $\cat M(G)=\latt_\kk(G)$ of lattices, \ie those $\kk G$-modules which are finitely generated projective over~$\kk$.
\end{Exa}

\begin{Exa}[Group ring spectra]
Let $\mathrm{Mod}_R$ be a symmetric monoidal stable model category of modules over some commutative ring spectrum~$R$ (for any of the highly structured meanings of ``commutative ring spectrum''). 
By specializing \Cref{Exa:models}, we get a Green 2-functor $\cat M$ such that $\cat M(G)= \Ho((\mathrm{Mod}_R)^G)= \Ho(\mathrm{Mod}_{RG})$ is the homotopy category of left modules over the group ring spectrum~$RG$. 
When $R=  H\kk$ is the Eilenberg--Mac\,Lane spectrum of a (discrete) field~$\kk$, we get back the derived category $\Der(\kk G)$ of \Cref{Exa:rep-theory}. For $R= \mathbb S$ the sphere spectrum, we get the ``very na\"ive'' $G$-spectra, \ie $G$-diagrams of spectra.
\end{Exa}

We now turn to the many Green 2-functors not arising this way from derivators. 
In some situations, we can take quotients of Green 2-functors:

\begin{Prop} \label{Prop:proj-formula-quot}
Let $\cat M$ be a Green 2-functor on any spannable pair $(\GG;\JJ)$ (\Cref{Def:quasi-Green}). 
Let $\cat N\subseteq \cat M$ be a Mackey sub-2-functor, \ie a collection of full additive subcategories $\cat N(G)\subseteq \cat M(G)$ stable under restriction and induction,
and suppose each $\cat N(G)$ is a two-sided tensor ideal in the monoidal category~$\cat M(G)$.
Then the additive quotient categories $\cat Q(G):= \cat M(G)/\cat N(G)$ inherit a unique structure of Green 2-functor~$\cat Q$ on $(\GG; \JJ)$ such that the quotient functors $\cat M(G)\to \cat Q(G)$ are monoidal.
If $\cat M$ is braided or symmetric than so is~$\cat Q$. 
\end{Prop}

\begin{proof}
A levelwise additive quotient by a Mackey sub-2-functor inherits a unique structure of Mackey 2-functor by \cite[Prop.\,4.2.5]{BalmerDellAmbrogio20}.
Similarly, since each $\cat N(G)$ is an ideal, the quotient categories and the induced restriction functors between them inherit monoidal structures from~$\cat M$ and the induced 2-isomorphisms are still monoidal, so 
that we again have a pseudomonoid in  $\2Fun(\GG^\op,\ADDick)$ as required.
Since the adjunctions and monoidal structures on $\cat Q$ are induced by those of $\cat M$, we see that the projection formulas similarly descend to~$\cat Q$. The rest is as easy.
\end{proof}

\begin{Exa}[Stable module categories] \label{Exa:quots}
Consider the archetypical symmetric Green 2-functor $\cat M\colon G\mapsto \kk G \MMod$ of linear representations over a field~$\kk$ (\Cref{Exa:rep-theory}). It is defined on $(\gpd; \gpdf)$, but we can restrict it to the non-Cartesian pair $(\gpdf;\gpdf)$ for which it is a Green 2-functor in the general sense of \Cref{Def:quasi-Green}.
Now note that projective modules form a tensor ideal $\cat N(G)=\kk G \PProj \subseteq \cat M(G)$, by the projection formula $U \otimes i_*(V)\cong i_* (i^*U \otimes V) $ for $i\colon 1\hookrightarrow G$ (and $G$ a group). Moreover, the family $\cat N(G)$ is stable under induction and restriction functors along all faithful~$i$, since $i_*$ and $i^*$ are adjoint on both sides hence are exact.
By \Cref{Prop:proj-formula-quot}, the stable module categories $\mathrm{Stab}(\kk G)= \cat M(G) / \cat N(G)$ inherit from~$\cat M$ the structure of a symmetric Green 2-functor for $(\gpdf;\gpdf)$.
(This example cannot arise directly from \Cref{Thm:mon-der} because $\cat M(1)\simeq 0$; see \cite[Rem.\,4.2.7]{BalmerDellAmbrogio20}.)
\end{Exa}

\begin{Thm}[Presheaf Green 2-functors]
\label{Thm:presheaves}
Let $\cat A$ be a cocomplete abelian category. 
Let $\cat S$ be a Green 2-functor for $(\GG;\JJ)$ whose values $\cat S(G)$ are essentially small additive categories.
Then there is a Green 2-functor $\cat M$ for $(\GG;\JJ)$ such that $\cat M(G)=\Fun_\mathrm{add}(\cat S(G)^\op, \cat A)$ is the category of additive presheaves, and such that its 2-functoriality and its tensor structure are extended from those of $\cat S$ along the Yoneda embeddings $\cat S(G)\hookrightarrow \cat M(G)$.
If $\cat S$ is braided or symmetric then so is~$\cat M$.
Moreover, we have $(i_*N)(X) \cong N(i^* X)$ for all $(i\colon H\to G)\in \JJ$, $N\in \cat M(H)$ and $X\in \cat S(G)$.
\end{Thm}

\begin{proof}
The underlying Mackey 2-functor of $\cat M$ and the formula for induction~$i_*$ are provided by \cite[Prop.\,7.3.2]{BalmerDellAmbrogio20} and its proof (which are only formulated for $(\GG;\JJ)=(\gpd;\gpdf)$, but everything applies equally to the general case). 
The tensor structure~$\odot_G$ on each $\cat S(G)$ extends to $\cat M(G)$ by the (additive) Day convolution product; the strong monoidal structure on each restriction functor~$u^*$ extends uniquely to $\cat M$ by the (2-)universal property of the coend which defines Day convolution. 
The corresponding external product $\boxdot$ on $\cat M$ is related to that on $\cat S$ by a similar coend formula; in particular, again by the universal property of coends, it suffices to check the bimorphism conditions of \Cref{Prop:bimorphism} on $\cat S$, where it holds by hypothesis. 
We conclude that this extended product on $\cat M$ defines a Green 2-functor. The remaining details are straightforward and left to the reader.
\end{proof}

\begin{Exa} [Ordinary Mackey functors] 
\label{Exa:Mack}
There exists a Green 2-functor $\cat M$ for $(\gpd; \gpdf)$ such that $\cat M(G)$ is the category $\Mackey_\kk(G)$ of  ordinary $\kk$-linear Mackey functors for~$G$, and another one such that $\cat M(G)$ is the category $\CohMackey_\kk(G)$ of cohomological Mackey functors for~$G$; the so-called box product provides both tensor structures (\cite{Bouc97}). 
In both cases, we obtain $\cat M$ by applying \Cref{Thm:presheaves} to $\cat A= \kk \MMod$ and a suitable Green 2-functor~$\cat S$. 
For the second~$\cat M$, since $\CohMackey_\kk(G)= \Fun_\mathrm{add}(\mathrm{perm}_\mathbb Z(G), \kk \MMod)$ as tensor categories by Yoshida's theorem (\cf \cite{DellAmbrogioHuglo21}), we can take $\cat S$ to be the (level-wise dual of the) Green 2-functor $\cat S\colon G\mapsto \mathrm{perm}_\mathbb Z(G)^\op$ as in \Cref{Exa:rep-theory}.
For the first~$\cat M$, we can take $\cat S(G)$ to be the additive category of spans of finite $G$-sets, which defines a Mackey 2-functor $\cat S$ by \cite[Thm.\,7.2.3]{BalmerDellAmbrogio20}; the cartesian product of $G$-sets induces a tensor structure on $\cat S(G)$ extending (by definition!) to the box product on $\cat M(G)$, and the preservation of inductions is easily checked so that we get the required Green 2-functor~$\cat S$. 
\end{Exa}

As in \cite[\S4.4]{BalmerDellAmbrogio20}, we can consider Mackey 2-functors of ``equivariant objects''. 
They are quite often equipped with the structure of a Green 2-functor:

\begin{Thm}[Equivariant objects]
\label{Thm:equiv-obj}
Let $G$ be a fixed finite groupoid, and let $\cat S\colon G^\op\to \PsMon(\ADD)$ be a pseudofunctor to the 2-category of  additive 
monoidal categories. 
Then there exists a Green 2-functor for $\GG= \JJ = \gpdG$, denoted $\Gamma(-; \cat S)$, whose value at $(H,i_H)$ is the category $\Gamma((H,i_H); \cat S):= \Gamma(H; \cat S\circ i_H^\op)$ of $H$-equivariant objects, \ie of sections of the Grothendieck fibration associated with $\cat S\circ i_H^\op\colon H^\op \to \ADD$.
If $\cat S$ takes values in braided or symmetric monoidal categories, the associated Green 2-functor inherits a braiding or symmetry accordingly. 
\end{Thm}

\begin{proof}
The underlying Mackey 2-functor for $\gpdG$ is given by \cite[Thm.\,4.4.16]{BalmerDellAmbrogio20}, applied to $\cat S$ composed with the forgetful 2-functor $\PsMon(\ADD)\to \ADD$.

W must define the pairing and show that it preserves inductions. 
We only sketch the construction, using the notations of \emph{loc.\,cit.}, leaving all (quite straightforward) verifications to the reader.  
The tensor product $(X,\varphi)\otimes (Y,\psi)$ of two objects $(X,\varphi)$, $(Y,\psi)$ of $\Gamma(H; \cat S\circ i_H^\op)$ has object-components $(X\otimes Y)_p:=X_p \otimes Y_p$ for $p\in \Obj H$, using the given internal tensor of $\cat S(p)$; and morphism-components
\[
\xymatrix{
(\varphi \otimes \psi)_g \colon X_p \otimes Y_p
 \ar[r]^-{\varphi_g \otimes \psi_g} & 
 g^*X_q \otimes g^* Y_q
  \ar[r]^-{\sim} &
 g^*(X_q \otimes Y_q)
}
\]
for all $(g\colon p\to q)\in \Mor H$, using the given strong monoidal structure of the functor~$g^*= \cat S (i_H (g))$.
The tensor product of morphisms is similarly diagonal: $(\xi \otimes \zeta)_p= \xi_p \otimes \zeta_p$.
This defines a monoidal structure on the category $\Gamma((H,i_H); \cat S)$ for each object $(H,i_H)\in \gpdG$.
For every 1-morphism $(i, \theta_i)\colon (K, i_K)\to (H,i_H)$ in $\gpdG$, the restriction functor $(i, \theta_i)^*\colon \Gamma((H,i_H); \cat S) \to \Gamma((K,i_K); \cat S)$ defined as in \emph{loc.\,cit.\ }receives a strong monoidal structure from those given on the pull-back functors $(\theta^{-1}_{i,q})^* = \cat S(\theta^{-1}_{i,q})$ (for all $q\in \Obj K$).
Similarly, the natural isomorphism induced by a 2-cell $\alpha\colon (i, \theta_i)\Rightarrow (j, \theta_j) $ is monoidal, which ultimately boils down to the natural transformation $\alpha^*= \cat S(i_H \alpha)$ being monoidal by hypothesis. 
We thus obtain a 2-functor $\Gamma(-;\cat S)\colon (\gpdG)^{\op} \to \PsMon(\ADD)$.

By \Cref{Prop:Green-lifting}, it only remains to see that the corresponding external pairing preserves inductions in both variables. 
We leave this lengthy but straightforward verification to the reader, who may compute using the explicit Kan formulas for the adjunctions (\cf \cite[Lemma 4.4.11]{BalmerDellAmbrogio20}).
\end{proof}

We can apply the latter result to equivariant sheaves in geometry:

\begin{Exa} [Equivariant sheaves] \label{Exa:sheaves}
Suppose a finite group $G$ acts on a locally ringed space~$X=(X,\mathcal O_X)$, and let 
$\Mod(X)$ be the symmetric monoidal category of sheaves of $\mathcal O_X$-modules.
Pulling back along elements $g\in G$ defines a pseudoaction of $G$ on $\Mod(X)$, \ie a pseudofunctor $\cat S\colon G^\op \to \PsMon(\ADD)$ which maps the unique (unnamed) object of $G$ to $\Mod(X)$. 
By \Cref{Thm:equiv-obj}, we get a symmetric Green functor $\cat M$ on $\gpdG$ whose value at a subgroup $H\leq G$ is the category $\Mod(X/\!\!/H)$ of $H$-equivariant sheaves of $\mathcal O_X$-modules.
Many variations are possible, by replacing $\Mod(X)$ with other monoidal categories. For instance, we may consider constructible sheaves over some commutative ring, or~-- if $X$ is a noetherian scheme~-- coherent or quasicoherent $\mathcal O_X$-modules. We may also use chain complexes, or (variously bounded) homotopy or derived categories. 
See \cite[\S4.4]{BalmerDellAmbrogio20} for details.
\end{Exa}

Let us sketch some more examples, leaving the details to interested readers.

\begin{Exa} [Equivariant stable homotopy] \label{Exa:SH}
As in \cite[Ex.\,4.3.8]{BalmerDellAmbrogio20}, there is a Mackey 2-functor $\cat M$ for $(\gpd;\gpdf)$ such that $\cat M(G)= \Ho(\textrm{Sp}^G)$ is the homotopy category of genuine $G$-equivariant spectra. The latter is a symmetric monoidal category by the smash product of $G$-spectra and the restriction functors $i^*$ are strong symmetric monoidal. This structure extends to a bimorphism $(\cat M,\cat M)\to \cat M$ turning $\cat M$ intro a symmetric Green functor.
\end{Exa}

\begin{Exa} [Equivariant Kasparov theory] \label{Exa:KK}
As in \cite[Ex.\,4.3.9]{BalmerDellAmbrogio20}, there is a Mackey 2-functor $\cat M$ for $(\gpd;\gpdf)$ (in fact for $(\gpd;\gpd)$!) such that $\cat M(G)= \KK^G$ is the $G$-equivariant Kasparov category of separable complex $G$-C*-algebras. The latter is a symmetric monoidal category by the minimal tensor product of C*-algebras equipped with the diagonal $G$-action. This structure extends to a bimorphism which turns $\cat M$ into a symmetric monoidal Green functor.
\end{Exa}

\begin{Rem}
We suspect many Green 2-functors to arise from $G$-equivariant ring spectra (\cf \cite{LewisMandell06}) and/or from suitable $E_\infty$-Green functors in the sense of Barwick \emph{et\,al.\ }\cite{Barwick17} \cite{BGS20}, by taking module categories of their values at all $H\leq G$. 
In any case, it would be interesting to make precise the way(s) that the categorification of Green functors due to Barwick \emph{et\,al.\ }is related to our own.
\end{Rem}

\begin{Rem}
All concrete examples of Green functors we have explicitly mentioned happen to be symmetric. 
Nonetheless, it is easy to find Green functors which are braided but not symmetric, or just ``plain'', for instance those  in~\Cref{Exa:models} deriving from suitably chosen monoidal model categories. 
\end{Rem}

%------------------------------------------------------------------------------
\section{Green 1-functors for general $(\GG;\JJ)$}
\label{sec:1-Green}%
%\bigbreak
%------------------------------------------------------------------------------

In this section we revisit the notion of (ordinary) \emph{Green functor} by introducing a general definition which specializes to all variants in use for finite groups (\Cref{Exas:1-Green}). We base our definition (\Cref{Def:Green-fun}) on the original view as ``levelwise ring with Frobenius formulas'' and verify that, when the pair $(\GG;\JJ)$ is Cartesian, these are just monoids for the induced monoidal structure (\Cref{Prop:1-Green-equivs}).

We begin by recalling the following flexible definition of ordinary Mackey functor.
Let $(\GG; \JJ)$ be a spannable pair (\Cref{Hyp:spannable-pair}) and fix a commutative ring~$\kk$.

\begin{Def}[{\cite{BalmerDellAmbrogio20} \cite{DellAmbrogio21ch}}]
\label{Def:Mackey-fun}
A \emph{Mackey functor for $(\GG;\JJ)$} is an additive functor
\[
M\colon \pih \Span (\GG;\JJ) \longrightarrow \kk\MMod
\]
taking values in $\kk$-modules and defined on the semiadditive (truncated ordinary) \emph{span category} of $(\GG;\JJ)$ (see \cite[Def.\,2.5.1]{BalmerDellAmbrogio20} or \cite[\S2.3]{DellAmbrogio21ch}). 
Concretely this means that $M$ assigns  to every $G\in \Obj \GG$ a $\kk$-module $M(G)$, to every 1-morphism $u\colon H\to G$ in $\GG$ a $\kk$-linear homomorphism $u^\sbull\colon M(G)\to M(H)$ and, if $u\in \JJ$, also a homomorphism $u_\sbull\colon M(H)\to M(G)$, satisfying the following four axioms:
\begin{enumerate}[\rm(a)]
\item
\label{it:1-Mackey-fun}%
\emph{Functoriality:}
We have $\id^\sbull=\id_\sbull=\id$, and for composable $K\xto{v} H\xto{u}G$ we have $(u\circ v)^\sbull=v^\sbull\circ u^\sbull$ and  when $u,v\in \JJ$ also $(u\circ v)_\sbull=u_\sbull\circ v_\sbull$.
\item
\label{it:1-Mackey-iso}%
\emph{Isomorphism invariance:} Whenever $u\simeq v$ are isomorphic in $\GG$, we have $u^\sbull = v^\sbull$ and when $u,v\in \JJ$ also $u_\sbull = v_\sbull$.
\item
\label{it:1-Mackey-add}%
\emph{Additivity:}
every coproduct $G\overset{i}{\to} G\sqcup H \overset{\;j}{\gets} H$ in~$\GG$ induces an isomorphism
$(i^\sbull, j^\sbull)^t \colon M(G\sqcup H) \overset{\sim}{\to} M(G)\oplus M(H)$ with inverse $(i_\sbull , j_\sbull)$ (hence $M(\emptyset) \cong 0$).
\item
\label{it:1-Mackey-formula}%
\emph{Mackey formula:}
For every Mackey square in~$\GG$ with $i$ and $q$ in~$\JJ$
\begin{equation*}
\vcenter{
\xymatrix@C=10pt@R=10pt{
& P \ar[ld]_-{p} \ar[rd]^{q} \ar@{}[dd]|{\isocell{\gamma}} & \\
H \ar[rd]_-{i} &  & K \ar[ld]^-{u} \\
& G &
}}
\end{equation*}
we have the equality $u^\sbull\circ i_\sbull=q_\sbull\circ p^\sbull$ of maps $M(H)\to M(K)$.
\end{enumerate}
As in \cite{DellAmbrogio21ch}, we denote by
\[ \Mackey_\kk(\GG;\JJ) := \Fun_\mathrm{add}(\pih \Span (\GG;\JJ), \kk\MMod) \]
the $\kk$-linear category of Mackey functors and natural transformations. Note that a natural transformation $\alpha \colon M\to N$ is the same thing as a family of maps $\alpha_G\colon M(G)\to M(G)$ ($G\in \Obj \GG$) which is compatible with respect to the restriction maps~$u^\sbull$ for all $u\colon H\to G$ and, for $u\in \JJ$, also with the induction maps~$u_\sbull$.
\end{Def}

We can now introduce a corresponding general notion of Green functor:

\begin{Def} [{Green 1-functor for~$(\GG;\JJ)$}]
\label{Def:Green-fun}
Let $M,N,L$ be Mackey functors for the spannable pair $(\GG;\JJ)$.
We define a \emph{Dress\footnote{Compare with the definition of ``pairing'' given in \cite[p.\,195]{Dress73}.} pairing} $\beta\colon (M,N)\to L$ to be a family of $\kk$-linear morphisms $\beta=\beta_G\colon M(G)\otimes_\kk N(G)\to L(G)$ for $G\in \Obj \GG$ satisfying:
\begin{enumerate}[\rm(a)]
\item
\label{it:1-Mackey-ring}%
Naturality for restrictions: $u^\sbull (\beta (x , y)) = \beta (u^\sbull(x), u^\sbull(y))$ for all $(H\overset{u}{\to} G)\in \GG$.
\item
\label{it:1-Mackey-Frob}%
The \emph{Frobenius formulas}  (or \emph{projection formulas})  $\beta( x , u_\sbull(y))=u_\sbull (\beta( u^\sbull (x) , y))$ and $\beta( u_\sbull(y), x) = u_\sbull \beta(y , u^\sbull (x))$ for all $x\in M(G)$ and $y\in M(H)$, when $u\in \JJ$.
\end{enumerate}
A \emph{Green functor for $(\GG;\JJ)$} is then a Mackey functor $M$ for $(\GG;\JJ)$ together with a Dress pairing $\smash{(M,M)\overset{\cdot}{\to} M}$ which is associative and unital, \ie one which turns each $M(G)$ into an associative unital $\kk$-algebra.
Similarly, a left (or right) module $N$ over the Green functor $M$ is a Mackey functor $N$ together with a pairing $(M,N)\to N$ (resp.\ $(N,M)\to N$) which turns each $N(G)$ into a left (or right) $M(G)$-module.
\end{Def}

The following reassuring result is well-known at least in some cases (\cf \cite[Prop.\,1.4]{Lewis80} or \cite[\S2]{Bouc97}), but we provide the general proof for completeness. 

\begin{Prop} \label{Prop:1-Green-equivs}
Suppose the pair $(\GG;\JJ)$ is Cartesian. Then:
\begin{enumerate}[\rm(1)]
\item  \label{it:Day-convo}
The products of $\GG$ induce an additive symmetric monoidal structure  on the span category $\pih \Span (\GG;\JJ)$ and thus, by Day convolution with respect to~$\otimes_\kk$, a $\kk$-linear closed symmetric monoidal structure on $\Mackey_\kk(\GG;\JJ)$.
\item  \label{it:corr-pairing}
The data of a Dress pairing $\beta\colon (M,N)\to L$ for $(\GG;\JJ)$ (as in \Cref{Def:Green-fun}) is equivalent to that of a morphism $\overline{\beta}\colon M\otimes N\to L$ in the tensor category of~\eqref{it:Day-convo}.
\item  \label{it:corr-Green}
Under the correspondence of~\eqref{it:corr-pairing}, a (commutative) Green functor corresponds to a (commutative) monoid in $\Mackey_\kk(\GG;\JJ)$, and similarly for modules.
\end{enumerate}
\end{Prop}

\begin{proof}
It is straightforward to check that the products $G\times H$ of $\GG$ extend to spans to define a symmetric monoidal structure on $\pih \Span (\GG;\JJ)$; it is additive in both variables since finite products distribute over finite coproducts (because $\GG$ is extensive by hypothesis, and reasoning as in \cite[Prop.\,2.5]{CLW93}; \cf \cite{PanchadcharamStreet07}).
We can then apply $\kk$-linear Day convolution (see \eg \cite[\S6.2]{Loregian21}) to get the tensor structure~$\otimes$ on $\Mackey_\kk(\GG;\JJ)$ as claimed in part~\eqref{it:Day-convo}. 

Note that, by the universal property of Day convolution as a coend, a morphism $\overline{\beta} \colon M\otimes N\to L$ in $\Mackey_\kk(\GG;\JJ)$ is uniquely determined by the $\kk$-linear maps
\[ 
\tilde \beta_{G,H} \colon 
\xymatrix{ 
M(G) \otimes_\kk N(H)
 \ar[r]^-{\textrm{can.}} &
 (M \otimes N)(G\times H)
  \ar[r]^-{\overline{\beta}_{G\times H}} &
  L(G\times H)
} 
\]
 natural in $G,H\in \pih \Span (\GG;\JJ)$, obtained by precomposing $\overline{\beta}$ with the coend structure map. 
Moreover, $\overline{\beta}$ is the associative and unital structure map of a monoid or module precisely when the family $\tilde \beta = \{\tilde \beta_{G,H}\}$ is associative and unital in an evident corresponding sense (\eg the case of a monoid $M=N=L$ yields precisely a lax monoidal structure on the functor~$M$, and the general case is very similar).

To prove part~\eqref{it:corr-pairing}, it will therefore suffice to see that the data of such a natural family~$\tilde \beta$ is equivalent to that of a Dress pairing~$\beta$.
Indeed, we get a bijective correspondence $\tilde \beta \leftrightarrow \beta$ just as in \Cref{Rem:mate-corr-pairing}, by restricting along diagonals and projections: 
Given $\tilde \beta$, we construct $\beta = \{\beta_G\}_G$ as the composite
\[
\beta_{G} \colon 
\xymatrix{ M(G) \otimes_\kk N(G)
 \ar[r]^-{\tilde \beta_{G,G}} &
 L(G\times G)
  \ar[r]^-{\delta^\sbull_G} &
  L(G)},
\]
whereas if we are given~$\beta$, we get the following family of maps:
\[
\tilde \beta_{G,H}\colon 
\xymatrix@C=28pt{
M(G) \otimes_\kk N(H) 
 \ar[r]^-{\pr^\sbull_1 \otimes \pr^\sbull_2} &
L(G \times H) \otimes_\kk N(G\times H)
 \ar[r]^-{\beta_{G \times H}} &
 L(G \times H)
}.
\]
Given a natural family~$\tilde \beta$, we must check that the family $\beta$ forms a Dress pairing. 
And indeed, $\beta$ is clearly natural with respect to restrictions (though usually not for inductions!), and it satisfies the two projection formulas by a simpler version of the proof of \Cref{Thm:proj-formula}. (Just draw in $\kk\MMod$ the analog of that proof's last diagram, which will commute by the naturality of $\tilde \beta$ for both restrictions and inductions together with the Mackey formula for the square~\eqref{eq:Mackey-for-proj}.) 

Conversely, if we are given a Dress pairing $\beta$ we must show that the corresponding~$\tilde \beta$ is natural in both variables and with respect to both restriction and induction maps. 
The naturality for restrictions follows immediately from that of~$\beta$.
The naturality with respect to inductions, say for two 1-cells $i\colon G'\to G$ and $j\colon H'\to H$ in~$\JJ$, follows by a simpler version of the proof of \Cref{Thm:proj-implies-bimorphism}.
(Just draw in $\kk\MMod$ the analog of that proof's large diagram, which will commute by the two Frobenius formulas for~$\beta$, the functoriality of induction and restriction, and the Mackey formula for the two squares~\eqref{eq:squares-quasi-Green}.)
This concludes the proof of part~\eqref{it:corr-pairing}.

For part~\eqref{it:corr-Green}, in view of the remarks made at the beginning, it suffices to check that a family $\tilde \beta$ is associative and unital precisely when the corresponding Dress pairing is. This is a straightforward verification we leave to the reader.
\end{proof}

\begin{Exas}[{See \cite{DellAmbrogio21ch} for more examples and details}]
\label{Exas:1-Green}
If $(\GG;\JJ)$ is $(\gpdf;\gpdf)$, $(\gpd; \gpdf)$ or $(\gpdG;\gpdG)$, as in \Cref{Exas:GGJJ}, the associated Mackey and Green functors in our sense are the classical notions of, respectively: globally defined Mackey and Green functors (a.k.a.\ bifree biset functors \cite{Bouc10}), globally defined inflation functors (a.k.a.\ right-free biset functors), and Mackey and Green functors for a fixed group~$G$.
With $\GG= \JJ = \gpd$ we get the variant of global Mackey and Green functors studied by Ganter \cite{Ganter13pp} and Nakaoka \cite{Nakaoka16} \cite{Nakaoka16a}, which also allows ``induction'' along non-faithful functors (\eg deflations). The category of the latter Mackey functors contains that of all biset functors as a reflective (tensor-ideal) subcategory; see \cite{DellAmbrogioHuglo21}.
\end{Exas}

\begin{Rem} \label{Rem:misalignment}
Let us review the analogies between Green 1-functors and Green 2-functors.
The more general definition of a Green 2-functor via the projection formulas (\Cref{Def:quasi-Green}) categorifies the original point of view on Green 1-functors (\Cref{Def:Green-fun}), while the definition via bimorphisms (\Cref{Def:2Green}) valid in the common situation of Cartesian pairs corresponds to viewing a Green 1-functor as a lax monoidal functor, \ie as the data $\tilde \beta$ in the proof of \Cref{Prop:1-Green-equivs}. 
 
In order to complete this picture, we should also be able to view Green 2-functors (for a fixed Cartesian pair) as pseudomonoids in a suitable tensor 2-category $\bickMack$ of Mackey 2-functors. 
For the Cartesian pairs in \Cref{Exas:GGJJ} where $\JJ$ consists of \emph{faithful} functors, at least (which unfortunately excludes the pair~$(\gpd;\gpd)$), we provided in \cite{BalmerDellAmbrogio21} a biequivalence $\bickMack \simeq \PsFun_\sqcup(\Spanhat, \ADD^\ic_\kk)$ for a certain bicategory $\Spanhat $ of \emph{Mackey 2-motives}, which contains the ``usual'' span bicategory $\Span$ but has more 2-cells. 
We could thus use Day convolution for monoidal bicategories (\cite{Corner19}) to put a tensor structure on $\bickMack$, but first we need to check that the Cartesian product of $\GG$ extends to a symmetric monoidal structure on $\Spanhat$. 
We conjecture this to be possible by extending the arguments in \cite{Hoffnung11pp}.
We do not pursue here this motivic approach, though elegant, as it would require yet another layer of machinery for a (so far) small return in applications.
\end{Rem}

%------------------------------------------------------------------------------
\section{The origins of classical Green functors}
\label{sec:origins}%
%\bigbreak
%------------------------------------------------------------------------------

In this final section we explain how ordinary Green functors can be obtained from Green 2-functors by several procedures deserving the name of ``decategorification''. We then show how the classical Green functors (as can be found \eg in~\cite{Webb00}) arise this way from the natural examples of Green 2-functors discussed in \Cref{sec:exa}. As before, $(\GG;\JJ)$ denotes a (almost always) Cartesian spannable pair (Hypotheses~\ref{Hyp:spannable-pair} and~\ref{Hyp:prods}) and $\kk$ the chosen commutative ring of coefficients.

\begin{Cons}[K-decategorification; see {\cite[Ch.\,2.5]{BalmerDellAmbrogio20}}]
\label{Cons:K-decat}
The ordinary $\mathrm K_0$ of additive categories defines a functor $\Kadd \colon \pih \Add \to \Ab$, which to a small additive category $\cat A\in \Add$ associates the Grothendieck group $\Kadd(\cat A)= \mathrm K_0(\cat A, \oplus, 0)$ and to an additive functor $F\colon \cat A \to \cat B$ the induced homomorphism $\Kadd(\cat A)\to \Kadd(\cat B)$; with isomorphic functors inducing the same homomorphism. 
By \cite[\S5.2]{BalmerDellAmbrogio20}, every Mackey 2-functor $\cat M\colon \GG^\op\to \Add$ for $(\GG;\JJ)$ and $\kk=\mathbb Z$ admits a canonical additive extension $\widehat{\cat M}\colon \Span(\GG;\JJ)\to \Add$ to the span bicategory. 
Then the composite map
\[ 
M := \Kadd \circ \pih(\widehat{\cat M}) \colon \pih \Span(\GG;\JJ) \to \pih \Add \to \Ab
\] 
is additive, \ie is an ordinary Mackey functor for $(\GG;\JJ)$ as in \Cref{Def:Mackey-fun}.
\end{Cons}

\begin{Rem}[Variant K-decategorifications]
\label{Rem:var-K-decat}
If the value-categories of $\cat M$ are equipped with further structure, we may prefer to use a more pertinent version of $\mathrm K_0$-groups. For instance, if $\cat M$ lands in exact (\eg abelian) categories and exact functors, or in triangulated categories and triangulated functors, we could use $\Kex$ or $\Ktr$ respectively, the $\mathrm 
K_0$-group of exact or triangulated categories.
\end{Rem}

\begin{Rem}[Higher K-theory]
\label{Rem:exhigher}
If $\cat M$ takes values in exact categories, we may consider Quillen's higher algebraic K-theory of exact categories $\cat E\mapsto \Kexhigher(\cat E) = \pi_* \Omega B Q\cat E$.
Since the $Q$-construction $Q\cat E$ is functorial for exact functors, and since isomorphic functors induce the same map on K-theory groups, precisely the same argument as in \Cref{Cons:K-decat} yields a Mackey functor $G\mapsto \Kexhigher(\cat M(G))$.
\end{Rem}

\begin{Cons}[Hom-decategorification; see {\cite[\S3]{BalmerDellAmbrogio21pp}}]
\label{Cons:Hom-decat}
Let $\cat M$ be a Mackey 2-functor for $(\GG;\JJ)$, and suppose $\{X_G , Y_G, \lambda_u, \rho_u\}_{G,u}$ is a given \emph{coherent family of pairs of objects in~$\cat M$}. 
It consists of objects $X_G, Y_G \in \cat M(G)$ for $G\in \Obj \GG$ and identifications $\lambda_u\colon X_H\overset{\sim}{\to} u^*X_G$  and $\rho_u\colon u^*Y_G\overset{\sim}{\to} Y_H$ for every 1-morphism $u\colon H\to G$ of~$\GG$, suitably compatible with the 2-morphisms of~$\GG$.
Then there exists an ordinary Mackey functor $M= M(\{X_G , Y_G, \lambda_u, \rho_u\}_{G,u})$ for $(\GG;\JJ)$ with value
\[
M(G) = \Hom_{\cat M(G)}(X_G, Y_G)
\]
at $G\in \Obj \GG$. 
Restriction maps $u^\sbull = M([G \overset{u}{\gets} H = H]) \colon M(G)\to M(H)$ and transfer maps $u_\sbull = M([H = H \overset{u}{\to} G])\colon M(H)\to M(G)$ (if $u\in \JJ$)  are defined by 
\[
u^\sbull(\varphi) := \rho_u \circ u^*(\varphi) \circ \lambda_u 
\quad\textrm{and}\quad
u_\sbull(\psi) := \leps \circ u_*(\rho^{-1}_u\psi\lambda_u^{-1}) \circ \reta
\]
respectively, for all $\varphi\in \cat M(G)(X_G,Y_G)$ and $\psi \in \cat M(H)(X_H,Y_H)$.
\end{Cons}

We are now ready for the last three abstract results of our theory, which specify three independent (but occasionally overlapping) methods of squeezing ordinary Green functors from a Mackey or Green 2-functor.

\begin{Thm}[Green functors via K-decategorification]
\label{Thm:Green-K-decat}
Let $\cat M$ be a Green 2-functor for~$(\GG;\JJ)$ and suppose its value categories $\cat M(G)$ are all small. Then:
\begin{enumerate}[\rm(1)]
\item \label{it:K-decat-ring}
 The K-decategorification $M = \Kadd\circ \pih (\widehat{\cat M})$ of~$\cat M$ (\Cref{Cons:K-decat}) is an ordinary Green functor for~$(\GG;\JJ)$, commutative if $\cat M$ is braided. As a ring, $M(G)$ is just the $\mathrm K_0$-ring of $\cat M(G)$ with multiplication induced by its internal tensor product. 
\item \label{it:K-decat-module}
Similarly, if $\cat N$ is a left (or right) module over $\cat M$ (\Cref{Def:actions}), its K-decat\-egorif\-icat\-ion $N$ is a left (or right) module over the Green functor~$M$ of~\eqref{it:K-decat-ring}. 
\end{enumerate}
Moreover, if $\cat M $ takes values in exact or triangulated categories and if the tensor products of $\cat M$ preserve exact sequences or triangles in both variables, the same conclusion holds for the K-decategorifications using $\Kex$, $\Kexhigher$ or $\Ktr$ instead (as in Remarks~\ref{Rem:var-K-decat} and~\ref{Rem:exhigher}).
\end{Thm}

\begin{proof}
We know from \cite[Prop.\,2.5.5]{BalmerDellAmbrogio20} that the K-decat\-eg\-orif\-ication of a Mackey 2-functor is an ordinary Mackey functor (the proof---essentially recalled in \Cref{Cons:K-decat}---immediately adapts to the exact and triangulated variations, as well as to higher K-theory as in \Cref{Rem:exhigher}). 
By hypothesis the tensor structure of a Green 2-functor preserves directs sums, respectively exact sequences or triangles, in both variables. It follows that it turns each K-group into a ring and each restriction map into a ring morphism; and similarly for the actions on modules.
Evidently the rings are commutative if $\cat M$ is braided.
Finally, the Frobenius formulas \eqref{it:1-Mackey-Frob} of \Cref{Def:Green-fun} follow immediately from \Cref{Thm:proj-formula}. 
Hence we get ordinary Green functors as claimed. (See \cite[IV\,\S6]{Weibel13} on how to induce parings in higher K-theory of exact categories.)
This settles part~\eqref{it:K-decat-ring}, and part~\eqref{it:K-decat-module} is similar.
\end{proof}

\begin{Thm}[Green functors via Hom-decategorification, I]
\label{Thm:Green-End-decat}
Let $\cat M$ be any Mackey (!) 2-functor for~$(\GG;\JJ)$. Then:
\begin{enumerate}[\rm(1)]
\item \label{it:Hom-decat-ring}
The Hom-decategorification $M=M(\{X_G , X_G, \lambda_u, \lambda^{-1}_u\}_{G,u})$ of $\cat M$ for any coherent choice of objects such that $X_G= Y_G$ and $\rho_u= \lambda_u^{-1}$ (\Cref{Cons:Hom-decat}) is a Green 2-functor for~$(\GG;\JJ)$. 
As a ring, each $M(G)$ is just $\End_{\cat M(G)}(X_G)$.
The resulting Green functors are commutative if the endomorphism ring of each $X_G$ is commutative. 
\item \label{it:Hom-decat-module}
Similarly, the Hom-decategorification for any coherent choice $\{X_G , Y_G, \lambda_u, \rho_u\}$ of pairs of objects in~$\cat M$ is a left module over $M(\{Y_G , Y_G, \rho^{-1}_u, \rho_u\}_{G,u})$ and a right module over $M(\{X_G , X_G, \lambda_u, \lambda^{-1}_u\}_{G,u})$, the latter two being the endomorphism Green functors as in~\eqref{it:Hom-decat-ring}.
\end{enumerate}
\end{Thm}

\begin{proof}
We know from \cite[Theorem\,3.7]{BalmerDellAmbrogio21pp} that the Hom-decategorification of any Mackey 2-functor $\cat M$ at any coherent family of pairs is a Mackey functor. The product and action pairings of the claims, which are simply given by composition of morphisms, are clearly preserved by the restriction maps $u^\sbull$ of these Mackey functors because we suppose $\lambda$ and~$\rho$ to be mutual inverses. Hence the only remaining doubts concern the Frobenius formulas. 

Let $u\colon H\to G$ be in~$\JJ$. 
We only verify the first formula $x \cdot u_\sbull (y) = u_\sbull ( u^\sbull (x)\cdot y)$ in the ring case~\eqref{it:Hom-decat-ring}, as the proofs of the second formula or for modules~\eqref{it:Hom-decat-module} are similar. Thus let $x\in M(G)= \End_{\cat M(G)}(X_G)$ and $y \in M(H)= \End_{\cat M(H)}(X_H)$, and consider the following diagram in~$\cat M(G)$ (where we write $\rho:=\rho_u= \lambda_u^{-1}$):
\[
\xymatrix{
&& && \\
X_G
 \ar[d]_{\reta}
 \ar[rr]^-{u_\sbull (y)}
 \ar@/^6ex/[rrrr]^-{u_\sbull( u^\sbull(x) \cdot y)}  &&
X_G
 \ar[rr]^-x && X_G \\
u_*u^* X_G
 \ar[rr]^-{u_*(\rho^{-1} y \rho)}
 \ar@/_6ex/[rrrr]_-{u_*( u^*(x) \cdot \rho^{-1} y \rho)} && 
u_*u^* X_G
 \ar[u]^{\leps}
 \ar[rr]^-{u_*u^*(x)} && 
u_*u^* X_G
 \ar[u]_{\leps} \\
&& &&
}
\]
The right square commutes by naturality of $\leps$, the left square by definition of~$u_\sbull(y)$, and the lower triangle by functoriality of~$u_*$.
Since the bottom curved arrow is also equal to $u_*( \rho^{-1} \rho \,u^*(x)\, \rho^{-1} y \rho) = u_*( \rho^{-1} ( u^\sbull (x) \, y) \rho)$, the outer square commutes by definition of $u_\sbull (u^\sbull (x)\cdot y)$. We deduce from all this that the top triangle also commutes, proving the claimed formula.

All remaining details are straightforward and are left to the reader.
\end{proof}

\begin{Thm} [Green functors via Hom-decategorification, II]
\label{Thm:Hom-decat-co-mon}
Let $\cat M$ be a Green 2-functor for~$(\GG;\JJ)$. Then:
\begin{enumerate}[\rm(1)]
\item
\label{it:Hom-decat-ring-bis}
 Suppose $\{X_G , Y_G, \lambda_u, \rho_u\}_{G,u})$ is a coherent family of pairs in $\cat M$ such that: Each $X_G$ is a comonoid  in~$\cat M(G)$, each $Y_G$ a monoid, each $\lambda_u$ a comonoid map and each $\rho_u$ a monoid map. Then its Hom-decategorification $M$ (as in \Cref{Cons:Hom-decat}) is a Green functor for~$(\GG;\JJ)$ where the multiplication on $M(G)=\Hom_{\cat M(G)}(X_G,Y_G)$ is the associated convolution product.
\item
\label{it:Hom-decat-module-bis}
The Mackey functor induced by a family whose objects are left (resp.\ right) modules and comodules over the monoids and comonoids of part~\eqref{it:Hom-decat-ring-bis} and whose coherence maps are morphisms of such, is a left (resp.\ right) module over the Green functor of~\eqref{it:Hom-decat-ring-bis}. 
\end{enumerate}
\end{Thm}

\begin{proof}
As before we verify~\eqref{it:Hom-decat-ring-bis} and leave the similar~\eqref{it:Hom-decat-module-bis} to the reader.

Write $(X_G,\delta_G,\epsilon_G)$ and $(Y_G,\mu_G,\iota_G)$ for the given comonoid and monoid structures.
By definition, the convolution product of two elements $x,y\in M(G)$ is $x\cdot y = \mu_G\circ  (x\otimes y)\circ \delta_G$.
For any $u\colon H\to G$, the restriction map $u^\sbull\colon M(G)\to M(H)$ preserves the convolution product by the commutativity of the following diagram:
\[
\xymatrix@C=12pt@R=8pt{
& && && && & \\
& u^*X_G
 \ar@/^6ex/[rrrrrr]^-{u^*(x \cdot y)}
 \ar[rr]_-{u^*\delta_G} &&
u^*(X_G \otimes X_G)
 \ar[rr]_-{u^*(x\otimes y)} &&
u^*(Y_G\otimes Y_G) 
 \ar[rr]_-{u^*\mu_G} &&
u^*Y_G
 \ar[dr]^{\rho_u} &  \\
X_H 
 \ar[ur]^{\lambda_u}
  \ar[dr]_{\delta_H} & && && && & Y_H \\
& X_H\otimes X_H
 \ar@/_6ex/[rrrrrr]_-{u^\sbull (x) \otimes u^\sbull (y)}
 \ar[rr]^-{\lambda_u\otimes \lambda_u} &&
u^*X_G\otimes u^*X_G
 \ar[uu]^{\mathrm{str}}_{\simeq}
 \ar[rr]^-{u^*x \otimes u^*y} &&
u^*Y_G \otimes u^*Y_G
  \ar[uu]^{\mathrm{str}}_{\simeq}
  \ar[rr]^-{\rho_u\otimes \rho_u} &&
u^*Y_G 
 \ar[ur]_{\mu_H} & \\
 & && && && &
}
\]
Note that the left and right pentagons commute because $\lambda_u$ is a comonoid morphism and $\rho_u$ is a monoid morphism. 
A similar diagram shows that $u^\sbull$ preserves the unit, that is $u^\sbull ( \iota_G \epsilon_G)= \iota_H \epsilon_H$. 
It remains to verify the Frobenius formulas; we only check $x \cdot u_\sbull (y) = u_\sbull ( u^\sbull (x)\cdot y)$ and omit the analogous verification of the other one. 
For any two $x\in M(G),y\in M(H)$, observe the following diagram, where $\rproj^1$ is the first right projection map of \Cref{Def:Frobs}:
\begin{equation} \label{eq:central-part}
\vcenter{
\xymatrix@C=10pt@L=6pt{
X_G \otimes X_G
 \ar[d]_\reta
 \ar[r]^-{\id \otimes \reta} &
X_G \otimes u_*u^* X_G
 \ar[d]_{\rproj^1}^\simeq
 \ar[rr]^-{x \otimes u_*(\rho^{-1}_u y \lambda^{-1}_u)} &&
Y_G \otimes u_*u^* Y_G
 \ar[d]_{\rproj^1}^\simeq
 \ar[r]^-{\id \otimes \leps} &
Y_G \otimes Y_G
  \\
u_*u^*(X_G \otimes X_G) & 
u_*(u^* X_G \otimes u^* X_G)
 \ar[l]^-{u_*(\mathrm{str})}_-\simeq
 \ar[rr]_-{u_*(u^*x \otimes \rho_u^{-1}y \lambda_u^{-1})} &&
u_*(u^*Y_G \otimes u^*Y_G)
 \ar[r]^-\simeq_-{u_*(\mathrm{str})}
 \ar@{..>}@<-4ex>[u]_{\lproj^1} &
u_*u^*(Y_G \otimes Y_G)
 \ar[u]_{\leps}
}
}
\end{equation}
It is commutative: The central square commutes by the naturality of~$\rproj^1$.
The left square commutes because it is the perimeter of the following diagram,
\[
\xymatrix{
X_G \otimes X_G
 \ar[d]_{\reta}
 \ar[rr]^-{\id \otimes \reta} &&
X_G \otimes u_*u^* X_G
 \ar[d]_{\reta}
 \ar@{}[dddr]^{\textrm{def.}}
 \ar `r[dr] `[ddd]^{\rproj^1_{X_G, u^*X_G}}   [ddd] & \\
u_*u^*(X_G \otimes X_G)
 \ar[rr]^-{u_*u^*(\id \otimes \reta)} &&
u_*u^*(X_G \otimes u_*u^* X_G) & \\
u_*(u^*X_G \otimes u^*X_G)
 \ar[u]^{u_*(\mathrm{str})}_\simeq
 \ar[rr]^-{u_*(\id \otimes u^* \reta)}
 \ar@{=}@/_3ex/[drr] &&
u_*(u^*X_G \otimes u^*u_* u^* X_G )
 \ar[u]^{u_*(\mathrm{str})}_\simeq
 \ar[d]_{u_*(\id \otimes \reps)} & \\
&& u_*(u^*X_G \otimes u^*X_G) &
}
\]
which commutes by: the naturality of $\reta$ (for the top square), the naturality of the strong monoidal isomorphism of~$u^*$  (the middle square), and a unit-counit relation for~$i^*\dashv i_*$ (the bottom triangle).
The right square in \eqref{eq:central-part} commutes by a similar diagram, after replacing $\rproj^1$ with $(\lproj^1)^{-1}$ (see \Cref{Thm:proj-formula}).
Thus \eqref{eq:central-part} is commutative, as claimed, and we are now ready to contemplate our last diagram:
\[
\xymatrix@!C@C=-34pt@R=22pt{
& & & && & & & \\
& X_G \otimes X_G
 \ar[rr]^-{\id\otimes \reta}
 \ar[dr]^\reta  &&
X_G \otimes u_*u^*X_G
 \ar[rr]^-{\underset{\phantom{bla}}{x \otimes u_*(\rho^{-1}y \lambda^{-1})}} &&
Y_G \otimes u_*u^*X_G
 \ar[rr]^-{\id \otimes \leps} && 
Y_G \otimes Y_G
 \ar[dr]^{\mu} & \\
X_G 
 \ar `u[uur] `[rrrrrrrr]^-{x\cdot u_\sbull(y)} [rrrrrrrr]
  \ar `d[dddr] `[rrrrrrrr]_-{u_\sbull(u^\sbull(x)\cdot y)} [rrrrrrrr]
 \ar[ur]^\delta
 \ar[dr]_\reta && 
u_*u^*(X_G \otimes X_G)
 \ar@{}[rrrr]^-{\eqref{eq:central-part}} &&&& 
u_*u^*(Y_G \otimes Y_G)
 \ar[ur]^\leps
 \ar[dr]_{u_*u^* \mu} & & Y_G \\
& u_*u^* X_G 
 \ar[ur]_{u_*u^* \delta} && 
u_*(\!u^*\!X_G \!\otimes\! u^*\!X_G\!)
 \ar[ul]_{\simeq}
 \ar[rr]^-{\underset{\phantom{bla}} { u_*(u^*\!x \otimes \rho^{-1}\! y\lambda^{-1}\!) }} &&
u_*(\!u^*Y_{\!G} \!\otimes\! u^*Y_{\!G}\!)
 \ar[ur]^\simeq
 \ar[d]_\simeq^{u_*(\rho\otimes \rho)} &&
u_*u^*Y_G 
 \ar[ur]_\leps
  \ar[d]_\simeq^{u_* \rho} & \\
& u_*X_H
 \ar[u]^{u_*\lambda}_\simeq
 \ar[rr]_-{u_*\delta} &&
u_*(X_H \otimes X_H)
 \ar[u]^\simeq_{u_*(\lambda\otimes\lambda)}
 \ar[rr]_-{\overset{\phantom{bla}}{u_*( (\rho u^*\!(x) \lambda) \otimes y)}} &&
u_*(Y_H\otimes Y_H)
 \ar[rr]_-{u_*\mu} &&
u_*Y_H & \\
& & & && & & &
}
\]
As before, the two pentagons commute (even before~$u_*$) because $\lambda=\lambda_u$ and $\rho=\rho_u$ are morphisms of (co)monoids. 
The two rhombuses commute by the naturality of $\reta$ and~$\leps$, respectively, and the middle square  (even before~$u_*$) by the functoriality of $\otimes = \otimes_H$.
Thus the whole diagram commutes, completing the proof.
\end{proof}

\begin{Rem} \label{Rem:grading}
If $\cat M$ takes values in $\mathbb Z$-graded categories and grading-preserv\-ing functors, we have graded versions of all the Green functors and modules in Theorems~\ref{Thm:Green-End-decat} and~\ref{Thm:Hom-decat-co-mon}, with values given instead by the graded Homs: $G\mapsto M(G)=\Hom^*_{\cat M(G)}(X_G , Y_G)$.
For Green 2-functors with values in tensor triangulated categories one typically uses the grading induced by the suspension functor.
\end{Rem}

\begin{Exa} \label{Exa:G-local}
In the local case $\GG= \JJ= \gpdG$ for a fixed~$G$, any pair of objects $X,Y\in \cat M(G)$ in the ``top'' category gives rise to a coherent family of pairs simply by restriction: $X_H = u^*X_G$ and $Y_H= u^*X_H$ for all $u\colon H\to G$, with $\lambda =  \rho = \id$. 
\end{Exa}

\begin{Exa} \label{Exa:units}
For any Green 2-functor~$\cat M$, the tensor unit objects $X_G = Y_G = \Unit$ form a coherent family of pairs thanks to the unit isomorphisms $u^*X_G\cong X_H$ of the strong monoidal functors~$u^*$.
The resulting Green functor $G\mapsto \End_{\cat M(G)}(\Unit)$ is commutative by the Eckmann--Hilton argument. 
In case the categories carry tensor-compatible gradings as in \cite{SuarezAlvarez04} (\eg when $\cat M$ takes values in tensor triangulated categories), we obtain a \emph{graded} commutative version $G \mapsto  \End_{\cat M(G)}^*(\Unit)$.
\end{Exa}

\begin{center}
$*\;\;*\;\;*$
\end{center}

We conclude by explaining how the classical Green functors used in algebra and topology arise by K- or Hom-decategorification, or sometimes both. 
All Hom-examples occur by specializing either \Cref{Exa:G-local} or \ref{Exa:units}, occasionally in graded form (\Cref{Rem:grading}).
The relevant pair $(\GG;\JJ)$ is always the same as for the used Green 2-functor.
Further variations are possible and are left to the interested reader. 

\begin{Exa}[Group cohomology]
\label{Exa:coh}
The cohomology ring Green functor $G\mapsto \mathrm H^*(G;\kk)$ arises by Hom-decategorification at the tensor units of the derived category Green 2-functor $G\mapsto \Der(\kk G)$ (\Cref{Exa:rep-theory}). For every $V \in \Der(\kk G)$, (hyper-) cohomology with twisted coefficients $H \mapsto \mathrm H^*(H;V|_H)$ yields a Mackey functor for~$G$ which is a module over the cohomology ring Green functor (for~$G$).
\end{Exa}

\begin{Exa}[Fixed points]
For a fixed group~$G$, suppose $A$ is a $G$-algebra, \ie a monoid in $\Mod(\kk G)$. Then $H\mapsto \mathrm H^0(H; A)= A^H$ yields a 
Green functor for $G$ whose value at $H\leq G$ is the subalgebra of $H$-fixed points in~$A$.
Indeed, this arises by Hom-decategorification as in \Cref{Thm:Hom-decat-co-mon} of the ($G$-local version of the) Green 2-functor $H\mapsto \Mod(\kk H)$ at the family of pairs induced (as in \Cref{Exa:G-local}) by the comonoid $X_G= \Unit = \kk$ and the monoid~$Y_G=A$.
Passing to the derived category, we get a Green functor for~$G$ whose value at~$H$ is the whole graded cohomology algebra $\mathrm H^*(H;A)$.
\end{Exa}

\begin{Exa}[Tate cohomology]
If  in \Cref{Exa:coh} we replace the derived category with the stable category Green 2-functor $G\mapsto \Stab(\kk G)$ (\Cref{Exa:quots}), we similarly obtain the Tate cohomology Green functor $G\mapsto \widehat{\mathrm H}^*(G;\kk)$ and Tate cohomology with any twisted coefficients as a ($G$-local) module over it. 
\end{Exa}

\begin{Exa}[The $\mathrm G_0$-ring]
\label{Exa:G_0}
If $\kk$ is a field (so that $-\otimes_\kk - $ is exact), consider the Green 2-functor $G\mapsto \mathrm{mod}(\kk G)$ of finite dimensional modules (\Cref{Exa:rep-theory}).
By K-decategorification via $\Kex$ (see \Cref{Rem:var-K-decat}), we get the Green 2-functor whose values are $\mathrm G_0(\kk G)$, the G-theory group of the group algebras.
\end{Exa}

\begin{Exa}[The trivial source ring]
\label{Exa:trivial-source}
The trivial source ring Green functor is the $\Kadd$-decategorification of the Green 2-functor $G \mapsto \mathrm{perm}_\kk(G)^\natural$ of trivial source modules (\Cref{Exa:rep-theory}). This is usually defined over a field $\kk$ of positive characteristic $p>0$, in which case such modules are known as $p$-permutation modules.
\end{Exa}

\begin{Exa}[The representation ring]
\label{Exa:Green-ring}
For a general~$\kk$, the representation ring (or Green ring) Green functor $G \mapsto \mathrm R_\kk(G)$
is the $\Kadd$-decat\-egorific\-ation of the Green 2-functor $G\mapsto \latt_\kk(G)$ of $\kk G$-lattices (\Cref{Exa:rep-theory}).
\end{Exa}

\begin{Exa} [Algebraic K-theory]
For a field~$\kk$ of characteristic zero, consider the Green 2-functor $G\mapsto \mathrm{proj}(\kk G)=\mathrm{mod}(\kk G)$ of finitely generated (necessarily projective) modules (\Cref{Exa:rep-theory}; note that $-\otimes_\kk-$ is exact on $\kk G$-projectives even if $\kk$ is not a field, but we also want a tensor unit!).
Via K-decategorification as in \Cref{Rem:exhigher}, we obtain from it a (graded) Green functor $G\mapsto \mathrm K_*^{\mathrm{alg}}(\kk G)$ whose values are the algebraic K-theory groups of the group algebra, with multiplication induced by the Hopf algebra structure of~$\kk G$. 
\end{Exa}

\begin{Exa}[The Burnside ring]
The Burnside ring Green functor $G\mapsto \mathrm B(G) = \mathrm K_0(G\sset, \sqcup, \times)$ (denoted $A(G)$ by topologists) arises in two ways. Firstly as the $\Kadd$-decategorification of the Green 2-functor of ``representable'' ordinary Mackey functors, \ie the additive category of spans of finite $G$-sets $G\mapsto \mathbb Z \pih \Span (G\sset)$ (see \Cref{Exa:Mack}), since its K-ring is the same as for $G$-sets. Secondly as the Hom-decategorification at the tensor units of the stable homotopy Green 2-functor $G\mapsto \SH(G)$ (\Cref{Exa:SH}). 
For any genuine $G$-spectrum $X\in \SH(G)$, the homotopy groups $H\mapsto \pi_*^H(X)= \SH(H)_*(\Unit, X|_H)$ yield a graded Mackey functor for~$G$, \ie a graded module over the Burnside ring (for~$G$). 
The latter can all be realized this way
by the existence of equivariant Eilenberg--Mac~Lane spectra.
\end{Exa}

\begin{Exa}[The complex representation ring]
Similarly, the complex representation ring Green functor $G\mapsto \mathrm R_\mathbb C(G)$ arises both as the K-decategorification of the Green 2-functor  $G\mapsto \fgmod(\mathbb C G)$ (as in Examples~\ref{Exa:G_0} or~~\ref{Exa:Green-ring} with $\kk = \mathbb C$), and as the Hom-de\-categor\-ific\-ation at the tensor units of the Kasparov theory Green 2-functor $G\mapsto \KK(G)$ (\Cref{Exa:KK}). For any separable $G$-C*-algebra $A\in \KK(G)$, the equivariant K-theory groups $H\mapsto  \mathrm K_*^H(A)= \KK(H)_*(\Unit, A|_H)$ yield a graded Mackey functor for~$G$ which is a module over the representation ring. This recovers results of \cite[\S4]{DellAmbrogio14}.
\end{Exa}

%------------------------------------------------------------------------------

\bibliographystyle{alpha}

%------------------------------------------------------------------------------
\end{document}